\documentclass[11pt]{amsart}
\usepackage{amsfonts,amssymb,amsthm}
\usepackage{amsmath,amscd}
\usepackage{pstricks}
\usepackage{pstricks,pst-node}
\usepackage{mathrsfs}
\usepackage[all]{xy}
\usepackage{color}
\usepackage{hyperref}

\DeclareMathAlphabet{\mathpzc}{OT1}{pzc}{m}{it}

\theoremstyle{plain}
\newtheorem{Thm}{Theorem}[section]
\newtheorem{Prop}[Thm]{Proposition}
\newtheorem{Not}[Thm]{Notation}
\newtheorem{Lem}[Thm]{Lemma}
\newtheorem{Coro}[Thm]{Corollary}
\theoremstyle{definition}

\newtheorem{Rem}[Thm]{Remark}

\newtheorem{example}{Example}
\numberwithin{equation}{section}

\newcommand{\sfP}{\mathsf{P}}
\newcommand{\sfH}{\mathsf{H}}
\newcommand{\sfS}{\mathsf{S}}
\newcommand{\sfZ}{\mathsf{Z}}
\newcommand{\sfF}{\mathsf{F}}
\newcommand{\sfQ}{\mathsf{Q}}
\newcommand{\sfT}{\mathsf{T}}

\newcommand{\cf}{\mathtt{cf}}

\newcommand{\Lcp}{\bar L}
\newcommand{\Mcp}{\bar M}

\newcommand{\Gk}{\h G_\field(n)}

\newcommand{\Ev}{{\mathsf{Ev}}}
\newcommand{\ev}{{\mathsf{ev}}}

\def\afS#1{\h{\mathcal S}_{\field}(#1)}

\def\afSk#1{\h{\mathcal S}_{\field}(n,#1)}
\newcommand{\afSrmbz}{\h{\mathcal S}_{\mathbb Z}(n,r)}
\newcommand{\afSrmbq}{\h{\mathcal S}_{\mathbb Q}(n,r)}
\newcommand{\afSrk}{\h{\mathcal S}_{\field}(n,r)}

\def\afxk#1{\widehat\xi_{#1,\field}}
\def\afsygr{{\h\fS_{r}}}
\newcommand{\affSr}{{\h\fS_{r}}}

\newcommand{\afVnq}{\h\sV_\mbq(n)}
\newcommand{\afunk}{\h\sU_\field(n)}

\newcommand{\afunrk}{\h\sU_\field(n,r)}

\newcommand{\afunz}{\h\sU_\mbz(n)}

\newcommand{\afunrz}{\h\sU_\mbz(n,r)}

\newcommand{\sAn}{{\h{\mathcal A}}_\field(n)}
\newcommand{\sAnr}{\h{\mathcal A}_\field(n,r)}
\newcommand{\An}{\h A_\field(n)}
\newcommand{\Anr}{\h A_\field(n,r)}

\newcommand{\afmsD}{{\mathscr D}^\vtg}
\newcommand{\afmbnn}{\mathbb N_\vtg^{n}}

\newcommand{\afThn}{\Theta_\vtg(n)}

\newcommand{\afThnpm}{\Theta_\vtg^\pm(n)}
\newcommand{\afThnp}{\Theta_\vtg^+(n)}
\newcommand{\afThnm}{\Theta_\vtg^-(n)}
\newcommand{\afThnr}{\Theta_\vtg(n,r)}
\newcommand{\afTh}{\Theta_\vtg}
\newcommand{\afMnq}{M_{\vtg,n}(\mathbb Q)}
\newcommand{\afMnk}{M_{\vtg,n}(\mpk)}
\newcommand{\afMnz}{M_{\vtg,n}(\mathbb Z)}

\newcommand{\afInr}{I_\vtg(n,r)}
\newcommand{\afLa}{\Lambda_\vtg}
\newcommand{\afLanr}{\Lambda_\vtg(n,r)}
\newcommand{\afbse}{\boldsymbol e^\vartriangle}
\newcommand{\afE}{E^\vartriangle}
\newcommand{\afjmath}{\jmath_\vtg}

\newcommand{\Pn}{\mathcal P_\field(n)}

\newcommand{\Qn}{\mathcal Q_\field(n)}
\newcommand{\Qnr}{\mathcal Q_\field(n,r)}
\newcommand{\tiQnr}{\ti{\mathcal Q}_\field(n,r)}
\newcommand{\QNr}{\mathcal Q_\field(N,r)}

\newcommand{\bse}{\boldsymbol{e}}

\newcommand{\bfi}{{\mathbf{i}}}
\newcommand{\bfj}{{\mathbf{j}}}
\newcommand{\bfk}{{\mathbf{k}}}
\newcommand{\bl}{{\mathbf{l}}}
\newcommand{\bfl}{{\mathbf{0}}}

\newcommand{\bfP}{{\mathbf{P}}}
\newcommand{\bfQ}{{\mathbf{Q}}}

\newcommand{\bfS}{{\mathbf{S}}}

\def\fS{{\frak S}}

\def\fA{{\frak A}}

\newcommand{\ms}{\mathscr}

\newcommand{\msD}{\mathscr D}

\def\hmod{{\text-}{\mathsf{mod}}}

\def\sfz{{\mathsf z}}

\def\sfF{{\mathsf F}}

\newcommand{\mpX}{\mathpzc X}

\newcommand{\mpk}{\mathpzc K}
\newcommand{\field}{\mathpzc K}

\def\sA{{\mathcal A}}

\def\sF{{\mathsf F}}
\def\sG{{\mathsf G}}

\def\sQ{{\mathcal Q}}

\def\sU{{\mathcal U}}
\def\sV{{\mathcal V}}

\def\sX{{\mathcal X}}

\newcommand{\mbnn}{\mathbb N^{n}}
\newcommand{\mbn}{\mathbb N}
\newcommand{\mbq}{\mathbb Q}
\newcommand{\mbc}{\mathbb C}
\newcommand{\mbz}{\mathbb Z}

\newcommand{\ttx}{\mathtt{x}}

\newcommand{\ttt}{\mathtt{t}}

\newcommand{\End}{\operatorname{End}}
\newcommand{\Hom}{\operatorname{Hom}}

\newcommand{\spann}{\operatorname{span}}
\newcommand{\diag}{\operatorname{diag}}

\newcommand{\ul}{\underline}

\newcommand{\la}{{\lambda}}
\newcommand{\La}{\Lambda}
\newcommand{\ga}{{\gamma}}
\newcommand{\Ga}{{\Gamma}}

\newcommand{\dt}{\delta}

\newcommand{\Og}{\Omega}
\newcommand{\Ogz}{\Omega_{\mathbb Z}}
\newcommand{\Ogk}{\Omega_{\mpk}}
\newcommand{\Ognk}{\Omega_{n,\mpk}}
\newcommand{\og}{\omega}

\newcommand{\vi}{\varphi}

\newcommand{\vep}{{\mathtt t}}

\newcommand{\al}{\alpha}
\newcommand{\bt}{\beta}
\newcommand{\sg}{\sigma}

\def\ggp#1#2{\left[\kern-3.2pt\left[{#1\atop #2}\right]\kern-3.2pt\right]}

\def\leq{\leqslant}\def\geq{\geqslant}
\def\le{\leqslant}

\newcommand{\bop}{\bigoplus}

\newcommand{\ot}{\otimes}

\newcommand{\han}{\subseteq}

\newcommand{\h}{\widehat}
\newcommand{\ti}{\widetilde}

\newcommand{\Lanr}{\Lambda(n,r)}

\newcommand{\lra}{\longrightarrow}
\newcommand{\ra}{\rightarrow}
\newcommand{\tra}{\twoheadrightarrow}

\newcommand{\xrk}{\xi_{r,\field}}

\newcommand{\afxr}{\widehat\xi_{r}}

\newcommand{\afzrk}{\widehat\zeta_{r,\field}}
\newcommand{\afxrk}{\widehat\xi_{r,\field}}
\newcommand{\afxrz}{\widehat\xi_{r}}
\newcommand{\afxi}{\widehat\xi}

\newcommand{\mnmod}{\!\!\!\mod\!}

\newcommand{\vtg}{{\!\vartriangle\!}}

\newcommand{\afuglq}{\sU_{\mathbb Q}(\widehat{\frak{gl}}_n)}
\newcommand{\afuslq}{\sU_{\mathbb Q}(\widehat{\frak{sl}}_n)}
\newcommand{\uglq}{\sU_{\mathbb Q}({\frak{gl}}_n)}
\newcommand{\afuglk}{\sU_\field(\widehat{\frak{gl}}_n)}

\newcommand{\afuglkN}{\sU_\field(\widehat{\frak{gl}}_N)}
\newcommand{\afuslk}{\sU_\field(\widehat{\frak{sl}}_n)}

\newcommand{\afuglkp}{\sU_\field^+(\widehat{\frak{gl}}_n)}
\newcommand{\afuglkm}{\sU_\field^-(\widehat{\frak{gl}}_n)}
\newcommand{\afuglkz}{\sU_\field^0(\widehat{\frak{gl}}_n)}

\newcommand{\afuglz}{\sU_\mbz(\widehat{\frak{gl}}_n)}

\newcommand{\tiafuglz}{\ti\sU_\mbz(\widehat{\frak{gl}}_n)}
\newcommand{\afuslz}{\sU_\mbz(\widehat{\frak{sl}}_n)}
\newcommand{\uglz}{\sU_\mbz({\frak{gl}}_n)}
\newcommand{\uglk}{\sU_\field({\frak{gl}}_n)}

\newcommand{\fSr}{\fS_r}

\newcommand{\afsl}{\widehat{\frak{sl}}_n}
\newcommand{\afgl}{\widehat{\frak{gl}}_n}
\newcommand{\afGL}{\widehat{GL}_n(\mpk)}

\newcommand{\Srk}{{\mathcal S}_\field(n,r)}

\begin{document}
\title{Affine Schur--Weyl Theory for Loop Semigroups and Polynomial Representations}

\author{Qiang Fu}
\address{School of Mathematical Sciences,
Key Laboratory of Intelligent Computing and Applications (Ministry of Education),
Tongji University, Shanghai, 200092, China.}
\email{q.fu@hotmail.com, q.fu@tongji.edu.cn}


\thanks{Supported by the National Natural Science Foundation
of China (12371032, 12431002)}

\begin{abstract}
Classical Schur--Weyl duality plays a fundamental role in the representation
theory of general linear and symmetric groups. In this paper, we develop an
affine Schur--Weyl theory for the Laurent polynomial loop semigroup.

Let $\widehat G_{\mathpzc K}(n)$ denote the Laurent polynomial loop semigroup
inside the loop group $GL_n(\mathpzc K((t)))$ over a field $\mathpzc K$, and let
$\widehat{\mathcal S}_{\mathpzc K}(n,r)$ be the affine Schur algebra.
The natural action of $\widehat G_{\mathpzc K}(n)$ on the affine
tensor space $\Omega_{\mathpzc K}^{\otimes r}$ induces an algebra
homomorphism
\(
\mathpzc K\widehat G_{\mathpzc K}(n)
\rightarrow
\widehat{\mathcal S}_{\mathpzc K}(n,r).
\)
We prove that this homomorphism is surjective for every field
$\mathpzc K$ satisfying $|\mathpzc K|>r$.
Furthermore, we establish the double centralizer property for affine Schur
algebras over
unique factorization domains for all \(n\geq2\). For \(n\geq2\) and any field
\(\mathpzc K\) with \(|\mathpzc K|>r\), this yields the corresponding centralizer property
for $\widehat G_{\mathpzc K}(n)$. In contrast,
for \(n=1\) and \(r\geq2\), the corresponding natural homomorphism is
not surjective over any field.
For $n\geq 2$, we use the double centralizer theorem to determine the center
of the affine Schur algebra and, when $\mathpzc K$ is a field with \(|\mathpzc K|>r\), show that $\widehat\zeta_{r,\mathpzc K}$ maps
the center of ${\mathpzc K}\widehat G_{\mathpzc K}(n)$ onto it.

 Over an algebraically closed field
$\mathpzc K$ of arbitrary characteristic, we prove that every finite-dimensional irreducible
$\widehat{\mathcal S}_{\mathpzc K}(n,r)$-module is a tensor product
of evaluation modules.
Moreover, we obtain explicit character formulas expressing these irreducible characters as products of characters of irreducible classical Schur algebra modules, showing that the finite-dimensional irreducible representation theory of affine
Schur algebras is controlled by that of classical Schur algebras.

Finally, we study polynomial representations of $\widehat G_{\mathpzc K}(n)$ over algebraically closed
fields $\mathpzc K$, and classify its finite-dimensional irreducible polynomial representations. For finite fields $\mathbb F_q$ with $q>r$, this yields a classification in the defining-characteristic setting.
\end{abstract}

 \sloppy \maketitle
\section{Introduction}

Schur algebras provide a fundamental connection between the
representation theory of general linear groups and symmetric groups.
In the classical setting, Schur algebras arise naturally from
polynomial representations of general linear groups. More precisely,
the diagonal action of $GL(V)$ on $V^{\otimes r}$ commutes with the
right permutation action of the symmetric group $\mathfrak S_r$, and
the resulting Schur--Weyl reciprocity gives rise to the classical
Schur algebras. The study of tensor representations goes back to the
classical works of Weyl \cite{Weyl}.
The modern theory of polynomial representations and Schur algebras, including modular and finite-field settings, has been developed extensively; see
\cite{Gr80,CL}.

The finite-field version of classical Schur--Weyl reciprocity has been studied by several authors.  Thrall \cite{Thrall}
initiated the study of this question by considering the surjectivity of
the natural homomorphism
\(
\mathpzc K GL_n(\mathcal{\mathpzc K})\rightarrow
\mathcal S_{\mathcal{\mathpzc K}}(n,r)
\). The precise field-size condition was later established by Benson and Doty,
who proved that the homomorphism is surjective for finite fields satisfying
$|\mathpzc K|>r$. This finite-field result provides an important model for affine generalizations of Schur--Weyl theory.
 Cruz \cite{Cruz} further studied the behavior of this
homomorphism over finite fields and commutative rings.
The quantum analogue of Schur--Weyl reciprocity was developed by
Jimbo \cite{Jimbo}. Its cyclotomic generalizations were further studied
in \cite{ATY,SS,Hu}.

The affine analogue of Schur--Weyl theory has been investigated from
several different viewpoints, including affine Hecke algebras,
affine quantum groups and affine $q$-Schur algebras. Affine
$q$-Schur algebras were introduced and studied in
\cite{Gr99}. Their connections with affine quantum
groups and geometric constructions have been developed in
\cite{GV,Lu99,DDF}. The categorical structures of affine $q$-Schur
algebras have also attracted considerable attention; see
\cite{MT1,MT2}.

Most existing constructions of affine Schur algebras are formulated
through affine quantum groups or enveloping algebras of loop
algebras. These approaches play a central role in the study of
quantum structures, canonical bases and categorifications.
From the original Schur-theoretic viewpoint, however, Schur algebras are closely related to polynomial representations of general linear groups, and the corresponding affine group-theoretic realization remains less developed.
A semigroup realization of affine Schur algebras was investigated by Yang \cite{Yang1,Yang2},
where affine Schur algebras were realized using semigroups of infinite periodic matrices over infinite fields. The semigroup considered in those works is different from the Laurent polynomial loop semigroup arising naturally from the loop group framework.

In this paper, we develop an affine Schur--Weyl theory for Laurent polynomial loop semigroups and establish a semigroup realization of affine Schur algebras over fields of arbitrary characteristic, including finite fields satisfying $|\mathpzc K|>r$.
 Let
$\widehat G_{\mathpzc K}(n)$ denote the affine loop semigroup
consisting of matrices in
$GL_n(\mathpzc K((\vep)))$ whose entries are Laurent polynomials in
$\vep$, where $\mathpzc K$ is a commutative ring with unity. Let
$\widehat{\mathcal S}_{\mathpzc K}(n,r)$ be the affine Schur algebra.
The natural action of $\widehat G_{\mathpzc K}(n)$ on the affine
tensor space $\Omega_{\mathpzc K}^{\otimes r}$ induces an algebra
homomorphism
\(
\widehat\zeta_{r,\mathpzc K}:
\mathpzc K\widehat G_{\mathpzc K}(n)
\longrightarrow
\widehat{\mathcal S}_{\mathpzc K}(n,r).
\)
The first main result of this paper establishes an affine analogue of finite-field Schur--Weyl reciprocity:  whenever
$\mathpzc K$ is a field with $|\mathpzc K|>r$, the above homomorphism
is surjective (Theorem~\ref{surjective2}). Consequently, the affine
Schur algebra can be realized as a quotient of a loop semigroup
algebra. This provides a semigroup-theoretic interpretation of affine
Schur algebras and gives an affine analogue of the classical Schur-theoretic construction.
The proof of this result requires genuinely affine arguments arising from the periodic structure of the extended affine symmetric group and the Laurent polynomial structure of the loop semigroup.  Moreover,
we show that the subgroup generated by \(\widehat G_{\mathpzc K}(n)\) is the rational loop group \(GL_n({\mathpzc K}(t))\) (Proposition \ref{rational loop group}), whereas
the natural action
does not extend to the rational or formal loop groups  $GL_n(\mathpzc K(t))$ and
$GL_n(\mathpzc K((t)))$, respectively (Proposition~\ref{not extension}).
Therefore, the loop semigroup is not merely an alternative presentation of a loop group approach, but an essential object underlying affine Schur--Weyl theory.

We establish the following  double  centralizer property for affine Schur
algebras. Let
\(\mathpzc K\) be a commutative unique factorization domain,
\(n\geq2\), and \(r\geq1\). Then the natural right action of
\(\widehat{\frak S}_r\) on \(\Omega_{\mathpzc K}^{\otimes r}\)
induces an algebra isomorphism
$
\rho_{r,\mathpzc K}:
\mathpzc K\widehat{\frak S}_r^{\,\mathrm{op}}
\xrightarrow{\sim}
\operatorname{End}_{\widehat{\mathcal S}_{\mathpzc K}(n,r)}
(\Omega_{\mathpzc K}^{\otimes r}),
$
where \(\widehat{\frak S}_r\) is the extended affine symmetric group
(Theorem~\ref{thm:second-centralizer-nge2}).
For \(n\geq r\), a simpler argument gives the same
property over arbitrary commutative rings (Theorem~\ref{thm:second-centralizer-stable}).
Together with Theorem~\ref{surjective2}, this yields the corresponding
centralizer property for $\widehat G_{\mathpzc K}(n)$ whenever $\mpk$ is a field with
\(|\mathpzc K|>r\) and \(n\geq2\) (Theorem \ref{cor:second-centralizer-loop-semigroup-nge}).
When
\(n=1\) and \(r\geq2\), the natural homomorphism \(
\rho_{r,\mathpzc K}:\mathpzc K\widehat{\mathfrak S}_r^{\,\mathrm{op}}
\rightarrow
\operatorname{End}_{\widehat S_{\mathpzc K}(1,r)}
(\Omega_{\mathpzc K}^{\otimes r})
\) is not
surjective over any field (Example 1 and Theorem \ref{thm:rank-one-counterexample}).
In the geometric setting of affine Hecke algebras and affine $q$-Schur
algebras, Pouchin~\cite[Theorem~8.1]{Pouchin} stated an analogous double
centralizer property in the case corresponding to $n<r$ in our notation.
A gap in the proof of this result was already pointed out in \cite{Fu13}.
 Our proof is independent of this geometric
construction and is based on the skew group algebra realization of
$\mathpzc K\widehat{\mathfrak S}_r$.
For $n\geq2$, as an application of the double centralizer property, we determine
the center of the affine Schur algebra and, when $\mathpzc K$ is a
field with $|\mathpzc K|>r$, show that the homomorphism
$\widehat\zeta_{r,\mathpzc K}$ maps the center of
$K\widehat G_{\mpk}(n)$ onto it.

We next study the representation theory of affine
Schur algebras. Over an algebraically closed field $\mathpzc K$ of
arbitrary characteristic, we prove that every finite-dimensional
irreducible
$\widehat{\mathcal S}_{\mathpzc K}(n,r)$-module is a tensor product
of evaluation modules. Moreover, the characters of these irreducible modules are
expressed explicitly as products of characters of irreducible
classical Schur algebra modules
(Theorems~\ref{classification}--\ref{character}). This yields a
characteristic-independent description of
irreducible modules of affine Schur algebras, showing that their finite-dimensional irreducible representation theory can be described in terms of the corresponding classical Schur algebra modules.
We further investigate the relationship between affine Schur
algebras and
their Lusztig subalgebras.
 We classify the finite-dimensional irreducible modules of the Lusztig subalgebras and show that they correspond naturally to those of affine Schur algebras
(Theorem~\ref{simple U(n,r)-modules}).

We also develop the polynomial representation theory of the Laurent polynomial loop semigroup.
In the affine setting, the relevant coordinate spaces are infinite-dimensional, which leads
naturally to a completed coefficient space. Using this framework, we prove that finite-dimensional
homogeneous polynomial representations of \(\widehat G_{\mathpzc K}(n)\) of degree \(r\) are
equivalent to finite-dimensional \(\widehat{\mathcal S}_{\mathpzc K}(n,r)\)-modules
(Theorem \ref{equivalence1}). Thus the affine Schur algebra governs the polynomial representation theory of
the natural loop semigroup in a manner analogous to the role of the classical Schur algebra
for \(GL_n\). Together with Theorem \ref{classification}, this gives a classification of the finite-dimensional
irreducible polynomial representations of \(\widehat G_{\mathpzc K}(n)\).

A further motivation for the loop semigroup approach is related to defining-characteristic representation theory of reductive groups over local fields.
Vign\'eras \cite{Vigneras} studied smooth representations of reductive
$p$-adic groups over coefficient fields
whose characteristic is different from the residue characteristic. She also
pointed out that the case where the characteristic of the coefficient
ring equals the residue characteristic requires further understanding.
Related to this direction, the present paper studies polynomial representations arising from the Laurent polynomial loop semigroup, providing a defining-characteristic framework for affine loop semigroups and their associated affine Schur algebras.
For a finite field $\mathbb F_q$ with $q>r$, we study
polynomial representations of
$\widehat G_{\mathbb F_q}(n)$ over algebraically closed coefficient
fields of characteristic $p$ (Theorem~\ref{classification of polynomial representations}).
This connects affine Schur algebras with finite-field Schur--Weyl theory in the defining-characteristic setting.

The paper is organized as follows. Section 2 recalls  affine Schur algebras and the  BLM realization of the enveloping algebra of $\widehat{\mathfrak{gl}}_n$.  Section 3 establishes the integral form of the enveloping algebra of $\widehat{\mathfrak{gl}}_n$ and proves the Frenkel--Mukhin conjecture in the classical setting.
Section~4
establishes the affine Schur--Weyl reciprocity for loop semigroups. Section~5
establishes the double centralizer property, gives a rank-one
counterexample,  and determines the center of the affine
Schur algebra.
 Sections 6--8 are devoted to evaluation maps and the
representation theory of affine Schur algebras, including tensor
product decompositions and character formulas. Section 9 studies
Lusztig algebras.   Section 10 studies polynomial representations of affine
loop semigroups and establishes their equivalence with affine Schur algebra modules, including the finite-field case in the defining-characteristic setting.

\begin{Not}\label{Notaion} \rm
Let $\field$ be a commutative ring with unity and $p=\text{char}\field\geq 0$.
For a positive integer $n$, let
$\afMnk$ be the set of all matrices
$A=(a_{i,j})_{i,j\in\mbz}$ with $a_{i,j}\in\field$ such that
\begin{itemize}
\item[(a)]$a_{i,j}=a_{i+n,j+n}$ for $i,j\in\mbz$; \item[(b)] for
every $i\in\mbz$, both sets $\{j\in\mbz\mid a_{i,j}\not=0\}$ and
$\{j\in\mbz\mid a_{j,i}\not=0\}$ are finite.
\end{itemize}
For $i,j\in\mbz$, let $\afE_{i,j}\in\afMnk$ be the matrix
$(e^{i,j}_{k,l})_{k,l\in\mbz}$ defined by
\begin{equation*}e_{k,l}^{i,j}=
\begin{cases}1,&\text{if $k=i+sn,l=j+sn$ for some $s\in\mbz$;}\\
0,&\text{otherwise}.\end{cases}
\end{equation*}
\end{Not}
\section{Affine Schur algebras and the BLM realization}
This section presents the BLM realization of the universal enveloping algebra of the affine Lie algebra $\widehat{\mathfrak{gl}}_n$.

\subsection{Affine Schur algebras via extended affine symmetric groups}
Let $\afsygr$ be the extended affine symmetric group, which consists of all permutations
$w:\mbz\ra\mbz$ such that $w(i+r)=w(i)+r$ for $i\in\mbz$.
For $1\leq i\leq r$  define the simple transposition $s_i \in \afsygr$ by
$$s_i(j)=\begin{cases}j, &\text{ for $j\not\equiv i,i+1\mnmod r$};\\
 j-1,&\text{ for
$j\equiv i+1\mnmod r$};\\
j+1,&\text{ for $j\equiv i\mnmod r$.}
\end{cases}$$
The subgroup of $\afsygr$ generated by $s_1,\ldots,s_{r-1}$
is isomorphic to the symmetric group $\fS_r$ on $r$ letters.
Let $\fA$ be the subgroup of $\afsygr$ consisting of permutations $y$ of
$\mbz$ satisfying $y(i)\equiv i(\text{mod}\,r)$.
This subgroup is isomorphic to
$\mbz^r$ via the map $y\mapsto
\la=(\la_1,\la_2,\ldots,\la_r)\in\mbz^r$, where $y(i)=\la_ir+i$, for
all $1\leq i\leq r$. We will henceforth identify $\fA$ with $\mbz^r$. For $1\leq i\leq r$ let $\bse_i=(0,\ldots,0,\underset {(i)}1,0,\ldots,0)\in\mbz^r$  denote the $i$-th standard basis vector.
The extended affine symmetric group $\afsygr$ is generated by the elements $s_i$ and $\bse_j$ for $1\leq i\leq r-1$ and $1\leq j\leq r$.

We now recall the definition of affine Schur algebras, which are central objects in the representation theory of $\widehat{\mathfrak{gl}}_n$.
Let $\Ogk=\field[\vep,\vep^{-1}]^n,$ where  $\field$ is a commutative ring with unity
and $\vep$ is an indeterminate. Then $\Ogk$ is a free $\field[\vep,\vep^{-1}]$-module with standard basis $\og_1,\og_2,\cdots,\og_n$. For $1\leq i\leq n$ and $l\in\mbz$ let
$$\og_{i+ln}=\vep^{-l}\og_i.$$
Then the set $\{\og_s\mid s\in\mbz\}$ forms a $\field$-basis for $\Ogk$.
For $\bfi=(i_1,\ldots,i_r)\in\mbz^r$, write
$$\og_\bfi=\og_{i_1}\ot\og_{i_2}\ot\cdots\ot \og_{i_r}.$$
The tensor space
$\Ogk^{\ot r}$ admits  a natural right $\affSr$-module structure defined by
\begin{equation}\label{action of affine symmetric group}
\begin{cases}
\og_\bfi s_j=\og_{i_1}\ot\cdots\ot\og_{i_{j+1}}\ot \og_{i_{j}}\ot\cdots\ot \og_{i_{r}}\\
\og_\bfi \bse_k=\og_{i_1}\ot\cdots\ot\og_{i_{k-1}}\ot \og_{i_{k}+n}\ot \og_{i_{k+1}}\ot\cdots\ot \og_{i_{r}}
\end{cases}
\end{equation}
where $\bfi\in\mbz^r$, $1\leq j\leq r-1$ and $1\leq k\leq r$.
The affine Schur algebras over $\field$ are defined as the endomorphism algebras
$$\afSrk:=\End_{\affSr}(\Ogk^{\ot r})$$
(cf. Ginzburg--Vasserot \cite{GV}, Lusztig \cite{Lu99}, Green \cite{Gr99}).
Representations of affine Schur algebras have been previously studied by Yang \cite{Yang1,Yang2}.

To further understand the structure of affine Schur algebras, consider the set
$$\afmbnn=\{(\la_i)_{i\in\mbz}\mid
\la_i\in\mbn,\,\la_i=\la_{i-n}\ \text{for}\ i\in\mbz\}.$$
For a fixed $r\in\mbn$ let
$$\afLa(n,r)=\big\{\la\in\afmbnn\mid\sum_{1\leq i\leq n}\la_i=r\big\}.$$
For $\la\in\afLanr$ let  $\fS_\la:=\fS_{(\la_1,\ldots,\la_n)}$
denote the corresponding standard Young subgroup of
the symmetric group $\fS_r$.
Then we have the isomorphism
$$\afSrk\cong\End_{\affSr}\biggl
(\bop_{\la\in\afLa(n,r)}\field\ul{\fS_\la}\affSr\biggr),$$
where
$\ul{\fS_\la}=\sum_{w\in\fS_\la}w$.
We will henceforth identify $\afSrk$ with $\End_{\affSr}\biggl
(\bop_{\la\in\afLa(n,r)}\field\ul{\fS_\la}\affSr\biggr)$.
Moreover, there is a base change property:
\begin{equation}\label{afSrk}
\afSrk\cong\afSrmbz\ot_\mbz\field,
\end{equation}
where $\afSrmbz$ denotes the affine Schur algebra over $\mathbb{Z}$.

\subsection{The Affine Schur homomorphism}
The connection between the universal enveloping algebra $\afuglq$ and affine Schur algebras is established through an algebra homomorphism, which we now describe in detail.

Let $\afuglq$ be the enveloping algebra of the loop
algebra $\afgl$, where  $\afgl:={\frak{gl}_n}(\mbq)\ot\mbq[\ttt,\ttt^{-1}]$.
Clearly, the map
$$\afMnq\lra\afgl ,\,\,\,\afE_{i,j+ln}\longmapsto E_{i,j}\ot \ttt^l, \,1\le i,j\le n,l\in\mbz $$
defines a Lie algebra isomorphism, where $E_{i,j}\in{\frak{gl}_n}(\mbq)$  denotes the elementary matrix $(\dt_{i,k}\dt_{j,l})_{1\leq k,l\leq n}$. Henceforth, we identify
$\afgl$ with $\afMnq$ under this isomorphism.

The vector space $\Og_\mbq$ carries a natural left $\afuglq$-module, with the action defined by
\begin{equation*}\label{natural module}
\afE_{i,j}\og_{k}=
\begin{cases}
\og_{i+tn}, &\text{if $k=j+tn$,}\\
0,&\text{otherwise.}
\end{cases}
\end{equation*}
This module structure extends to the tensor space $\Og_\mbq^{\ot r}$ via the Hopf algebra structure of $\afuglq$.
Crucially, the left action of $\afuglq$
commutes with the right action of the extended affine symmetric group
$\affSr$ on $\Og_\mbq^{\ot r}$, endowing $\Og_\mbq^{\ot r}$ with the structure of a $\afuglq$-$\affSr$ bimodule.  This bimodule structure  induces a $\mbq$-algebra homomorphism
\begin{equation}\label{afxr}
\afxr:\afuglq\ra\afSrmbq.
\end{equation}

\subsection{Integral bases of affine Schur algebras}
To understand the image of $\afxr$ and the structure of affine Schur algebras, we first need to describe explicit bases for $\afSrmbz$.
Let
$\afThn=\{A\in\afMnz\mid a_{i,j}\in\mbn,\,\forall i,j\}$ and let $$\afThnr=\{A\in\afThn\mid\sg(A)=r\},$$ where $\sg(A)=\sum_{1\leq i\leq n,\,j\in\mbz}a_{i,j}$.
For $\la,\mu\in\afLanr$, we define the sets
$$\afmsD_\la=\{d\mid d\in\affSr,\ell(wd)=\ell(w)+\ell(d)\text{ for
$w\in\fS_\la$}\}$$ and
$\afmsD_{\la,\mu}=\afmsD_{\la}\cap{\afmsD_{\mu}}^{-1}$, where $\ell$
denotes the length function on $\affSr$ and ${\afmsD_{\mu}}^{-1}=\{d^{-1}\mid d\in\afmsD_{\mu}\}$.

For $1\leq i\leq n$ and $k\in\mbz$ let
\begin{equation*}\label{set R}
R_{i+kn}^{\la}=\{\la_{k,i-1}+1,\la_{k,i-1}+2,\ldots,\la_{k,i-1}+\la_i=\la_{k,i}\},
\end{equation*}
where $\la_{k,i-1}=kr+\sum_{1\leq t\leq i-1}\la_t$.
There is
a bijective map
\begin{equation*}
{\afjmath}:\{(\la, d,\mu)\mid
d\in\afmsD_{\la,\mu},\la,\mu\in\afLanr\}\lra\afThnr
 \end{equation*}
sending $(\la, w,\mu)$ to $A=(a_{k,l})$, where $a_{k,l}=|R_k^\la\cap wR_l^\mu|$ for all $k,l\in\mbz$ (see Varagnolo--Vasserot \cite[\S7.4]{VV99}). This combinatorial parametrization is essential for constructing bases of affine Schur algebras.

For $A=\afjmath(\la,d,\mu)\in\afThnr$ with $\la,\mu\in\afLanr$, $d\in\afmsD_{\la,\mu}$, we define the element
$e_A\in\afSrmbz$ by
\begin{equation*}\label{def of standard basis}
e_A(\ul{\fS_\nu} h)=\dt_{\mu\nu}\sum_{w\in\fS_\la
d\fS_\mu}w h\
\end{equation*}
where $\nu\in\afLanr$ and $h\in\mbz\affSr$. By \cite[Th.~2.2.4]{Gr99} the set
$\{e_A\mid A\in\afThnr\}$ forms a $\mbz$-basis for $\afSrmbz$.

\subsection{The BLM realization of  the affine enveloping algebra}
We now review the BLM realization of the universal enveloping algebra $\afuglq$.
Let
$$\afThnpm=\{A\in\afThn\mid a_{i,i}
=0\text{ for all $i$}\}.$$
For $A\in\afThnpm$ and ${\bf j}\in \afmbnn$, define the elements
$$A[{\bf j},r]= \sum_{\mu\in\afLa(n,r-\sg(A))}
\mu^\bfj e_{A+\diag(\mu)}\in\afSrmbq,
$$
where $\mu^\bfj=\prod_{1\leq i\leq n}\mu_i^{j_i}$.
For $i\in\mbz$, let $\bar{i}$ denote the residue class modulo $n$, and define
$\afbse_i\in\afmbnn$ by
\begin{equation*}
(\afbse_i)_j=
\begin{cases}
1&\text{if $\bar j=\bar i $}\\
0&\text{otherwise}.
\end{cases}
\end{equation*}
By \cite[Th.~6.1.5]{DDF}, the homomorphism $\afxr$ has an explicit description on generators:
\begin{equation*}
\afxr(\afE_{i,j})=\afE_{i,j}[\bfl,r],\quad \afxr(H_i)=0[\afbse_i,r]
\end{equation*}
for $i\not=j$, where $H_i=\afE_{i,i}$. These formulas show how the Chevalley generators of the affine Lie algebra are represented in the affine Schur algebra.
The maps $\afxr$ induce an algebra homomorphism
\begin{equation*}\label{zeta}
\afxi=\prod_{r\geq 0}\afxr:\afuglq\ra\prod_{r\geq 0}\afSrmbq.
\end{equation*}

For $A\in\afThnpm$ and $\bfj\in\afmbnn$, we define
$$A[\bfj]=(A[\bfj,r])_{r\geq 0}\in\prod_{r\geq 0}\afSrmbq.$$
Let $\afVnq$ be the $\mbq$-subspace of $\prod_{r\geq 0}\afSrmbq$ spanned by
$A[\bfj]$ for $A\in\afThnpm$, $\bfj\in\afmbnn$.

\begin{Thm}[{\cite[Th.~6.3.4]{DDF}}]\label{realization}
The $\mbq$-vector space $\afVnq$ is a subalgebra of $\prod_{r\geq 0}\afSrmbq$ with $\mbq$-basis $\{A[\bfj]\mid A\in\afThnpm,\,\bfj\in\afmbnn\}$. Furthermore, the map $\afxi$ induces a $\mbq$-algebra isomorphism $\afxi:\afuglq\ra\afVnq$.
\end{Thm}
 Henceforth,
we identify $\afuglq$ with $\afVnq$ via $\afxi$,
and consequently identify generators such as $\afE_{i,j}$ with their images  $\afE_{i,j}[\bfl]$, etc.
\section{Integral forms and the classical Frenkel--Mukhin conjecture}
The investigation of integral forms has played a significant role in the representation theory of Lie algebras.
In \cite[\S7.2]{FM}, Frenkel--Mukhin introduced  an integral form $U_q^{\text{res}}(\afgl)$ of quantum affine $\frak{gl}_n$, and conjectured that
this form is a Hopf subalgebra. We show in Theorem \ref{generators2} that this conjecture admits a positive resolution in the classical setting ($q=1$), establishing $\afuglz$ as a Hopf subalgebra of $\afuglq$.

\subsection{The integral form $\afuglz$ }
Following \cite{FM}, we define an integral form $\afuglz$ of $\afuglq$  through explicit generators.
For $1\leq i\leq n$ and $l\in\mbz$ with $l\not=0$ let
\begin{equation}\label{gil}
g_{i,l}=\afE_{i,i+ln}\in\afuglq.
\end{equation}
These elements satisfy the commutation relation
$g_{i,l}g_{j,m}=g_{j,m}g_{i,l}.$
For $1\leq i\leq n$ and $k\geq 0$ we define the elements $\La_{i,\pm k}$ in $\afuglq$ via
generating functions in an indeterminate $u$
\begin{equation}\label{Laik generating function}
\La_i^\pm(u):=\sum_{k\geq 0}\La_{i,\pm k}u^k=\exp\big(-\sum_{k\geq 1}\frac{g_{i,\pm k}}{k}u^k\big).
\end{equation}
The integral form $\afuglz$ is defined as the $\mbz$-subalgebra of $\afuglq$ generated by the elements $(\afE_{i,j})^{(k)}$, $\La_{i,\pm k}$ and $\big({H_i\atop t}\big)$ for $1\leq i\leq n$, $j\in\mbz$, $\bar i\not=\bar j$, $k,t\in\mbn$,
where $(\afE_{i,j})^{(k)}=\frac{(\afE_{i,j})^k}{k!}$, $H_i=\afE_{i,i}$ and
$\big({H_i\atop t}\big)=\frac{H_i(H_i-1)\cdots (H_{i}-t+1)}{t!}.$
\subsection{An alternative integral form $\tiafuglz$}
An alternative approach to integral forms arises from the BLM realization. Define
$$\tiafuglz=\spann_\mbz\big\{A[\bfl]\bigg({H\atop\la}\bigg)B[\bfl]\mid A\in\afThnp,\,B\in\afThnm,\,\la\in\afmbnn\big\}\han\afuglq,$$
where $\big({H\atop\la}\big)=\prod_{1\leq i\leq n}\big({H_i\atop\la_i}\big)$ and
$$\afThnp=\{A\in\afThn\mid a_{i,j}=0\text{ for }i\geq j\}\text{ and }\afThnm=\{A\in\afThn\mid a_{i,j}=0\text{ for }i\leq j\}.$$
By \cite{Fu13},
the $\mbz$-module $\tiafuglz$ constitutes a $\mbz$-subalgebra of $\afuglq$.
A fundamental result we establish is the equality  $\afuglz=\tiafuglz$.
The following result from \cite[Th. 4.9 and Lem. 4.10]{Fu19} characterizes the generators of $\tiafuglz$.

\begin{Prop}\label{generators}
$(1)$ The algebra $\tiafuglz$ is generated by the elements $(k\afE_{i,j})[\bfl]$ and $\big({H_i\atop t}\big)$ for $1\leq i\leq n$, $j\in\mbz$, $i\not=j$, $k,t\in\mbn$.

$(2)$ Let $k\in\mbn$ and $i,j\in\mbz$ with $\bar i\not=\bar j$. Then the following equation holds in $\afuglq$
\begin{equation*}
(\afE_{i,j})^{(k)}
=(k\afE_{i,j})\{\bfl\}.
\end{equation*}
\end{Prop}

\subsection{Proof of the classical Frenkel-Mukhin conjecture}
To establish the Frenkel--Mukhin conjecture in the classical case, we require several technical results. We begin with multiplication formulas from \cite{DDF}.

\begin{Thm}\label{Multiplication Formulas at v=1}
Assume $h\in\mbz$, $m\in \mbz\backslash\{0\}$, $\bfj\in\afmbnn$, and $A\in\afThnpm$. The
following multiplication formulas hold in $\afuglq$
\begin{equation*}
\begin{split}
\afE_{h,h+mn}[\bfl]A[\bfj] &=\sum_{s\not\in\{h,h-mn\}\atop a_{h,s}\geq 1}(a_{h,s+mn}+1)(A+\afE_{h,s+mn}-\afE_{h,s})[\bfj]\\
&\quad+\sum_{0\leq t\leq j_h\atop a_{h,h-mn}\geq 1}(-1)^t\left({j_h\atop t}
\right)(A-\afE_{h,h-mn})[\bfj+(1-t)\afbse_h]\\
&\quad+(a_{h,h+mn}+1)\sum_{0\leq t\leq j_h}\left({j_h\atop
t}\right)(A+\afE_{h,h+mn})[\bfj-t\afbse_h].\\
\end{split}
\end{equation*}
\end{Thm}

For $1\leq i\leq n$, $l>0$ and $k\geq 0$ we define refined elements $\La_{i,\pm k,l}$ in $\afuglq$ through
generating functions
\begin{equation}\label{Laikl generating function}
\sum_{k\geq 0}\La_{i,\pm k,l}u^k=\exp\big(-\sum_{k\geq 1}\frac{g_{i,\pm kl}}{k}u^k\big).
\end{equation}
Note that $\La_{i,\pm k}=\La_{i,\pm k,1}$  by definition.
These elements satisfy the following property.
\begin{Lem}\label{Laikl}
We have $\La_{i,0,l}=1$ and
\begin{equation*}
k\La_{i,\pm k,l}=-\sum_{0\leq s\leq k-1}g_{i,\pm(k-s)l}\La_{i,\pm s,l}
\end{equation*}
for $1\leq i\leq n$, $k\geq 1$ and $l>0$.
\end{Lem}

The following identity is a key combinatorial identity that relates these $\Lambda$-elements to the BLM basis.

\begin{Prop}\label{Laikl formula}
For $1\leq i\leq n$, $k\geq 0$ and $l>0$ we have
$\La_{i,\pm k,l}=(-1)^k(k\afE_{i,i\pm ln})[\bfl]$.
\end{Prop}
\begin{proof}
We argue by induction on $k$. The case where $k=0,1$ is trivial.
For $k\geq 2$, the induction hypothesis and Lemma \ref{Laikl} yield
\begin{equation*}
k\La_{i,\pm k,l}=-\sum_{0\leq s\leq k-1}g_{i,\pm(k-s)l}(-1)^s(s\afE_{i,i\pm ln})[\bfl].
\end{equation*}
Applying  Theorem \ref{Multiplication Formulas at v=1}, we have for $1\leq s\leq k-1$
\begin{equation*}
\begin{split}
g_{i,\pm(k-s)l}(s\afE_{i,i\pm ln})[\bfl]
&=((s-1)\afE_{i,i\pm ln}+\afE_{i,i\pm(k+1-s)ln})[\bfl]\\
&\qquad+(\dt_{s,k-1}(k-1)+1)
(s\afE_{i,i\pm ln}+\afE_{i,i\pm(k-s)ln})[\bfl].
\end{split}
\end{equation*}
It follows that
\begin{equation*}
\begin{split}
k\La_{i,\pm k,l}&=-g_{i,\pm kl}-\sum_{1\leq s\leq k-1}(-1)^s g_{i,\pm(k-s)l}(s\afE_{i,i\pm ln})[\bfl]\\
&=-g_{i,\pm kl}-\sum_{1\leq s\leq k-1}(-1)^s
((s-1)\afE_{i,i\pm ln}+\afE_{i,i\pm(k+1-s)ln})[\bfl]\\
&\qquad-\sum_{1\leq s\leq k-2}(-1)^s(s\afE_{i,i\pm ln}+\afE_{i,i\pm(k-s)ln})[\bfl]
-(-1)^{k-1}k(k\afE_{i,i\pm ln})[\bfl].
\end{split}
\end{equation*}
Now observe that
\begin{equation*}
\begin{split}
&\quad -g_{i,\pm kl}-\sum_{1\leq s\leq k-2}(-1)^s(s\afE_{i,i\pm ln}+\afE_{i,i\pm(k-s)ln})[\bfl]\\
&=-\sum_{0\leq s\leq k-2}(-1)^s(s\afE_{i,i\pm ln}+\afE_{i,i\pm(k-s)ln})[\bfl]\\
&=\sum_{1\leq s\leq k-1}(-1)^s
((s-1)\afE_{i,i\pm ln}+\afE_{i,i\pm(k+1-s)ln})[\bfl].
\end{split}
\end{equation*}
Substituting this into the previous expression for $k\La_{i,\pm k,l}$, we see that the two summations cancel exactly, leaving only the final term
$k\La_{i,\pm k,l}=(-1)^kk(k\afE_{i,i\pm ln})[\bfl]$. Dividing both sides by $k$ yields the desired identity.
\end{proof}

We now prove the classical case of the Frenkel-Mukhin conjecture in \cite[\S 7.2]{FM}.

\begin{Thm}\label{generators2}
$(1)$ We have $\afuglz=\tiafuglz$ and $\afuglz$ is a $\mbz$-Hopf subalgebra of $\afuglq$.

$(2)$ For $r\in\mbn$, we have $\afxr(\afuglz)=\afSrmbz$.
\end{Thm}
\begin{proof}
By Proposition \ref{generators} and \ref{Laikl formula},  the algebra $\tiafuglz$ is generated by the elements $(\afE_{i,j})^{(k)}$, $\La_{i,\pm k,l}$ and $\big({H_i\atop t}\big)$ for $1\leq i\leq n$, $j\in\mbz$, $\bar i\not=\bar j$, $l>0$, $k,t\in\mbn$. Applying  \cite[Lem. 1.6]{JM07} we conclude that $\afuglz=\tiafuglz$.
Hence by \cite[Cor. 9.3]{Fu18} and \cite[Th. 6.7]{Fu13} we see that  $\afuglz$ is a $\mbz$-Hopf subalgebra of $\afuglq$ and $\afxr(\afuglz)=\afSrmbz$.
\end{proof}

This result provides a complete proof of the Frenkel-Mukhin conjecture in the classical setting, verifying that their proposed integral form indeed possesses the desired Hopf algebra structure.

\subsection{Base change to general coefficient rings}
The integral form $\afuglz$ enables the extension of our construction to arbitrary commutative rings.
Let $\field$ be a commutative ring with unity, and define
\begin{equation}\label{Unk}
\afuglk=\afuglz\ot_\mbz\field.
\end{equation}
We refer to $\afuglk$ as the hyper loop algebra of $\afgl$ over $\field$.
We shall denote the images of $(\afE_{i,j})^{(k)}$, $\La_{i,\pm k,l}$, $\big({H_i\atop k}\big)$, $A[\bfj]$, etc. in $\afuglk$ by the same letters.

Recall from \eqref{afSrk} that the affine Schur algebra over $\field$ satisfies $\afSrk\cong\afSrmbz\otimes\field$. We will identify $\afSrk$ with $\afSrmbz\ot \field$. Again,
by abuse of notation, we shall denote the images of $e_A,A[\la,r]$, etc. in $\afSrk$ by the same letters.

By Theorem \ref{generators2} (2), the restriction of $\afxr$ to $\afuglz$ gives a surjective $\mbz$-algebra homomorphism
\begin{equation*}
\afxrz:\afuglz\tra\afSrmbz.
\end{equation*}
Base change then yields a surjective algebra homomorphism
\begin{equation}\label{afxrk}
\afxrk:=\afxrz\ot 1:\afuglk\tra\afSrk.
\end{equation}

\section{Affine Schur--Weyl reciprocity for loop semigroups}
This section establishes an affine analogue of Schur--Weyl reciprocity for  $\Gk$ over the commutative ring $\field$. We construct an explicit algebra homomorphism from the   algebra $\field \Gk$ to the affine Schur algebra $\afSrk$ and prove its surjectivity under mild conditions on $\field$.
Our results extend the classical finite field case studied by Benson--Doty \cite{BD} to the affine setting and provide explicit combinatorial formulas generalizing the Carter--Lusztig approach \cite{CL}.

Although Theorem~\ref{surjective2} is an affine analogue of the classical
Schur--Weyl reciprocity, its proof requires new arguments arising
from the periodic structure of the extended affine symmetric group
and the infinite-dimensional nature of the loop module
$\Omega_{\mpk}^{\otimes r}$.
In particular, the Laurent polynomial structure of the loop
semigroup $\widehat G_{\mpk}(n)$ is essential for realizing the affine
Schur algebra as a quotient of its semigroup algebra via the
natural action on the affine tensor space.
\subsection{The loop semigroup action on affine tensor space}
Let $M_n(\field[\vep,\vep^{-1}])$ be the set of
$n\times n$ matrices with coefficients in the Laurent polynomial ring $\field[\vep,\vep^{-1}]$, where $\vep$ is an indeterminate.
The $\mpk$-linear map
\begin{equation}\label{vartheta}
\vartheta:M_n(\mpk[\vep,\vep^{-1}])\lra \afMnk,\,\,\,E_{i,j} \vep^l\longmapsto \afE_{i,j+ln}, \,1\le i,j\le n,l\in\mbz
\end{equation}
is a ring isomorphism, where $E_{i,j}$ denotes the elementary matrix $(\delta_{k,i}\delta_{j,l})_{1\le
k,l\le n}$ in $M_n(\field[\vep,\vep^{-1}])$. We will consistently  identify $\afMnk$ with $M_n(\mpk[\vep,\vep^{-1}])$ via this isomorphism.

Let $\afGL:=GL_n(\field((\vep)))$ be the group of invertible
$n\times n$ matrices with coefficients in $\field((\vep))$, where $\field((\vep))$ is the ring of formal Laurent series in $\vep$ with coefficients in $\field$.
Let \begin{equation}\label{the definition of G}
\Gk=\{A\in \afGL\mid a_{i,j}\in\field[\vep,\vep^{-1}],\,\forall i,j\}
=\afGL\cap M_n(\field[\vep,\vep^{-1}]).
\end{equation}
Let $\field \Gk$ be the semigroup $\field$-algebra of $\Gk$. Recall that $\Ogk=\mpk[\vep,\vep^{-1}]^n$ and
$\Ogk^{\ot r}$ carries a natural right $\affSr$-module structure defined in \eqref{action of affine symmetric group}. $\Gk$ acts on $\Ogk$
by matrix multiplication, and thus diagonally on  $\Ogk^{\ot r}$.
Clearly,
the action of $\Gk$ and $\affSr$ on $\Ogk^{\ot r}$ commute. Hence  $\Ogk^{\ot r}$ is a $\Gk$-$\affSr$-bimodule.
This bimodule structure  induces a $\field$-algebra homomorphism
\begin{equation}\label{afzrk}
\afzrk:\field \Gk\ra\afSrk.
\end{equation}
Our main goal is to prove the surjectivity of this map under appropriate conditions on $\field$.

\subsection{Generalization of the Carter--Lusztig Formula}
We now establish affine analogues of the Carter--Lusztig formulas \cite[(25)]{CL}, which give explicit combinatorial descriptions of the images under $\afzrk$ of certain natural elements in $\field \Gk$. In what follows, we define a family of operators on the tensor space $\Ogz^{\ot r}$ that will play a key role in these analogues.

For $i,j\in\mbz$ and a strictly increasing sequence $\la\in\mbn^m$ (i.e., $1\leq\la_s<\la_t\leq r$ for $s<t$), we define
$$F_{i,j,\la}=1^{\ot\la_1-1}\ot\afE_{i,j}\ot 1^{\la_2-\la_1-1}\ot\afE_{i,j}
\ot\cdots\ot 1^{\ot\la_m-\la_{m-1}-1}\ot\afE_{i,j}\ot 1^{\ot r-\la_m}\in\End_\mbz(\Ogz^{\ot r}).$$
More generally, for $\mu\in\mbn^m$ such that $1\leq\mu_s\leq r$ and $\mu_s\not=\mu_t$ whenever $s\not=t$, there exists a permutation $\sg\in\fS_m$ such that
$\la=\mu\sg=(\mu_{\sg(1)},\mu_{\sg(2)},\ldots,\mu_{\sg(m)})$
 is strictly increasing.  We then define $$F_{i,j,\mu}=F_{i,j,\la}\in\End_\mbz(\Ogz^{\ot r}).$$

\begin{Lem}\label{xrz(Eij)}
Let $i,j\in\mbz$ with $\bar i\not=\bar j$. Then we have
$$\afxrz((\afE_{i,j})^{(m)})(w)=
\sum_{\la\in\mbn^m\atop 1\leq\la_1<\la_2<\cdots<\la_m\leq r}
F_{i,j,\la}(w)$$
for $m\geq 1$ and $w\in\Ogz^{\ot r}$.
\end{Lem}
\begin{proof}
We  proceed by induction on $m$. The case where $m=1$ is trivial.
Assume $m\geq 2$ and that the formula holds for $m-1$. By the induction hypothesis we have
\begin{equation*}
\begin{split}
\afxrz((\afE_{i,j})^{(m)})(w)&=\frac{1}{m}\afE_{i,j}(\afxrz((\afE_{i,j})^{(m-1)})(w))
=\frac{1}{m}\afE_{i,j}\sum_{\la\in\mbn^{m-1}\atop 1\leq\la_1<\la_2<\cdots<\la_{m-1}\leq r}
F_{i,j,\la}(w).
\end{split}
\end{equation*}
Since $(\afE_{i,j})^2\Ogz=0$, we conclude that
\begin{equation*}
\begin{split}
\afxrz((\afE_{i,j})^{(m)})(w)&=\frac{1}{m}
\sum_{\tiny\mbox{$\begin{array}{c}
\la\in\mbn^{m-1} \\ 1\leq\la_1<\la_2<\cdots<\la_{m-1}\leq r
\\ 1\leq k\leq r,\,k\not=\la_s,\,\forall s\end{array}$}} \!\!
F_{i,j,(k,\la)}(w)\\
&=\frac{1}{m}\sum_{\mu\in\mbn^m\atop 1\leq\mu_s<\mu_t \leq r,\,\forall s<t}
|X_\mu| F_{i,j,\mu}(w)
\end{split}
\end{equation*}
where $X_\mu$ is the set of the elements $(k,\la)\in\mbn^m$ such that $\la\in\mbn^{m-1}$, $1\leq \la_1<\la_2<\cdots<\la_{m-1}\leq r$, $1\leq k\leq r$, $k\not=\la_s$, $\forall s$ and   $\{\mu_1,\mu_2,\ldots,\mu_{m}\}=\{k,\la_1,\la_2,\ldots,\la_{m-1}\}$. Since $|X_\mu|=m$ for any $\mu$ we conclude that $$\afxrz((\afE_{i,j})^{(m)})(w)=\sum_{\mu\in\mbn^m\atop 1\leq\mu_1<\mu_2<\cdots<\mu_m\leq r} F_{i,j,\mu}(w).$$
The proof is completed.
\end{proof}

\begin{Lem}
Let $i\in\mbz$. Then we have
$$\afxrz\bigg(\bigg({H_i\atop m}\bigg)\bigg)(w)=
\sum_{\la\in\mbn^m\atop 1\leq\la_1<\la_2<\cdots<\la_m\leq r}
F_{i,i,\la}(w)$$
for $m\geq 1$ and $w\in\Ogz^{\ot r}$.
\end{Lem}
\begin{proof}
We proceed by induction on $m$. The case where $m=1$ is trivial.
Assume $m \geq 2$ and that the formula holds for $m-1$.
By the induction hypothesis we have
\begin{equation*}
\begin{split}
\afxrz\bigg(\bigg({H_i\atop m}\bigg)\bigg)(w)&=\frac{H_i-m+1}{m}\afxrz\bigg(\bigg({H_i\atop m-1}\bigg)\bigg)(w)
=\frac{H_i-m+1}{m} \sum_{\la\in\mbn^{m-1}\atop 1\leq\la_1<\la_2<\cdots<\la_{m-1}\leq r}
F_{i,i,\la}(w).
\end{split}
\end{equation*}
A key property of  $H_i$ is that
$H_i^2(\og_j)=\dt_{\bar i,\bar j}\og_j=H_i(\og_j)$ for any $j\in\mbz$. It follows that
$$H_iF_{i,i,\la}(w)=(m-1)F_{i,i,\la}(w)+\sum_{1\leq k\leq r,\,k\not=\la_s,\,\forall s}
F_{i,i,(k,\la)}(w)$$
for $\la\in\mbn^{m-1}$ with $1\leq\la_s<\la_t\leq r,\,\forall s<t$.
This implies that
\begin{equation*}
\begin{split}
\afxrz\bigg(\bigg({H_i\atop m}\bigg)\bigg)(w)&=\frac{1}{m}
\!\!\sum_{\tiny\mbox{$\begin{array}{c}
\la\in\mbn^{m-1} \\ 1\leq\la_1<\la_2<\cdots<\la_{m-1}\leq r
\\ 1\leq k\leq r,\,k\not=\la_s,\,\forall s\end{array}$}}
F_{i,i,(k,\la)}(w) =\sum_{\mu\in\mbn^m\atop 1\leq\mu_1<\mu_2<\cdots<\mu_m\leq r} F_{i,i,\mu}(w).
\end{split}
\end{equation*}
The assertion follows.
\end{proof}

For $i\in\mbz$, $1\leq k\leq r$ and $l\in\mbz$ with $l\not=0$ let
$$R_{i,l,k}=1^{\ot k-1}\ot\afE_{i,i+ln}\ot 1^{\ot r-k}\in\End_\mbz(\Ogz^{\ot r}).$$
Here, $1$ denotes the identity endomorphism on $\Ogz$.
For $i\in\mbz$, $\la\in\mbn^m$, $\mu\in\mbz^m$ with
$1\leq\la_s\not=\la_t\leq r$ for $s\not=t$ and $\mu_s\not=0$ for all $s$, we set
$$F_{i,\mu,\la}=R_{i,\mu_1,\la_1}R_{i,\mu_2,\la_2}\cdots R_{i,\mu_m,\la_m}
\in\End_\mbz(\Ogz^{\ot r}).$$
Then by definition we have
$$F_{i,i+ln,\la}=F_{i,l{\bf 1}_m,\la}$$
for $\la\in\mbn^m$, $1\leq\la_s\not=\la_t\leq r$ for $s\not=t$, $1\leq i\leq n$ and $l\in\mbz$ with $l\not=0$, where ${\bf 1}_m=(1,1,\ldots,1)\in\mbn^m$
denote the constant sequence of length $m$  with all entries $1$.

\begin{Lem}\label{xrz(Laiml)}
Let $1\leq i\leq n$, $l\in\mbz$ with $l>0$. Then we have
$$\afxrz((-1)^m\La_{i,\pm m,l})(w)=
\sum_{\la\in\mbn^m\atop 1\leq\la_1<\la_2<\cdots<\la_m\leq r}
F_{i,i\pm ln,\la}(w)$$
for $m\geq 1$ and $w\in\Ogz^{\ot r}$.
\end{Lem}
\begin{proof}
We proceed  by induction on $m$. The case where $m=1$ is trivial.
Assume now that $m\geq 2$.
By Lemma \ref{Laikl} we have the recurrence:
\begin{equation*}
m\La_{i,\pm m,l}=-\sum_{0\leq k\leq m-1}g_{i,\pm(m-k)l}\La_{i,\pm k,l}
\end{equation*}
Applying $\afxrz$ and using the inductive hypothesis:
\begin{equation*}
\begin{split}
\afxrz(g_{i,\pm(m-k)l}\La_{i,\pm k,l})(w)&=(-1)^k\afxrz(g_{i,\pm(m-k)l})
\sum_{\la\in\mbn^{k}\atop 1\leq\la_1<\la_2<\cdots<\la_k\leq r }
F_{i,i\pm ln,\la}(w)
\end{split}
\end{equation*}
for $0\leq k\leq m-1$. Note that we have $$\afE_{i,s}\afE_{i,t}(\og_j)=\dt_{\bar i,\bar j}\og_{j-(s+t)n}=\afE_{i,s+t}(\og_j)$$
for $j\in\mbz$ and $s,t\in\mbz$.
This implies that
\begin{equation*}
\begin{split}
\afxrz(g_{i,\pm(m-k)l})
F_{i,i\pm ln,\la}(w)&=\sum_{1\leq t\leq r\atop t\not=\la_s,\,\forall s}F_{i,(\pm(m-k)l,\pm l{\bf 1}_k),(t,\la)}(w)
 +\sum_{1\leq j\leq k}F_{i,\pm l{\bf 1}_k\pm(m-k)l\bse_j,\la}(w)
\end{split}
\end{equation*}
for $0\leq k\leq m-1$, $\la\in\mbn^{k}$ with $1\leq\la_s<\la_t\leq r$, $\forall s<t$,
where $\bse_j=(0,\ldots,0,\underset {(j)}1,0,\ldots,0)\in\mbz^k$.
It follows that
\begin{equation*}
\afxrz(g_{i,\pm(m-k)l}\La_{i,\pm k,l})(w)=a_{i,k}+b_{i,k},
\end{equation*}
for $0\leq k\leq m-1$
where
\begin{equation*}
\begin{split}
a_{i,k}&=\sum_{\tiny\mbox{$\begin{array}{c}
\la\in\mbn^{k}\\ 1\leq\la_1<\la_2<\cdots<\la_k\leq r
\\1\leq t\leq r,\,t\not=\la_s,\,\forall s\end{array}$}}\!\!(-1)^kF_{i,(\pm(m-k)l,\pm l{\bf 1}_k),(t,\la)}(w)\\
b_{i,k}&=\sum_{\tiny\mbox{$\begin{array}{c}
\la\in\mbn^{k}\\ 1\leq\la_1<\la_2<\cdots<\la_k\leq r
\\ 1\leq j\leq k\end{array}$}}\!\!(-1)^kF_{i,\pm(m-k)l\bse_j\pm l{\bf 1}_k,\la}(w).
\end{split}
\end{equation*}
Since $a_{i,k}=-b_{i,k+1}$ for $0\leq k<m-1$ and $b_{i,0}=0$ we have
\begin{equation*}
\begin{split}
\afxrz(m\La_{i,\pm m,l})(w)&=-\sum_{0\leq k\leq m-1}(a_{i,k}+b_{i,k})=-a_{i,m-1}\\
&=(-1)^{m}\sum_{\tiny\mbox{$\begin{array}{c}
\la\in\mbn^{m-1}\\ 1\leq\la_1<\la_2<\cdots<\la_{m-1}\leq r
\\1\leq t\leq r,\,t\not=\la_s,\,\forall s\end{array}$}}\!\!F_{i,i\pm ln,(t,\la)}(w)\\
&=(-1)^{m}\sum_{\mu\in\mbn^m\atop 1\leq\mu_1<\mu_2<\cdots<\mu_m\leq r}
|X_\mu| F_{i,i\pm ln,\mu}(w),
\end{split}
\end{equation*}
where $X_\mu$ is the set of the elements $(t,\la)\in\mbn^m$ such that $\la\in\mbn^{m-1}$, $1\leq \la_1<\la_2<\cdots<\la_{m-1}\leq r$, $1\leq t\leq r$, $t\not=\la_s$, $\forall s$ and  $\{\mu_1,\mu_2,\ldots,\mu_m\}=\{t,\la_1,\la_2,\ldots,\la_{m-1}\}$. Since $|X_\mu|=m$ we have
$$\afxrz((-1)^m\La_{i,\pm m,l})(w)=
\sum_{\mu\in\mbn^m\atop 1\leq\mu_1<\mu_2<\cdots<\mu_m\leq r}
F_{i,i\pm ln,\mu}(w).$$
The proof is completed.
\end{proof}

For $i,j\in\mbz$ and $x\in\field$ let $\afE_{i,j}(x)=I+x\afE_{i,j}\in \afMnk$, where $I\in\afMnk$ is the identity matrix.
\begin{Prop}\label{zrk(Eij(x))}
Let $x\in\field$.

$(1)$ Let \(i\ne j\in\mathbb Z\). If \(j=i-ln\) for some \(l>0\), assume in addition that
\(x=0\) or \(x\) is invertible in \(\mpk\).
Then $\afE_{i,j}(x)\in \Gk$ and
$$\afzrk(\afE_{i,j}(x))=
\begin{cases}
\sum_{0\leq m\leq r}x^m\afxrk((\afE_{i,j})^{(m)})
&\text{if $\bar i\not=\bar j$}\\
\sum_{0\leq m\leq r}x^m\afxrk((-1)^m\La_{i,\pm m,l})
&\text{if $j=i\pm ln$ for some $l>0$}
\end{cases}$$

$(2)$ If $x+1$ is invertible in $\field$ then we have
$\afE_{i,i}(x)\in \Gk$ and
$$\afzrk(\afE_{i,i}(x))=\sum_{0\leq m\leq r}x^m\afxrk\bigg(\bigg({H_i\atop m}\bigg)\bigg).$$
\end{Prop}
\begin{proof}
Clearly we have $\afE_{i,j}(x)\in \Gk$. Furthermore, by definition we have
\begin{equation*}
\begin{split}
&\qquad\afzrk(\afE_{i,j}(x))(w_1\ot w_2\ot\cdots\ot w_r)\\
&=(I+x\afE_{i,j})(w_1)\ot (I+x\afE_{i,j})(w_2)\ot\cdots\ot
(I+x\afE_{i,j})(w_r)\\
&=w_1\ot w_2\ot\cdots\ot w_r+
\sum_{\tiny\mbox{$\begin{array}{c}1\leq m\leq r\\
\la\in\mbn^m\\
1\leq\la_1<\la_2<\cdots<\la_m\leq r\end{array}$}} x^m F_{i,j,\la}(w_1\ot w_2\ot\cdots\ot w_r)
\end{split}
\end{equation*}
for $w_1\ot w_2\ot\cdots\ot w_r\in\Ogk^{\ot r}$. Now assertion (1) follows from Lemma
\ref{xrz(Eij)} and \ref{xrz(Laiml)}. Assertion (2) can be proved similarly.
\end{proof}

Combining Propositions \ref{Laikl formula}, \ref{generators}(2), and \ref{zrk(Eij(x))} we obtain the following result.
\begin{Coro}\label{zrk(Eij(x)) coro}
Let $i\not=j\in\mbz$ and $x\in\field$.
If \(j=i-ln\) for some \(l>0\), assume in addition that
\(x=0\) or \(x\) is invertible in \(\mpk\). Then
$$\afzrk(\afE_{i,j}(x))=\sum_{0\leq m\leq r}x^m\afxrk((m\afE_{i,j})[\bfl])
=\sum_{0\leq m\leq r}x^m (m\afE_{i,j})[\bfl,r].$$
\end{Coro}

\subsection{Surjectivity of the loop semigroup representation}
We now generalize the Schur-Weyl reciprocity over finite fields established  by Benson-Doty \cite[Th. 4.3]{BD} to the affine case.
This extension clarifies the relationship between representations of
\(\mathpzc K\widehat G_{\mathpzc K}(n)\) and affine Schur algebras in
positive characteristic.

\begin{Thm}\label{surjective1}
Assume that there is a subset $\mpX$ of $\field$ such that $|\mpX|>r$ and $a-b$ is invertible in $\field$ for any $a\not=b\in\mpX$. Then we have $\afzrk(\field \Gk)=\afSrk$.
\end{Thm}
\begin{proof}
By hypothesis, we can choose distinct elements  $k_0,k_1,\ldots,k_r\in\mpX\han\field$ such that $k_s-k_t$ is invertible in $\field$ for $0\leq s\not=t\leq r$.
Let
$a_t=k_t-k_0$ for $(0\leq t\leq r).$
Then \(a_0=0\) and \(a_t\) is invertible in $\mpk$ for \(1\leq t\leq r\).
Moreover,
$a_s-a_t=k_s-k_t$ for $s\ne t.$
For fixed $i \neq j$, consider the elements $\afE_{i,j}(a_t) = I + a_t \afE_{i,j} \in \Gk$ for $t = 0, 1, \ldots, r$.
By Corollary \ref{zrk(Eij(x)) coro} for $i\not=j$ we have
$$(\afzrk(\afE_{i,j}(a_0)),\afzrk(\afE_{i,j}(a_1)),\ldots,\afzrk(\afE_{i,j}(a_r)))
=(1,\afE_{i,j}[\bfl,r],\ldots,(r\afE_{i,j})[\bfl,r])B,$$
where $B$ is the Vandermonde matrix:
\begin{equation}\label{Vandermonde matrix}
B=\begin{pmatrix}
1&1&\cdots&1\\
a_0&a_1&\cdots&a_r\\
\hdotsfor{4}\\
a_0^r&a_1^r&\cdots&a_r^r
\end{pmatrix}.
\end{equation}
The determinant of $B$ is $\det(B) = \prod_{0 \leq t < s \leq r} (a_s - a_t)$, which is invertible in $\field$ by our choice of $a_t$'s.
Therefore, $B$ is invertible, and we can express each $(m\afE_{i,j})[\bfl,r]$ as a linear combination of the $\afzrk(\afE_{i,j}(a_t))$'s.
It follows that
$(m\afE_{i,j})[\bfl,r]\in\afzrk(\field \Gk)$ for $0\leq m\leq r$ and $i\not=j$.
Now, by Theorem \ref{generators2} (2) and Proposition \ref{generators},  the algebra $\afSrk$ is generated by the elements $(m\afE_{i,j})[\bfl,r]$ and $\afxrk\big(\big({H_i\atop t}\big)\big)$ for $1\leq i\leq n$, $j\in\mbz$ with $i\not=j$ and $m,t\in\mbn$. Hence by \cite[Th. 9.3]{DFW12}, the algebra $\afSrk$ is generated by the elements $(m\afE_{i,j})[\bfl,r]$ for $1\leq i\leq n$, $j\in\mbz$ with $i\not=j$ and $m\in\mbn$.
However,
note that $(m\afE_{i,j})[\bfl,r]=0$ for $m>r$ and $i\not=j$. Therefore $\afSrk$ is actually generated by the elements $(m\afE_{i,j})[\bfl,r]$
 for $1\leq i\leq n$, $j\in\mbz$ with $i\not=j$ and $0\leq m\leq r$. Since all these generators belong to $\afzrk(\field \Gk)$, we conclude that $\afzrk(\field \Gk)=\afSrk$.
\end{proof}
\begin{Rem}
The condition on $\mpX$ ensures we have enough evaluation points to make the Vandermonde matrix invertible. This is analogous to the situation in classical Schur-Weyl duality where we require the field to be large enough to guarantee surjectivity.
\end{Rem}

As an immediate consequence of Theorem \ref{surjective1}, we obtain the following result.
\begin{Thm}\label{surjective2}
Assume that $\field$ is a field and $|\field|>r$. Then the algebra homomorphism $\afzrk:\field \Gk\ra\afSrk$ is surjective.
\end{Thm}

\begin{Rem}
The finite-field case of Theorem~\ref{surjective2} has a natural connection with
representation theory of reductive groups over local fields considered by Vign\'eras~\cite{Vigneras}.
Let $q=p^f$. Then
\(
\widehat G_{\mathbb F_q}(n)\subseteq GL_n(\mathbb F_q((\vep))),
\)
where $\mathbb F_q((t))$ is a non-archimedean local field. If $q>r$,
Theorem~\ref{surjective2} gives a surjective algebra homomorphism
\[
\mathbb F_q\widehat G_{\mathbb F_q}(n)
\longrightarrow
\widehat {\mathcal S}_{\mathbb F_q}(n,r).
\]
Consequently, for every field extension $R/\mathbb F_q$, base change yields
a surjective algebra homomorphism
\begin{equation}\label{hat zeta^q_{r,R}}
\widehat\zeta^{\,q}_{r,R}:
R\widehat G_{\mathbb F_q}(n)
\cong
R\otimes_{\mathbb F_q}
\mathbb F_q\widehat G_{\mathbb F_q}(n)
\longrightarrow
\widehat {\mathcal S}_R(n,r).
\end{equation}
In particular, if $R$ is an algebraically closed extension of
$\mathbb F_q$, then $\operatorname{char} R=p$.

Vign\'eras~\cite{Vigneras} studies smooth $R$-representations of $GL_n(F)$,
where $F$ is a non-archimedean local field of residual
characteristic $p$,
over algebraically closed coefficient fields $R$ whose characteristic is
different from $p$, and remarks that
``the case where the characteristic of $R$ is $p$ remains mysterious.''
In the special case $F=\mathbb F_q((\vep))$, the present construction
therefore lies on the defining-characteristic side, but in a different
representation-theoretic setting: we consider polynomial representations
arising from the natural Laurent-polynomial semigroup
$\widehat G_{\mathbb F_q}(n)$, rather than smooth representations of the
full group $GL_n(\mathbb F_q((\vep)))$. Thus, at the level of this natural subsemigroup, our construction provides a defining-characteristic counterpart to the cross-characteristic setting considered by Vign\'eras.
\end{Rem}
\subsection{Failure of extension from the loop semigroup to loop groups}

Throughout this subsection, we assume that $\mathpzc K$ is a
field. Although $\widehat G_{\mathpzc K}(n)$ is only a semigroup,
it generates a natural subgroup of the loop group
\(
\widehat{GL}_n(\mathpzc K)
=
GL_n(\mathpzc K((\vep))).
\)
The following observation identifies this subgroup explicitly.

\begin{Prop}\label{rational loop group}
The subgroup of $GL_n(\mathpzc K((\vep)))$ generated by
$\widehat G_{\mathpzc K}(n)$ is $GL_n(\mathpzc K(\vep))$. In other
words,
\(
\bigl\langle\widehat G_{\mathpzc K}(n)\bigr\rangle
=
GL_n(\mathpzc K(\vep)).
\)
\end{Prop}

\begin{proof}
Clearly we have
\(
\bigl\langle\widehat G_{\mathpzc K}(n)\bigr\rangle
\subseteq
GL_n(\mathpzc K(\vep)).
\)
Conversely, let
\(
A\in GL_n(\mathpzc K(\vep)).
\)
There exists a nonzero element
$d\in\mathpzc K[\vep,\vep^{-1}]$ such that
\(
dA\in M_n(\mathpzc K[\vep,\vep^{-1}]).
\)
Since
\(
\det(dA)=d^n\det(A)\ne0,
\)
the matrix $dA$ is invertible over $\mathpzc K(\vep)$, and hence
over $\mathpzc K((\vep))$. Therefore
\(
dA\in\widehat G_{\mathpzc K}(n).
\)
Similarly,
\(
dI\in\widehat G_{\mathpzc K}(n)
\)
where $I\in M_n(\mathpzc K[\vep,\vep^{-1}])$ is the identity matrix.
It follows that
\(
A=(dI )^{-1}(dA)
\in
\bigl\langle\widehat G_{\mathpzc K}(n)\bigr\rangle.
\)
Thus
\(
GL_n(\mathpzc K(\vep))
\subseteq
\bigl\langle\widehat G_{\mathpzc K}(n)\bigr\rangle,
\)
and the assertion follows.
\end{proof}

One might therefore ask whether the homomorphism
$\widehat\zeta_{r,\mathpzc K}$ defined in \textup{(4.3)}
extends from the semigroup algebra
$\mathpzc K\widehat G_{\mathpzc K}(n)$ to the group algebra of
the rational loop group $GL_n(\mathpzc K(\vep))$. The following
result shows that this is not the case.

\begin{Prop}\label{not extension}
Assume that $r>0$. The algebra homomorphism
\(
\widehat\zeta_{r,\mathpzc K}:
\mathpzc K\widehat G_{\mathpzc K}(n)
\rightarrow
\widehat{\mathcal S}_{\mathpzc K}(n,r)
\)
does not extend to an algebra homomorphism
\(
\mathpzc KGL_n(\mathpzc K(\vep))
\rightarrow
\widehat{\mathcal S}_{\mathpzc K}(n,r).
\)
Consequently, it does not extend to an algebra homomorphism
\(
\mathpzc KGL_n(\mathpzc K((\vep)))
\rightarrow
\widehat{\mathcal S}_{\mathpzc K}(n,r).
\)
\end{Prop}

\begin{proof}
Let
\(
g=(1-\vep)I ,
\) where $I\in M_n(\mathpzc K[\vep,\vep^{-1}])$ is the identity matrix. Then
we have
\(
g\in\widehat G_{\mathpzc K}(n).
\)
On the other hand, $\widehat\zeta_{r,\mathpzc K}(g)$ is not invertible
in $\widehat{\mathcal S}_{\mathpzc K}(n,r)$.
 Indeed, let
\(
\varepsilon:\Omega_{\mathpzc K}\rightarrow\mathpzc K^n
\)
be the evaluation map at $t=1$. Then
\(
\varepsilon^{\otimes r}\circ
\widehat\zeta_{r,\mathpzc K}(g)=0,
\)
while $\varepsilon^{\otimes r}\ne0$ since $r>0$. Thus
$\widehat\zeta_{r,\mathpzc K}(g)$ is not surjective, and hence is
not invertible.

If $\widehat\zeta_{r,\mathpzc K}$ extended to
$\mathpzc KGL_n(\mathpzc K(t))$, then the image of the invertible
element $g$ would have to be invertible, a contradiction. Hence no
such extension exists. The assertion for
$GL_n(\mathpzc K((t)))$ follows immediately from
\(
GL_n(\mathpzc K(t))\subseteq GL_n(\mathpzc K((t))).
\)
\end{proof}

This shows that the semigroup framework is essential for the
tensor-space construction considered here, since the natural
action of $\widehat G_{\mathpzc K}(n)$ does not extend to the
rational loop group $GL_n(\mathpzc K(\vep))$ that it generates.

\section{The double centralizer property and centers of affine Schur algebras}

In Theorem~\ref{surjective2} we established one direction of affine Schur--Weyl reciprocity for the Laurent polynomial loop semigroup. We now consider the opposite centralizer. We first establish the
double centralizer property in the case \(n\geq r\) over arbitrary
commutative rings.  We then remove the restriction \(n\geq r\) over unique
factorization domains and prove the result for all \(n\geq2\).
We also show that,  for \(n=1\), the corresponding natural homomorphism is not
surjective for \(r\geq2\). Finally, as an application of the
double centralizer property, we determine the center of the affine
Schur algebra and relate it to the centers of the hyper loop algebra
and the loop semigroup algebra.

\subsection{The case \(n\geq r\) over arbitrary commutative rings}
We begin with the case \(n\geq r\), where the double centralizer property
follows directly from the standard idempotent realization of the affine
tensor space. This argument works over an arbitrary commutative ring.

The right action of $\widehat{\frak S}_r$ on $\Omega_{\mathpzc K}^{\otimes r}$
induces an algebra homomorphism
\[
\rho_{r,\mathpzc K}:
\mathpzc K\widehat{\frak S}_r^{\,\mathrm{op}}
\longrightarrow
\operatorname{End}_{\widehat{\mathcal S}_{\mathpzc K}(n,r)}
(\Omega_{\mathpzc K}^{\otimes r}).
\]

\begin{Thm}\label{thm:second-centralizer-stable}
Assume that $n\geq r$.

$(1)$ The algebra homomorphism
\(
\rho_{r,\mathpzc K}:
\mathpzc K\widehat{\frak S}_r^{\,\mathrm{op}}
\rightarrow
\operatorname{End}_{\widehat{\mathcal S}_{\mathpzc K}(n,r)}
(\Omega_{\mathpzc K}^{\otimes r})
\)
is an isomorphism.

$(2)$ If there is a subset $\mpX$ of $\field$ such that $|\mpX|>r$ and $a-b$ is invertible in $\field$ for any $a\not=b\in\mpX$, then
\(
\rho_{r,\mathpzc K}:
\mathpzc K\widehat{\frak S}_r^{\,\mathrm{op}}
\rightarrow
\operatorname{End}_{\widehat G_{\mathpzc K}(n)}
(\Omega_{\mathpzc K}^{\otimes r})
\)
is an isomorphism.
\end{Thm}

\begin{proof}
 Let $\varpi=(\varpi_1,\cdots,\varpi_n)\in\afLanr$, where $\varpi_i=1$ if $1\leq i\leq r$, and $\varpi_{i}=0$ if $r<i\leq n$.
Since $n\geq r$ we have $\Omega_{\mathpzc K}^{\otimes r}\cong \afSrk e_{\diag(\varpi)}$. It follows that $$\operatorname{End}_{\widehat{\mathcal S}_{\mathpzc K}(n,r)}
(\Omega_{\mathpzc K}^{\otimes r})\cong (e_{\diag(\varpi)}\afSrk e_{\diag(\varpi)})^{\,\mathrm{op}}\cong \mathpzc K\widehat{\frak S}_r^{\,\mathrm{op}}.$$
The assertion follows. Under these identifications, the resulting isomorphism is precisely \(\rho_{r,\mathpzc K}\). This proves (1).

For (2), by Theorem~\ref{surjective1}, the images of
$\mathpzc K\widehat G_{\mathpzc K}(n)$ and
$\widehat{\mathcal S}_{\mathpzc K}(n,r)$ in
$\operatorname{End}_{\mathpzc K}(\Omega_{\mathpzc K}^{\otimes r})$
coincide. Hence
\(
\operatorname{End}_{\widehat G_{\mathpzc K}(n)}
(\Omega_{\mathpzc K}^{\otimes r})
=
\operatorname{End}_{\widehat{\mathcal S}_{\mathpzc K}(n,r)}
(\Omega_{\mathpzc K}^{\otimes r}),
\)
and the assertion follows from (1).
\end{proof}

\subsection{The case \(n\geq 2\) over unique factorization domains}

In the geometric setting of affine Hecke algebras and affine
$q$-Schur algebras, Pouchin~\cite[Theorem~8.1]{Pouchin} stated an
analogous double centralizer property in the case corresponding to
$n<r$ in our notation. A gap in the proof of this result was already
pointed out in \cite{Fu13}. More precisely, for the surjection
\(
\phi_I:X\longrightarrow Y_I,
\)
the proof defines an injection
\(
\psi_I:\mbc_G(Y_I\times Y_I)
\longrightarrow
\mbc_G(Y_I\times X).
\)
However, the proof subsequently applies $\psi_I$ to an arbitrary element
$g\in \mbc_G(Y_I\times X)$. Since the domain of $\psi_I$ is
$\mbc_G(Y_I\times Y_I)$, the expression $\psi_I(g)$ is not defined for such
a general $g$. Consequently, the final argument does not, as written,
establish the asserted surjectivity.
The double centralizer property for affine quantum Schur
algebras was proved in \cite{FuHu} over fields of characteristic zero
when the Hecke parameter is not a root of unity.

We now remove the restriction \(n\geq r\) under the additional
assumption that the coefficient ring is a unique factorization
domain. Our argument below is independent of Pouchin's geometric construction.
We first extend scalars to a suitable field, where the relevant
skew group algebra becomes a matrix algebra and the double centralizer
property can be proved directly. We then prove that the same property
holds over the original coefficient ring.

Throughout this subsection, $\mathpzc K$ is a commutative unique factorization domain,
$n\geq 2$, and $r\geq 1$.  Set
\[
\sfT=\Omega_{\mathpzc K}^{\otimes r},
\qquad
\sfH=\mathpzc K\widehat{\frak S}_r,
\qquad
\sfS=
\widehat{\mathcal S}_{\mathpzc K}(n,r)
=\operatorname{End}_{\sfH}(\sfT),\qquad \sfP=\mathpzc K[x_1^{\pm1},\ldots,x_r^{\pm1}].
\]
Since $
\mathpzc K\mathbb Z^r
\simeq
\sfP$, the natural action of \({\frak S}_r\) on \(\mathbb Z^r\) induces the action on \(\sfP\) given by
$w(x_i)=x_{w(i)}$ for $w\in {\frak S}_r$.
Consequently, there is a natural algebra isomorphism
\begin{equation}\label{sfH iso P rtimes Sr}
\sfH
=
\mathpzc K\widehat{\mathfrak S}_r\cong
\mathpzc K(\mathbb Z^r\rtimes {\frak S}_r)
\cong
\sfP\rtimes {\frak S}_r ,
\end{equation}
where \(\sfP\rtimes {\frak S}_r\) denotes the skew group algebra associated with this action.
We shall identify \(\sfH\) with \(\sfP\rtimes {\frak S}_r\) via this isomorphism. Note that for
$f\in \sfP$, $w\in {\frak S}_r,$ we have
$
wfw^{-1}=w(f) .$

Let $I(n,r)=\{\mathbf i =(i_1,\ldots,i_r)\mid 1\leq i_j\leq n,\,\text{for}\,1\leq j\leq r\}$.
Since $\omega_{i+ln}=\ttt^{-l}\omega_i$, we have
$$
\sfT=
\bigoplus_{\mathbf i\in I(n,r)}
\omega_{\mathbf i}\sfP,$$
where $\omega_{\mathbf i}
=\omega_{i_1}\otimes\cdots\otimes\omega_{i_r}.$
For $w\in {\frak S}_r$, write
\[
(\mathbf i w)_j=i_{w(j)}
\qquad (1\leq j\leq r).
\]
Then we have $\omega_{\mathbf i}w=\omega_{\mathbf i w}$.
Let
$\sfZ=\sfP^{{\frak S}_r}$. We first prove the following basic freeness property.
\begin{Lem}\label{lem:sc-freeness}
Both \(\sfH\) and \(\sfT\) are free \(\sfZ\)-modules of finite rank.
\end{Lem}

\begin{proof}
By \cite[Rem.~2.14]{Rouquier}, \(\sfP\) is a free
\(\sfZ\)-module of finite rank. Since
$
\sfH=\bigoplus_{w\in{\frak S}_r}\sfP w
$
and
$
\sfT=
\bigoplus_{\mathbf i\in I(n,r)}
\omega_{\mathbf i}\sfP$, it follows that
both \(\sfH\) and \(\sfT\) are free \(\sfZ\)-modules of finite rank.
\end{proof}

The following lemma describes the center of the skew group algebra
$
\sfH=\sfP\rtimes\mathfrak S_r .$
This description may be viewed as the \(v=1\) degeneration of the
Bernstein center theorem for affine Hecke algebras. In the present
group algebra case, however, it follows directly from the skew group
algebra structure.

\begin{Lem}\label{center}
Let \(Z(\sfH)\) denote the center of the algebra \(\sfH\). Then
$Z(\sfH)=\sfZ .$
\end{Lem}

\begin{proof}
It is clear that
$
\sfZ\subseteq Z(\sfH).$
Conversely, let
$z=\sum_{w\in\mathfrak S_r}f_w w\in Z(\sfH)$,
where $f_w\in\sfP .$
For each \(f\in\sfP\), we have
$
\sum_{w\in\mathfrak S_r}f_w w(f)w=zf=fz=
\sum_{w\in\mathfrak S_r}ff_w w.$
It follows that
$f_w(w(f)-f)=0$
for
$f\in\sfP$ and $w\in\mathfrak S_r$.
If \(w\neq1\), choose \(i\) with \(w(x_i)\neq x_i\).
Then
$f_w(w(x_i)-x_i)=0,$
and hence \(f_w=0\), since \(\sfP\) is an integral domain.
Thus
$
z=f_1\in\sfP .$
Finally, since \(z\) commutes with every element of
\(\mathfrak S_r\), we have
$
wf_1=f_1w ,$
which implies
$f_1\in\sfP^{\mathfrak S_r}=\sfZ.$
Therefore
$
Z(\sfH)=\sfZ.$
\end{proof}

Since $\mathpzc K$ is an integral domain, both $\sfP$ and $\sfZ$ are integral domains. Put
\[
\sfQ=\operatorname{Frac}(\sfZ),
\qquad
\sfF=\operatorname{Frac}(\sfP).
\]
We next describe the relation among \(\sfP\), \(\sfQ\), and \(\sfF\).
\begin{Lem}\label{lem:sc-generic-coefficient-field}
There is a natural isomorphism
$
\sfP\otimes_\sfZ \sfQ\cong \sfF.
$
Moreover,
$\sfF^{\frak S_r}=\sfQ$
 and $[\sfF:\sfQ]=|{\frak S}_r|=r!.$
\end{Lem}

\begin{proof}
Let $S=\sfZ\setminus\{0\}$.  Since $\sfP$ is a domain and $S\subseteq \sfP\setminus\{0\}$, we have
\(
S^{-1}\sfP\subseteq \sfF
\).
Conversely, let $0\ne f\in \sfP$ and define
$
N(f)=\prod_{w\in {\frak S}_r}w(f).$
Then $N(f)\in \sfP^{{\frak S}_r}=\sfZ$ and $N(f)\ne0$.  Hence $N(f)$ is invertible in
$S^{-1}\sfP$, and
$f^{-1}
=
\frac{1}{N(f)}\cdot \prod_{ {w\in {\frak S}_r,\,w\ne1}}w(f)\in S^{-1}\sfP.
$
Thus $S^{-1}\sfP=\sfF$. It follows that
$
\sfP\otimes_\sfZ\sfQ=\sfP\otimes_\sfZ S^{-1}\sfZ\cong S^{-1}\sfP=\sfF.$

Clearly $\sfQ\subseteq \sfF^{{\frak S}_r}$.  On the other hand, let
$x=a/b\in \sfF^{{\frak S}_r}$ with $a,b\in \sfP$ and $b\ne0$. Since $x\in\sfF^{{\frak S}_r}$ and $N(b)\in\sfP^{{\frak S}_r}$, we have
\[
xN(b)
=a\prod_{\substack{w\in {\frak S}_r,\,w\ne1}}w(b)
\in \sfP\cap \sfF^{{\frak S}_r}=\sfP^{{\frak S}_r}=\sfZ.
\]
It follows that $x\in \sfQ$. Therefore $\sfF^{{\frak S}_r}=\sfQ$.

By \cite[Th. V.2.15]{Hun},   \(\sfF\) is a finite Galois extension of $\sfF^{{\frak S}_r}$ with Galois group ${\frak S}_r $. Hence by  \cite[Th. V.2.5(i)]{Hun},
we have
$
[\sfF:\sfQ]=[\sfF:\sfF^{{\frak S}_r}]
=
|\operatorname{Aut}_{\sfF^{{\frak S}_r}}\sfF|
=
|{\frak S}_r|
=
r!.$
\end{proof}

We shall use the following linear independence result for ring homomorphisms.
\begin{Lem}\label{lem:sc-hom-independence}
Let $A$ be a commutative ring, let $L$ be a field, and let
$
\psi_1,\ldots,\psi_s:A\rightarrow L$
be pairwise distinct unital ring homomorphisms.  Then
$\psi_1,\ldots,\psi_s$, regarded as functions from $A$ to $L$, are linearly
independent over $L$.
\end{Lem}

\begin{proof}
The proof is the same as the standard argument for the linear
independence of characters; see, for example,
\cite[Ch.~VI, \S4, Th.~4.1]{Lang}.
Suppose that a nontrivial linear relation exists. Choose one with the
smallest possible number of nonzero terms. Thus there exist distinct
indices
$i_1,\ldots,i_t\in\{1,\ldots,s\}$
and nonzero coefficients
$c_1,\ldots,c_t\in L$
such that
$\sum_{j=1}^{t}c_j\psi_{i_j}(a)=0$ for
$a\in A,$
where \(t\) is minimal. Then  we have $t\geq 2$.
Since  $
\psi_{i_1}\neq\psi_{i_t}$, there exists \(b\in A\) such that
$\psi_{i_1}(b)\neq\psi_{i_t}(b).$
Replacing \(a\) by \(ba\) in the above relation and subtracting
\(\psi_{i_t}(b)\) times the original relation, we obtain
\[
\sum_{j=1}^{t}
c_j\psi_{i_j}(ba)-\sum_{j=1}^{t}
c_j\psi_{i_t}(b)\psi_{i_j}(a)=\sum_{j=1}^{t-1}
c_j\bigl(\psi_{i_j}(b)-\psi_{i_t}(b)\bigr)
\psi_{i_j}(a)=0
\qquad(a\in A).
\]
The coefficient of \(\psi_{i_1}\) is nonzero, so this gives a nontrivial
linear relation involving fewer than \(t\) homomorphisms, contradicting
the minimality of \(t\).
\end{proof}
We now apply the preceding results to identify the skew group algebra \(\sfF\rtimes{\frak S}_r\).
 \begin{Lem}\label{lem:sc-skew-matrix}
There is an isomorphism of $\sfQ$-algebras
$
\sfF\rtimes {\frak S}_r
\cong
\operatorname{End}_{\sfQ}(\sfF).$
Consequently,
$\sfF\rtimes {\frak S}_r
\cong
\operatorname{Mat}_{r!}(\sfQ).$
\end{Lem}

\begin{proof}
Consider the natural action homomorphism
\begin{equation}\label{eq:sc-skew-action}
\sfF\rtimes {\frak S}_r
\longrightarrow
\operatorname{End}_{\sfQ}(\sfF),
\qquad
\sum_{w\in{\frak S}_r}a_w w
\longmapsto
\left(
x\longmapsto
\sum_{w\in{\frak S}_r}a_w w(x)
\right).
\end{equation}
Since
$
\sfQ=\sfF^{{\frak S}_r},$
the above map is well defined and $\sfQ$-linear.
By Lemma~\ref{lem:sc-hom-independence}, the automorphisms
$
w:\sfF\rightarrow\sfF
\
(w\in{\frak S}_r)$
are linearly independent over $\sfF$. Hence the map
\eqref{eq:sc-skew-action} is injective.
On the other hand, by Lemma~\ref{lem:sc-generic-coefficient-field},
$
\dim_{\sfQ}(\sfF\rtimes{\frak S}_r)
=
[\sfF:\sfQ]\cdot |{\frak S}_r|
=
(r!)^2=\dim_{\sfQ}\operatorname{End}_{\sfQ}(\sfF).$
Thus
we have
$\sfF\rtimes{\frak S}_r
\cong
\operatorname{End}_{\sfQ}(\sfF)
\cong
\operatorname{Mat}_{r!}(\sfQ).$
\end{proof}

Let
$\sfH_\sfQ=\sfH\otimes_{\sfZ}\sfQ$ and $\sfT_\sfQ=\sfT\otimes_{\sfZ}\sfQ.$
The following consequence of Lemma~\ref{lem:sc-freeness} will be used
throughout the sequel.
\begin{Lem}\label{lem:sc-injective-extension}
The natural maps
\[
\sfH\longrightarrow\sfH_\sfQ,\qquad h\longmapsto h\otimes1,
\]
and
\[
\sfT\longrightarrow\sfT_\sfQ,\qquad m\longmapsto m\otimes1,
\]
are injective. Hence we identify \(\sfH\) with its image in
\(\sfH_\sfQ\) and \(\sfT\) with its image in
\(\sfT_\sfQ\).
\end{Lem}

Combining the preceding results, we obtain the following description of
\(\sfH_\sfQ\).
{ \begin{Lem}\label{lem:sc-HQ-matrix}
There is an isomorphism of $\sfQ$-algebras
$
\sfH_\sfQ
\cong
\sfF\rtimes{\frak S}_r
\cong
\operatorname{Mat}_{r!}(\sfQ).$
\end{Lem}

\begin{proof}
By \eqref{sfH iso P rtimes Sr}
$
\sfH=\sfP\rtimes{\frak S}_r
=
\bigoplus_{w\in{\frak S}_r}\sfP w.
$
Hence there is a natural $\sfQ$-algebra isomorphism
\[
\sfH\otimes_{\sfZ}\sfQ
\longrightarrow
(\sfP\otimes_{\sfZ}\sfQ)\rtimes{\frak S}_r
\]
given by
$
(fw)\otimes g
\longmapsto
(f\otimes g)w ,
\
f\in\sfP,\ g\in\sfQ,\ w\in{\frak S}_r .
$ Therefore by Lemmas \ref{lem:sc-generic-coefficient-field} and \ref{lem:sc-skew-matrix} we have
$
\sfH_\sfQ
\cong
\sfF\rtimes{\frak S}_r\cong
\operatorname{Mat}_{r!}(\sfQ).$
\end{proof}}

Let
\[
\sfS_\sfQ=
\operatorname{End}_{\sfH_\sfQ}(\sfT_\sfQ).
\]
We now establish the double centralizer property over \(\sfQ\).

\begin{Lem}\label{lem:sc-generic-double-centralizer}
The right action of $\sfH_\sfQ$ on $\sfT_\sfQ$ induces an algebra
isomorphism
$\rho_\sfQ:
\sfH_\sfQ^{\mathrm{op}}
\rightarrow
\operatorname{End}_{\sfS_\sfQ}(\sfT_\sfQ).$
\end{Lem}
\begin{proof}
The right action of $\sfH_\sfQ$ on $\sfT_\sfQ$ induces an algebra homomorphism
$
\rho_\sfQ:
\sfH_\sfQ^{\mathrm{op}}
\rightarrow
\operatorname{End}_{\sfS_\sfQ}(\sfT_\sfQ).
$
By Lemma~\ref{lem:sc-HQ-matrix},
$\sfH_\sfQ\cong\operatorname{Mat}_{r!}(\sfQ),$
so \(\sfH_\sfQ\) is a simple $\sfQ$-algebra. Hence $\ker(\rho_\sfQ)=0$ or $\ker(\rho_\sfQ)=\sfH_\sfQ^{\mathrm{op}}$. Since \(\sfT_\sfQ\neq0\), the latter case cannot occur. Therefore
\(\rho_\sfQ\) is injective.

Let
$A=\rho_\sfQ(\sfH_\sfQ^{\mathrm{op}})
\subseteq
\operatorname{End}_\sfQ(\sfT_\sfQ).$
Then
$A\cong\sfH_\sfQ^{\mathrm{op}}$ is a simple algebra. Moreover,
$
\operatorname{End}_{A}(\sfT_\sfQ)
=
\operatorname{End}_{\sfH_\sfQ}(\sfT_\sfQ)
=
\sfS_\sfQ .
$
By the double centralizer theorem for finite dimensional semisimple algebras
(see, for example, \cite[Th. 5.18.1]{EGHLSVY}), we have
$A=\operatorname{End}_{\sfS_\sfQ}(\sfT_\sfQ).
$ Since \(\rho_\sfQ\) is injective, it follows that
$\rho_\sfQ$ is an algebra isomorphism.
\end{proof}

To relate the double centralizer over \(\sfQ\) to the original algebra
\(\sfS\), we need the following base-change result.
\begin{Lem}\label{lem:sc-base-change}
There is a natural isomorphism of $\sfQ$-algebras
$
\sfS\otimes_{\sfZ}\sfQ
\xrightarrow{\ \sim\ }
\operatorname{End}_{\sfH_\sfQ}(\sfT_\sfQ)=\sfS_\sfQ.$
\end{Lem}

\begin{proof}
 Since \(\sfH\) is finitely generated as a
\(\sfZ\)-module, we may choose \(\sfZ\)-module generators
$h_1,\ldots,h_N$
of \(\sfH\). For \(1\leq i\leq N\), define
$R_{h_i}:\sfT\longrightarrow\sfT$ such that $R_{h_i}(t)=th_i$
for $t\in \sfT$.
By Lemma \ref{center} we have
\(\sfZ= Z(\sfH)\). Hence we have \(R_{h_i}\in\operatorname{End}_{\sfZ}(\sfT)\) for all $i$.
Define a \(\sfZ\)-linear map
$\alpha:
\operatorname{End}_{\sfZ}(\sfT)\rightarrow \operatorname{End}_{\sfZ}(\sfT)^{\oplus N}
$
by
$
\alpha(f)
=
\bigl(
fR_{h_1}-R_{h_1}f,\ldots,
fR_{h_N}-R_{h_N}f
\bigr).$
Then
$
\ker\alpha=\operatorname{End}_{\sfH}(\sfT)=\sfS.$
Hence we have an exact sequence
\[
0\longrightarrow
\sfS
\longrightarrow
\operatorname{End}_{\sfZ}(\sfT)
\xrightarrow{\alpha}
\operatorname{End}_{\sfZ}(\sfT)^{\oplus N}.
\]
By \cite[Prop. 3.3 and 3.5]{AM} we see that
$\sfQ=\operatorname{Frac}(\sfZ)$ is a flat
\(\sfZ\)-module. Hence we have the following exact sequence
\[
0\longrightarrow
\sfS\otimes_{\sfZ}\sfQ
\longrightarrow
\operatorname{End}_{\sfZ}(\sfT)\otimes_{\sfZ}\sfQ
\xrightarrow{\alpha\otimes1}
\operatorname{End}_{\sfZ}(\sfT)^{\oplus N}\otimes_{\sfZ}\sfQ.
\]
Consequently,
$\sfS\otimes_{\sfZ}\sfQ
=
\ker(\alpha\otimes1)$.

By Lemma \ref{lem:sc-freeness}  the natural
map
$
\Phi:
\operatorname{End}_{\sfZ}(\sfT)\otimes_{\sfZ}\sfQ
\rightarrow
\operatorname{End}_{\sfQ}(\sfT_\sfQ)$
defined by
$\Phi(f\otimes a)(t\otimes b)=f(t)\otimes ab$
is an isomorphism of \(\sfQ\)-algebras.
The map $\Phi$ induces a \(\sfQ\)-algebra isomorphism
$$\Phi_N:\operatorname{End}_{\sfZ}(\sfT)^{\oplus N}\otimes_{\sfZ}\sfQ\longrightarrow
\operatorname{End}_{\sfQ}(\sfT_\sfQ)^{\oplus N}$$
such that $\Phi_N((f_1,\ldots,f_N)\otimes a)=(\Phi(f_1\otimes a),\ldots,\Phi(f_N\otimes a))$
for $a\in\sfQ$ and $f_i\in \operatorname{End}_{\sfZ}(\sfT)$.

Define a \(\sfQ\)-linear map
$$
\alpha_\sfQ:
\operatorname{End}_{\sfQ}(\sfT_\sfQ)
\longrightarrow
\operatorname{End}_{\sfQ}(\sfT_\sfQ)^{\oplus N}
$$
such that
$
\alpha_\sfQ(g)
=\bigl(
gR_{h_1\otimes1}-R_{h_1\otimes1}g,\ldots,
gR_{h_N\otimes1}-R_{h_N\otimes1}g
\bigr)$.
Since \(h_1,\ldots,h_N\) generate \(\sfH\) as a \(\sfZ\)-module,
the elements
$h_1\otimes1,\ldots,h_N\otimes1$
generate
$\sfH_\sfQ=\sfH\otimes_{\sfZ}\sfQ$
as a \(\sfQ\)-vector space. Therefore
$\ker\alpha_\sfQ
=
\operatorname{End}_{\sfH_\sfQ}(\sfT_\sfQ).$
Since $\alpha_\sfQ\circ\Phi=\Phi_N\circ(\alpha\otimes 1)$, we have
$\operatorname{End}_{\sfH_\sfQ}(\sfT_\sfQ)=\ker\alpha_\sfQ=\Phi(\ker(\alpha\otimes 1))=\Phi(\sfS\otimes_{\sfZ}\sfQ)$. Thus we have
$\sfS\otimes_{\sfZ}\sfQ\cong\operatorname{End}_{\sfH_\sfQ}(\sfT_\sfQ)$.
\end{proof}

 By Lemma \ref{lem:sc-HQ-matrix}
every element $h$ of $\sfH_\sfQ$ has a unique expression
\begin{equation*}
h=\sum_{w\in {\frak S}_r}w f_w,
\qquad f_w\in \sfF.
\end{equation*}
Moreover, for $f\in \sfF$ and $\mathbf i\in I(n,r)$ we have
\begin{equation}\label{eq:sc-right-action-formula}
(\omega_{\mathbf i}f)w
=\omega_{\mathbf i}(fw)
=(\omega_{\mathbf i}w)w^{-1}(f)
=\omega_{\mathbf i w}w^{-1}(f).
\end{equation}
To deduce the original double centralizer property from the one over
\(\sfQ\), we need to show that an element of \(\sfH_\sfQ\) preserving
\(\sfT\) actually belongs to \(\sfH\).
\begin{Lem} \label{lem:sc-denominator-clearing}
Let
$h=\sum_{w\in {\frak S}_r}wf_w\in\sfH_\sfQ,
\,f_w\in \sfF,$
and assume that $\sfT h\subseteq \sfT$.  Then
$f_w\in \sfP$ for every $w\in {\frak S}_r$ and hence $h\in\sfH$.
\end{Lem}

\begin{proof}
Assume that \(f_u\notin\sfP\) for some \(u\in{\frak S}_r\).  Since $\mathpzc K$ is a unique factorization domain,
\(
\sfP=\mathpzc K[x_1^{\pm1},\ldots,x_r^{\pm1}]
\)
is again a unique factorization domain.  Hence there is an irreducible element $d\in \sfP$ such that $v_d(f_u)<0$.  Here, for
$0\ne f=a/b\in \sfF$ with $a,b\in\sfP$, we set
$v_d(f)=v_d(a)-v_d(b),$
where \(v_d(c)\) is the multiplicity of \(d\) in the irreducible
factorization of \(0\neq c\in\sfP\), and we put $v_d(0)=+\infty$.

Since
\(v_d(f_u)<0\),  we have
$m=-\min_{w\in {\frak S}_r}v_d(f_w)>0.$
Let
$g_w=d^mf_w.$
Then $g_w\in \sfP_{(d)}$ for all $w$, where
$\sfP_{(d)}=(\sfP\setminus(d))^{-1}\sfP,$
and there is some $w_0\in {\frak S}_r$ with $v_d(g_{w_0})=0$.  Hence in the residue field
$L_d:=\sfP_{(d)}/\sfP_{(d)}d
\cong
\operatorname{Frac}(\sfP/(d))$
we have
\begin{equation}\label{gw0 not equal 0}
 \overline{g_{w_0}}\not=0.
\end{equation}
For $w\in {\frak S}_r$, define
\begin{equation}\label{phiw}
\phi_w:\sfP\longrightarrow L_d,
\qquad
\phi_w(f)=\overline{w^{-1}(f)}.
\end{equation}

Fix
$\mathbf i\in I(n,r)$.
Since $h=\sum_{w\in {\frak S}_r}wf_w$ and $\sfT h\subseteq \sfT$, by
\eqref{eq:sc-right-action-formula}, we have
\begin{equation}\label{eq:sc-localized-relation-before-mod-d}
\sum_{w\in {\frak S}_r}
\omega_{\mathbf i w}w^{-1}(f)g_w=((\omega_{\mathbf i}f)h)d^m \in \sfT d^m
\end{equation}
for $f\in \sfP$.
Let
$
\sfT_{(d)}
=\sfT\otimes_{\sfP}\sfP_{(d)}
=
\oplus_{\mathbf j\in I(n,r)}(\omega_{\mathbf j}\otimes 1)\sfP_{(d)}.
$
Since $\sfT$ is a free right $\sfP$-module, the natural map
$\sfT\to \sfT_{(d)}$ is injective, and we identify $\sfT$ with its image. Thus we regard
\(\omega_{\mathbf j}\) as the element
\(\omega_{\mathbf j}\otimes1\in\sfT_{(d)}\). We denote its image in
\(\sfT_{(d)}/\sfT_{(d)}d\) by
\(\bar{\omega}_{\mathbf j}\). Since
$m>0$, we have
$
\sfT d^m\subseteq\sfT_{(d)}d.$
Hence by \eqref{eq:sc-localized-relation-before-mod-d}, in the quotient
$\sfT_{(d)}/\sfT_{(d)}d
\cong
\oplus_{\mathbf j\in I(n,r)}\bar \omega_{\mathbf j}L_d$
we have
\begin{equation*}
\sum_{\mathbf j\in I(n,r)}
\bar \omega_{\mathbf j}\sum_{w\in {\frak S}_r,\,\mathbf i w=\mathbf j}
\phi_w(f)\,
\overline{g_w}
=\sum_{w\in {\frak S}_r}
\bar \omega_{\mathbf i w}\,
\overline{w^{-1}(f)}\,
\overline{g_w}
=0
\end{equation*}
for $f\in \sfP$, where $\phi_w$ is defined in \eqref{phiw}.
It follows that
\begin{equation}\label{eq:sc-mod-d-relation}
\sum_{w\in {\frak S}_r,\,\mathbf i w=\mathbf j}
\phi_w(f)\,
\overline{g_w}
=0
\end{equation}
for $f\in \sfP$ and $\mathbf j\in I(n,r)$.

We distinguish two cases.

\smallskip
\noindent\emph{Case 1: $d$ is not associated with any $x_i-x_j$ for
$i\ne j$.}
If
$\phi_{w_1}=\phi_{w_2}$, then for every $1\leq k\leq r$,
$x_{w_1^{-1}(k)}-x_{w_2^{-1}(k)}\in\sfP\cap\sfP_{(d)}d=(d),$
since $(d)$ is a prime ideal of $\sfP$.
If $w_1\ne w_2$, one can choose $k$ with
$w_1^{-1}(k)\ne w_2^{-1}(k)$; then  $d|x_{w_1^{-1}(k)}-x_{w_2^{-1}(k)}$ in $\sfP$.  It follows that $d$ is associated with $x_{w_1^{-1}(k)}-x_{w_2^{-1}(k)}$, contrary to the present assumption. Therefore we have $w_1=w_2$. Fix
$\mathbf i\in I(n,r)$.
By
\eqref{eq:sc-mod-d-relation} we have
\begin{equation}\label{eq:sc-case1-hom-relation}
\sum_{\substack{w\in {\frak S}_r\\
\mathbf i w=\mathbf iw_0}}
\phi_w(f)\,\overline{g_w}=0
\end{equation}
for $f\in \sfP$.
By Lemma~\ref{lem:sc-hom-independence}, $\overline{g_w}=0$ for all $w\in {\frak S}_r$ with $\mathbf i w=\mathbf iw_0$.  In particular we have
$\overline{g_{w_0}}=0$, contradicting \eqref{gw0 not equal 0}.

\smallskip
\noindent\emph{Case 2: $d$ is associated with a variable difference.}
Replacing $d$ by an associate, we may assume
$d=x_i-x_j
\ (i\ne j).$   Let
$\pi_d:\sfP\longrightarrow \sfP/(d)$
be the quotient map and let
$\iota:\sfP/(d)\hookrightarrow L_d$ be the natural embedding.  Then
$
\phi_w=\iota\circ\pi_d\circ w^{-1}$, where $\phi_w$ is defined in \eqref{phiw}.
If $\phi_{w_1}=\phi_{w_2}$, then we have
$\pi_d=\pi_d\circ {w_2}^{-1}w_1.$
Hence
$({w_2}^{-1}w_1)(f)-f\in (x_i-x_j)$ for all $f\in\sfP.$
It follows that
\[
({w_2}^{-1}w_1)(x_k)-x_k=x_{{w_2}^{-1}w_1(k)}-x_k\in (x_i-x_j)
\]
for $1\leq k\leq r$.
Therefore,
${w_2}^{-1}w_1(k)=k $ for \(k\notin\{i,j\}\).  Hence
${w_2}^{-1}w_1=1$ or ${w_2}^{-1}w_1=(i,j).$
Consequently,
\begin{equation}\label{w1=w2}
 w_2=w_1
 \text{ or } w_2=w_1(i,j).
 \end{equation}
Since $n\geq2$, we may choose
$
\mathbf b=(b_1,\ldots,b_r)\in I(n,r)$
with
$
b_i=1,
\
b_j=2.$
Then $\mathbf b(i,j)\ne\mathbf b$.  Define
$\mathbf i=\mathbf b w_0^{-1}$.
If $\mathbf i w=\mathbf b$, then
$
\mathbf i(w (i,j))
=(\mathbf i w)(i,j)
=\mathbf b(i,j)
\ne\mathbf b.$ Therefore by \eqref{w1=w2}, the homomorphisms
$
\{\phi_w\mid \mathbf i w=\mathbf b\}$
are pairwise distinct.  By \eqref{eq:sc-mod-d-relation} we have
\[
\sum_{\substack{w\in {\frak S}_r\\
\mathbf i w=\mathbf b}}
\phi_w(f)\,\overline{g_w}=0
\qquad (f\in \sfP).
\]
Hence, by Lemma~\ref{lem:sc-hom-independence} we have
$\overline{g_{w_0}}=0$, a contradiction.

Thus both cases lead to a contradiction, and hence
\(f_w\in\sfP\) for every \(w\in{\frak S}_r\).
\end{proof}

We are now ready to prove the double centralizer property for all
\(n\geq2\).
\begin{Thm}\label{thm:second-centralizer-nge2}
Let $\mathpzc K$ be a commutative unique factorization domain, $n\geq 2$, and $r\geq 1$.  Then the natural algebra homomorphism
$
\rho_{r,\mathpzc K}:
\mathpzc K\widehat{\frak S}_r^{\,\mathrm{op}}
\rightarrow
\operatorname{End}_{\widehat{\mathcal S}_{\mathpzc K}(n,r)}
(\Omega_{\mathpzc K}^{\otimes r})
$
is an isomorphism.
\end{Thm}
\begin{proof}
Let $\theta\in\operatorname{End}_{\sfS}(\sfT)=\operatorname{End}_{\widehat{\mathcal S}_{\mathpzc K}(n,r)}
(\Omega_{\mathpzc K}^{\otimes r}).$
For \(z\in\sfZ\), define
$R_z:\sfT\longrightarrow\sfT$ by $R_z(m)=mz$ for $m\in\sfT$.
By Lemma~\ref{center}, we have
\(R_z\in\operatorname{End}_\sfH(\sfT)=\sfS\) for  \(z\in\sfZ\). Hence, since
\(\theta\in\operatorname{End}_{\sfS}(\sfT)\),
$\theta(mz)=\theta(m)z$
for $m\in\sfT$ and $z\in\sfZ$.
Therefore \(\theta\) induces a
\(\sfQ\)-linear map
$\theta_\sfQ:\sfT_\sfQ\rightarrow\sfT_\sfQ$
such that $\theta_\sfQ(m\otimes a)=\theta(m)\otimes a$ for $m\in\sfT$ and $a\in\sfQ$.
Since \(\theta\in\operatorname{End}_{\sfS}(\sfT)\), by Lemma~\ref{lem:sc-base-change} we have
$\theta_\sfQ
\in
\operatorname{End}_{\sfS_\sfQ}(\sfT_\sfQ).$
By Lemma~\ref{lem:sc-generic-double-centralizer}, there exists a unique
element \(h\in\sfH_\sfQ\) such that
\begin{equation}\label{eq:sc-theta-right-h}
\theta_\sfQ(m)=mh
\qquad
(m\in\sfT_\sfQ).
\end{equation}
By Lemma \ref{lem:sc-injective-extension} we may regard \(\sfT\) as a submodule of \(\sfT_\sfQ\). Then, for \(m\in\sfT\),
$mh=\theta_\sfQ(m)=\theta(m)\in\sfT.$
Hence
$\sfT h\subseteq\sfT.$ Therefore by
Lemma~\ref{lem:sc-denominator-clearing}, we have \(h\in\sfH\). So \(\rho_{r,\mathpzc K}\) is surjective.

 Let \(a\in\sfH\) and suppose that $\sfT a=0.$
Then
\(a\otimes1\in\sfH_\sfQ\) acts trivially on \(\sfT_\sfQ\). By
Lemma~\ref{lem:sc-generic-double-centralizer}, we have
$
a\otimes1=0$ in  $\sfH_\sfQ.$  It follows from Lemma \ref{lem:sc-injective-extension} that \(a=0\). Therefore
\(\rho_{r,\mathpzc K}\) is injective, and the proof is complete.
\end{proof}

Combining Theorems \ref{surjective1} and \ref{thm:second-centralizer-nge2} we obtain the following result.
\begin{Thm}\label{cor:second-centralizer-loop-semigroup-nge}
Let $\mathpzc K$ be a commutative unique factorization domain, $n\geq 2$, and $r\geq 1$.
Assume that there is a subset $\mpX$ of $\field$ such that $|\mpX|>r$ and $a-b$ is invertible in $\field$ for any $a\not=b\in\mpX$. Then the natural algebra homomorphism
$
\rho_{r,\mathpzc K}:
\mathpzc K\widehat{\frak S}_r^{\,\mathrm{op}}
\rightarrow
\operatorname{End}_{\widehat G_{\mathpzc K}(n)}
(\Omega_{\mathpzc K}^{\otimes r})$
is an isomorphism. In particular, the conclusion holds whenever $\mathpzc K$ is a
field with $|\mathpzc K|>r$.
\end{Thm}

\subsection{A rank-one counterexample}
We now consider the case \(n=1\), which behaves differently.
We first prove directly, over an arbitrary non-zero commutative ring,
that \(\rho_{2,\mathpzc K}\) is not surjective, and then give a uniform
proof over fields for all \(r\geq2\).

Assume now that $n=1$ and $r\geq 2$.
Since
\(
\Omega_{\mathpzc K}
=
\mathpzc K[\ttt,\ttt^{-1}],
\) and
\(\omega_{1+l}=\ttt^{-l}\omega_1,
\)
put $x=\ttt^{-1}$ and identify
\[
\Omega_{\mathpzc K}^{\otimes r}
\cong
\sfP:=
\mathpzc K[x_1^{\pm1},\ldots,x_r^{\pm1}]
\]
via
\(
\omega_{1+l_1}\otimes\cdots\otimes\omega_{1+l_r}
\longmapsto
x_1^{l_1}\cdots x_r^{l_r}.
\)

For $w\in\widehat{\frak S}_r$, let
\[
R_w:\Omega_{\mathpzc K}^{\otimes r}\longrightarrow
\Omega_{\mathpzc K}^{\otimes r},
\qquad
R_w(u)=u\cdot w\quad\text{for $u\in\Omega_{\mathpzc K}^{\otimes r}$}.
\] Via the above identification, we regard $R_w$ as an endomorphism of $\sfP$.

For $a\in \sfP$, denote by $M_a$ the multiplication operator
\[
M_a:\sfP\longrightarrow \sfP,\qquad f\longmapsto af.
\]
Then the generator $\bse_k$ of $\affSr$ acts as
multiplication by $x_k$:
\[
R_{\bse_k}=M_{x_k},
\qquad
R_{\bse_k^{-1}}=M_{x_k^{-1}}.
\]
Moreover, for $\sigma\in\fS_r$, the right action is given by
\[
(f\cdot\sigma)(x_1,\ldots,x_r)
=
f(x_{\sigma^{-1}(1)},\ldots,x_{\sigma^{-1}(r)}).
\]
This right action is related to the left action used in
Section~5.2 by
\[
f\cdot\sigma=\sigma^{-1}(f).
\]
Let
$
\sfP^{\fS_r}
=
\{a\in \sfP\mid a\cdot\sigma=a
\text{ for every }\sigma\in\fS_r\}.$

\begin{Lem}\label{lem:rank-one-affine-schur}
We have
\(
\widehat{\mathcal S}_{\mathpzc K}(1,r)
=
\operatorname{End}_{\widehat{\frak S}_r}(\sfP)
=
\{M_a\mid a\in \sfP^{\frak S_r}\}.
\)
\end{Lem}

\begin{proof}
Let
\(
T\in\operatorname{End}_{\widehat{\frak S}_r}(\sfP).
\)
Since $T$ commutes with $R_{\bse_k}=M_{x_k}$ and
$R_{\bse_k^{-1}}=M_{x_k^{-1}}$, we have
\[
T(x_k^{\pm1}f)=x_k^{\pm1}T(f)
\]
for $f\in \sfP$.
It follows that
\(
T(gf)=gT(f)
\)
for \(g,f\in \sfP\).
Hence \[
T(g)=gT(1)=ga
\qquad(g\in \sfP)
\]
where
\(
a=T(1).
\)
Therefore we have
\(
T=M_a.
\)

We now determine when $M_a$ commutes with $R_\sigma$ for $\sg\in\fS_r$. For
$a,f\in \sfP$ and $\sigma\in\fS_r$, we have
\[
M_aR_\sigma(f)
=
a(f\cdot\sigma),\quad
R_\sigma M_a(f)
 =
(af)\cdot\sigma =
(a\cdot\sigma)(f\cdot\sigma).
\]
Therefore, if $a\cdot\sigma=a$, then
\(
R_\sigma M_a(f)
=
a(f\cdot\sigma)
=
M_aR_\sigma(f)
\)
for every $f\in \sfP$, and hence
\(
R_\sigma M_a=M_aR_\sigma.
\)
Conversely, suppose that
\(
R_\sigma M_a=M_aR_\sigma.
\) Then we have
\[
a\cdot\sigma
=
R_\sigma M_a(1)
=
M_aR_\sigma(1)
=
a.
\]
Thus
\[
M_aR_\sigma=R_\sigma M_a
\quad\Longleftrightarrow\quad
a\cdot\sigma=a.
\]
Consequently, $M_a$ commutes with $R_\sigma$ for every
$\sigma\in\fS_r$ if and only if
\(
a\in \sfP^{\fS_r}.
\)
This proves the assertion.
\end{proof}

Let $s_1=(1,\,2)\in\fS_r$. Define
\[
\partial_1:\sfP\longrightarrow \sfP
\]
by
\(
\partial_1(f)
=
\frac{f-f\cdot s_1}{x_1-x_2}.
\)
Since \(f-f\cdot s_1\) is divisible by \(x_1-x_2\), the map $\partial_1$ is a well-defined $\mathpzc K$-linear endomorphism of
$\sfP$.

\begin{Lem}\label{lem:divided-difference-centralizer}
We have
\(
\partial_1
\in
\operatorname{End}_{\widehat{\mathcal S}_{\mathpzc K}(1,r)}(\sfP).
\)
\end{Lem}

\begin{proof}
For $a\in \sfP^{\fS_r}$, we have
\[
\begin{aligned}
\partial_1(af)
&=
\frac{af-(af)\cdot s_1}{x_1-x_2} =
\frac{af-(a\cdot s_1)(f\cdot s_1)}{x_1-x_2}\ =
a\frac{f-f\cdot s_1}{x_1-x_2} =
a\partial_1(f)
\end{aligned}
\]
for $f\in \sfP$.
It follows that
\(
\partial_1M_a=M_a\partial_1.
\)
Therefore by Lemma~\ref{lem:rank-one-affine-schur} we conclude that \(
\partial_1
\in
\operatorname{End}_{\widehat{\mathcal S}_{\mathpzc K}(1,r)}(\sfP).
\)
\end{proof}

We first consider the smallest case, where the failure of surjectivity can be seen directly.

\begin{example}\label{ex:rank-one-r2}
Assume $n=1$ and $r=2$ and $\mpk$ is a non-zero  commutative ring
with unity. Then
\(
\sfP=\mathpzc K[x_1^{\pm1},x_2^{\pm1}]
\)
and
\(
\partial_1
\notin
\rho_{2,\mathpzc K}
\bigl(
\mathpzc K\widehat{\fS}_2^{\,\mathrm{op}}
\bigr).
\) Hence \(
\rho_{2,\mathpzc K}:
\mathpzc K\widehat{\frak S}_2^{\,\mathrm{op}}
\rightarrow
\operatorname{End}_{\widehat{\mathcal S}_{\mathpzc K}(1,2)}
(\Omega_{\mathpzc K}^{\otimes 2})
\)
is not surjective.

Indeed, every element in the image of
$\mathpzc K\widehat{\fS}_2^{\,\mathrm{op}}$
can be written in the form
\(
M_p+M_qR_{s_1},
\
p,q\in \sfP.
\)
Suppose that
\(
\partial_1=M_p+M_qR_{s_1}.
\)
Then we have
\[
\begin{split}
0=\partial_1(1)&=(M_p+M_qR_{s_1})(1)=p+q,\\
1=\partial_1(x_1)&=(M_p+M_qR_{s_1})(x_1)
=px_1+qx_2.
\end{split}
\]
Thus
\(
q=-p
\) and  \(1=p(x_1-x_2) \),
which is impossible since
\(
x_1-x_2
\) is not a unit in $\sfP$.
\end{example}

To treat arbitrary $r\geq2$, we need the following immediate
consequence of Lemma~\ref{lem:sc-hom-independence}.

\begin{Lem}\label{lem:artin-independence-permutations}
Let
\(
\sfF=\mathpzc K(x_1,\ldots,x_r)
\). Assume that $\mathpzc K$ is a field, that
\(
\sigma_1,\ldots,\sigma_m\in\fS_r
\)
are distinct, and that
\(
a_1,\ldots,a_m\in \sfF.
\)
If
\[
\sum_{j=1}^m
a_j(f\cdot\sigma_j)=0
\qquad
\text{for every }f\in \sfP,
\]
then
\(
a_1=\cdots=a_m=0.
\)
\end{Lem}

\begin{proof}
This follows immediately from Lemma~\ref{lem:sc-hom-independence},
since the maps
$\sfP\rightarrow\sfF,\
f\mapsto f\cdot\sigma_j$
are pairwise distinct unital ring homomorphisms.
\end{proof}

We can now state the rank-one counterexample.

\begin{Thm}\label{thm:rank-one-counterexample}
Let $\mathpzc K$ be a field and $r\geq2$.  If $n=1$, then
$
\rho_{r,\mathpzc K}:
\mathpzc K\widehat{\frak S}_r^{\,\mathrm{op}}
\rightarrow
\operatorname{End}_{\widehat{\mathcal S}_{\mathpzc K}(1,r)}
(\Omega_{\mathpzc K}^{\otimes r})
$
is not surjective. Hence the natural
homomorphism
$
\mathpzc K\widehat{\frak S}_r^{\,\mathrm{op}}
\rightarrow
\operatorname{End}_{\widehat G_{\mathpzc K}(1)}
(\Omega_{\mathpzc K}^{\otimes r})
$
is not surjective.
\end{Thm}

\begin{proof}
By Lemma~\ref{lem:divided-difference-centralizer},
\(
\partial_1
\in
\operatorname{End}_{\widehat{\mathcal S}_{\mathpzc K}(1,r)}(\sfP)\han \operatorname{End}_{\widehat G_{\mathpzc K}(1)}
(\Omega_{\mathpzc K}^{\otimes r}).
\) It is enough to show that \(
\partial_1
\notin
\rho_{r,\mathpzc K}
\bigl(
\mathpzc K\widehat{\frak S}_r^{\,\mathrm{op}}
\bigr).
\)
Suppose, to the contrary, that $\partial_1$ belongs to the image of
$\rho_{r,\mathpzc K}$.  Then  there exist
$p_\sigma\in \sfP$ such that
\[
\partial_1
=
\sum_{\sigma\in\mathcal S_r}
M_{p_\sigma}R_\sigma.
\]
Hence, for every $f\in \sfP$,
\[
\frac{f-f\cdot s_1}{x_1-x_2}
=
\sum_{\sigma\in\fS_r}
p_\sigma(f\cdot\sigma).
\]
Regarding this equality as an equality in $\sfF$, we obtain
\[
0={}
\left(
p_1-\frac1{x_1-x_2}
\right)f +
\left(
p_{s_1}+\frac1{x_1-x_2}
\right)(f\cdot s_1) +
\sum_{ {\sigma\in\fS_r,\,
\sigma\neq1,s_1}}
p_\sigma(f\cdot\sigma).
\]
By
Lemma~\ref{lem:artin-independence-permutations} we have
\(
p_1=\frac1{x_1-x_2}.
\)
But
\(
p_1\in \sfP,
\
\frac1{x_1-x_2}\notin \sfP,
\)
a contradiction. Therefore $\partial_1$ is not in the image of
$\rho_{r,\mathpzc K}$.
\end{proof}

\subsection{Centers of affine Schur algebras}
We conclude this section with an application of the double centralizer
property to the center of the affine Schur algebra.
  Recall the notation
$
\sfP=\mathpzc K[x_1^{\pm1},\ldots,x_r^{\pm1}]$
from \S5.2.   For $h\in\mathpzc K\affSr$, let
\[
R_h:\Ogk^{\ot r}\longrightarrow\Ogk^{\ot r},
\qquad
w\longmapsto wh.
\]
Throughout this subsection, $Z(A)$ denotes the center of an algebra $A$.
\begin{Lem}\label{lem:center-affine-schur}
Let $\mathpzc K$ be a commutative unique factorization domain,
$n\geq2$, and $r\geq1$. Then the map
\[
Z(\mathpzc K\affSr)
\longrightarrow
Z(\afSrk),
\qquad
z\longmapsto R_z,
\]
is an algebra isomorphism. Consequently,
$
Z(\afSrk)
\cong
\sfP^{\fS_r}
=
\mathpzc K[\sigma_1,\ldots,\sigma_{r-1},\sigma_r^{\pm1}]$,
where $
\sigma_k=\sigma_k(x_1,\ldots,x_r)
=
\sum_{1\leq i_1<\cdots<i_k\leq r}
x_{i_1}\cdots x_{i_k}$
is the $k$-th elementary symmetric polynomial.
\end{Lem}

\begin{proof}
Clearly, for  $z\in Z(\mathpzc K\affSr)$ we have $R_z\in Z(\afSrk)$.
Conversely, let $c\in Z(\afSrk)$. Then by
Theorem~\ref{thm:second-centralizer-nge2}, there is a unique
$h\in\mathpzc K\affSr$ such that
$
c=R_h.$
Since $c\in\afSrk$, we have $R_hR_x=R_xR_h$ for $x\in\mathpzc K\affSr$.
By
Theorem~\ref{thm:second-centralizer-nge2} the right action of $\mathpzc K\affSr$ on $\Ogk^{\ot r}$ is faithful. Therefore
$xh=hx$ for every $x\in\mathpzc K\affSr$, and hence
$h\in Z(\mathpzc K\affSr)$. Thus
$Z(\mathpzc K\affSr)\cong Z(\afSrk).$ The final assertion follows from Lemma~\ref{center} and the
fundamental theorem of symmetric polynomials.
\end{proof}
We next relate this center to the center of the hyper loop algebra.
For $m\neq 0$, let
$
\sfz_{ m}
=
\sum_{i=1}^n g_{i, m}
\in\afuglq
$
where $g_{i, m}$ is defined in
 \eqref{gil}.
Define $C_k^\pm\in\afuglq$ by
\begin{equation}\label{Cpm generating function}
C^\pm(u)
:=
\prod_{i=1}^n\La_i^\pm(u)
=
\sum_{k\geq0}C_k^\pm u^k,
\end{equation}
where $\La_i^\pm(u)$ is defined in \eqref{Laik generating function}. By definition we have
\[
C_k^\pm
=
\sum_{\substack{k_1+\cdots+k_n=k\\
k_i\geq0}}
\La_{1,\pm k_1}\cdots\La_{n,\pm k_n} \in\afuglz.
\]
The image of $C_k^\pm$ in $\afuglk$ will also be denoted by $C_k^\pm$.

\begin{Lem}\label{lem:center-elements-hyper-loop}
For every $k\geq0$, the image of $C_k^\pm$ in $\afuglk$
belongs to $Z(\afuglk)$. Moreover,
\[
\afxrk(C_k^\pm)
=
\begin{cases}
(-1)^k
R_{\sigma_k(x_1^{\mp 1},\ldots,x_r^{\mp 1})},
&0\leq k\leq r,\\
0,&k>r,
\end{cases}
\]
where $\afxrk$ is defined in \eqref{afxrk}.
\end{Lem}

\begin{proof}
By \eqref{Laik generating function},
$
C^\pm(u)
=
\exp (
-\sum_{m\geq1}\frac{\sfz_{\pm m}}{m}u^m
 ).
$
Since each $\sfz_{\pm m}$ is central in $\afuglq$, so is
$C_k^\pm$. As $C_k^\pm\in\afuglz$, we have
$
C_k^\pm\in Z(\afuglz).$
Hence its image in $\afuglk$ belongs to $Z(\afuglk)$.

Under the identification
$\mathpzc K\mathbb Z^r\cong\sfP$
used in \eqref{sfH iso P rtimes Sr}, the element $x_a$
corresponds to $\bse_a$.  It is easy to see that
\[
\afxr(\sfz_m)
=
R_{p_m(x_1^{-1},\ldots,x_r^{-1})},
\qquad
\afxr(\sfz_{-m})
=
R_{p_m(x_1,\ldots,x_r)},
\]
where $p_m(y_1,\ldots,y_r)
=
\sum_{a=1}^r y_a^m$.
Therefore, in
$\operatorname{End}_{\mbq}(\Omega_{\mbq}^{\ot r})[[u]]$,
\[
\begin{aligned}
\sum_{k\geq0}\afxr(C_k^\pm)u^k
=
R_{\exp\left(
-\sum_{m\geq1}
\frac{p_m(x_1^{\mp 1},\ldots,x_r^{\mp 1})}{m}u^m
\right)} =
R_{\prod_{a=1}^r(1-x_a^{\mp 1}u)}
 =
\sum_{k=0}^r
(-1)^k
R_{\sigma_k(x_1^{\mp 1},\ldots,x_r^{\mp 1})}u^k.
\end{aligned}
\]
Comparing coefficients proves the asserted formulas over $\mbq$.
Since both
sides are defined over $\mbz$, the formulas hold over $\mbz$, and
hence over $\mathpzc K$ by base change.
\end{proof}

We can now determine the image of the center of the hyper loop
algebra.

\begin{Thm}\label{thm:center-affine-schur}
Let $\mathpzc K$ be a commutative unique factorization domain,
$n\geq2$, and $r\geq1$. Then
$
\afxrk\bigl(Z(\afuglk)\bigr)
=
Z(\afSrk).$
\end{Thm}

\begin{proof}
Since
$\afxrk:\afuglk\rightarrow\afSrk$
is surjective, we have
$
\afxrk\bigl(Z(\afuglk)\bigr)
\subseteq
Z(\afSrk).$
On the other hand, by Lemma~\ref{lem:center-affine-schur}, the algebra
$Z(\afSrk)$ is generated by
$R_{\sigma_1},\ldots,R_{\sigma_r},
R_{\sigma_r^{-1}}.$
Furthermore, by Lemma~\ref{lem:center-elements-hyper-loop},
$R_{\sigma_k}\in \afxrk\bigl(Z(\afuglk)\bigr)$
for $1\leq k\leq r$, and  $
R_{\sigma_r^{-1}}=R_{\sigma_r(x_1^{-1},\ldots,x_r^{-1})}
\in \afxrk\bigl(Z(\afuglk)\bigr).
$
Therefore
$
Z(\afSrk)
\subseteq
\afxrk\bigl(Z(\afuglk)\bigr).$ Consequently, $
Z(\afSrk)
=
\afxrk\bigl(Z(\afuglk)\bigr).$
\end{proof}
We next obtain the corresponding result for the loop semigroup algebra.
\begin{Thm}\label{thm:center-loop-semigroup}
Let $\mathpzc K$ be a commutative unique factorization domain,
$n\geq2$, and $r\geq1$. Suppose that there is a subset
$\mpX$ of $\mathpzc K$ such that $|\mpX|>r$ and $a-b$ is
invertible in $\mathpzc K$ for any distinct $a,b\in\mpX$.
Then
$\widehat\zeta_{r,\mathpzc K}
\Bigl(
Z\bigl(\mathpzc K\widehat G_{\mathpzc K}(n)\bigr)
\Bigr)
=
Z\bigl(\widehat S_{\mathpzc K}(n,r)\bigr)$, where $\afzrk$ is defined in \eqref{afzrk}.
In particular, the conclusion holds whenever $\mathpzc K$ is
a field with $|\mathpzc K|>r$.
\end{Thm}

\begin{proof}

By Theorem~\ref{surjective1}, we have
$
\widehat\zeta_{r,\mathpzc K}
\Bigl(
Z\bigl(\mathpzc K\widehat G_{\mathpzc K}(n)\bigr)
\Bigr)
\subseteq
Z\bigl(\widehat{\mathcal S}_{\mathpzc K}(n,r)\bigr).
$
On the other hand, for $a\in\mathpzc K$, set
$
g_a=(\ttt^{-1}+a)I,$
where $I\in M_n(\mathpzc K[\vep,\vep^{-1}])$ is the identity matrix. Then we have
$g_a\in
Z\bigl(\mathpzc K\widehat G_{\mathpzc K}(n)\bigr).$
Furthermore we have
$$
\widehat\zeta_{r,\mathpzc K}(g_a)
=
R_{\prod_{s=1}^r(x_s+a)}=R_{\sum_{k=0}^r a^{r-k}\sigma_k}=
\sum_{k=0}^r
a^{r-k}R_{\sigma_k},
$$
where $\sigma_0=1$.
By hypothesis, we can choose distinct elements  $a_0,a_1,\ldots,a_r\in\mpX\han\field$ such that
$a_s-a_t$ is invertible in $\field$ for $0\leq s\not=t\leq r$.
Then
\[
(
\widehat\zeta_{r,\mathpzc K}(g_{a_0}),
\widehat\zeta_{r,\mathpzc K}(g_{a_1}),
\cdots,
\widehat\zeta_{r,\mathpzc K}(g_{a_r})
)=(R_{\sigma_r},
R_{\sigma_{r-1}},
\cdots,
R_{\sigma_0})B,
\]
where $B$ is the Vandermonde matrix defined in \eqref{Vandermonde matrix}. Since
$\det(B)=\prod_{0\leq i<j\leq r}(a_j-a_i)$
is invertible, we have
$
R_{\sigma_k}\in
\widehat\zeta_{r,\mathpzc K}
(
Z(\mathpzc K\widehat G_{\mathpzc K}(n))
)$
for $0\leq k\leq r.$
Furthermore, we have
$R_{\sigma_r^{-1}}=R_{x_1^{-1}\cdots x_r^{-1}}=
\widehat\zeta_{r,\mathpzc K}(\ttt I)\in
\widehat\zeta_{r,\mathpzc K}
(Z(\mathpzc K\widehat G_{\mathpzc K}(n))).
$
Therefore
by Lemma~\ref{lem:center-affine-schur},
$
Z\bigl(\widehat{\mathcal S}_{\mathpzc K}(n,r)\bigr)
\subseteq
\widehat\zeta_{r,\mathpzc K}
\Bigl(
Z\bigl(\mathpzc K\widehat G_{\mathpzc K}(n)\bigr)
\Bigr),$
and the assertion follows.
\end{proof}

\section{ Evaluation homomorphisms and their compatibility}
In this section, we investigate the evaluation maps $\Ev_a$ and $\ti\ev_a$, which will be used in the representation theory of $\afuglk$ and the affine Schur algebra $\afSrk$ in subsequent sections. These maps allow us to specialize the affine algebras to their finite counterparts by evaluating the loop parameter at specific values, and thereby relate the affine and finite representation theories.

\subsection{The evaluation maps $\Ev_a:\afuglk\ra\uglk$}
Let $\uglq$ be the enveloping algebra of $\frak{gl}_n(\mbq)$.
There is a natural algebra homomorphism
$$\Ev:\afuglq\ra\uglq\ot\mbq[\ttt,\ttt^{-1}]$$
defined by $\Ev(\afE_{i,j+ln})= E_{i,j}\ot\ttt^l$ for $1\leq i,j\leq n$ and $l\in\mbz$,
where $E_{i,j}\in{\frak{gl}_n}(\mbq)$ denotes the elementary matrix  $(\dt_{i,k}\dt_{j,l})_{1\leq k,l\leq n}$.
Let $\uglz$ be the $\mbz$-subalgebra of $\uglq$ generated by the elements $(E_{i,j})^{(k)}$ and $\big({H_i\atop t}\big)$ for $1\leq i\not=j\leq n$, $k,t\in\mbn$, where $H_i=E_{i,i}$.
We begin with some lemmas.
\begin{Lem}\label{Ev(E i,j+ln)}
For $1\leq i\not=j\leq n$ and $l\in\mbz$ we have $\Ev((\afE_{i,j+ln})^{(k)})=E_{i,j}^{(k)}\ot\ttt^{lk}$.
\end{Lem}
\begin{Lem}\label{lem for {Ev(La{i,k,l})}}
For $1\leq i\leq n$ and $k\geq 1$ the following formula holds in $\uglq:$
 $$k\bigg({H_i\atop k}\bigg)
=\sum_{1\leq s\leq k}(-1)^{s+1}H_i\bigg({H_i\atop k-s}\bigg).$$
\end{Lem}
\begin{proof}
We proceed  by induction on $k$. The case where $k=1$ is trivial.
For $k\geq 2$, by the induction hypothesis we have
\begin{equation*}
\begin{split}
\sum_{1\leq s\leq k}(-1)^{s+1}H_i\bigg({H_i\atop k-s}\bigg)
&=H_i\bigg({H_i\atop k-1}\bigg)+\sum_{1\leq t\leq k-1}(-1)^tH_i\bigg({H_i\atop k-1-t}\bigg)\\
&=
H_i\bigg({H_i\atop k-1}\bigg)-(k-1)\bigg({H_i\atop k-1}\bigg)\\
&=k\bigg({H_i\atop k}\bigg).
\end{split}
\end{equation*}
This completes the inductive step and the proof.
\end{proof}

\begin{Lem}\label{Ev(La{i,k,l})}
For $1\leq i\leq n$, $k\in\mbn$ and $l>0$ we have $\Ev(\La_{i,\pm k,l})=(-1)^k\big({H_i\atop k}\big)\ot\ttt^{\pm kl}$.
\end{Lem}
\begin{proof}
We argue by induction on $k$. The cases $k=0,1$ are straightforward.
Assume now that $k\geq 2$. By Lemma \ref{Laikl} and the induction hypothesis we have
\begin{equation*}
\begin{split}
\Ev(\La_{i,\pm k,l})&=-\frac{1}{k}\sum_{0\leq s\leq k-1}\Ev(g_{i,\pm(k-s)l}\La_{i,\pm s,l})\\
&=-\frac{1}{k}\sum_{0\leq s\leq k-1}(H_i\ot\ttt^{\pm(k-s)l})
\cdot ((-1)^s\bigg({H_i\atop s}\bigg)\ot\ttt^{\pm sl})\\
&=-\frac{1}{k}\sum_{0\leq s\leq k-1} (-1)^s H_i\bigg({H_i\atop s}\bigg)\ot\ttt^{\pm kl}.
\end{split}
\end{equation*}
Applying  Lemma \ref{lem for {Ev(La{i,k,l})}} yields the desired result:  $\Ev(\La_{i,\pm k,l})=(-1)^k\big({H_i\atop k}\big)\ot\ttt^{\pm kl}$.
\end{proof}

Combining Lemmas \ref{Ev(E i,j+ln)} and \ref{Ev(La{i,k,l})} we obtain the following result.
\begin{Lem}\label{Ev(afuglz)}
We have $\Ev(\afuglz)\subseteq\uglz\ot\mbz[\ttt,\ttt^{-1}]$.
\end{Lem}
Recall from \eqref{Unk} that $\afuglk=\afuglz\ot_\mbz\field.$ Let $\uglk=\uglz\ot\field$.
By restricting the map $\Ev:\afuglq\ra\uglq\ot\mbq[\ttt,\ttt^{-1}]$ to $\afuglz$ and applying  Lemma \ref{Ev(afuglz)} we obtain a map $\Ev:\afuglz\ra\uglz\ot\mbz[\ttt,\ttt^{-1}]$.
Base change then induces an algebra homomorphism:
$$\Ev_\field:=\Ev\ot 1:\afuglk\ra\uglk\ot\field[\ttt,\ttt^{-1}].$$
Let $\field^*$ be the set of invertible elements in $\field$. For $a\in\field^*$, define the algebra homomorphism $\pi_a:\uglk\ot\field[\ttt,\ttt^{-1}]\ra\uglk$ by
$\pi_a(u\ot f(\ttt))=f(a)u$ for $u\in\uglk$ and $f(\ttt)\in\field[\ttt,\ttt^{-1}]$.
The evaluation map at $a$ for $\afuglk$ is defined as the composition:
$$\Ev_a:\afuglk\stackrel{\Ev_\field}{\lra}\uglk\ot\field[\ttt,\ttt^{-1}]\stackrel{\pi_a}{\lra}
\uglk.$$

\subsection{Evaluation maps $\ti\ev_a:\afSrk\ra\Srk$ for affine Schur algebras}\label{Ognk definiton}
Let $\Ognk$ be the $\field$-submodule of $\Ogk$ spanned by $\{\og_s\mid 1\leq s\leq n\}$, and define the Schur algebra over $\field$ as:
$$\Srk:=\End_{\field\fSr}
(\Ognk^{\ot r}).$$

Recall the map $\afxrk$ defined in \eqref{afxrk}. Its restriction to $\uglk$  gives
a surjective algebra homomorphism
$$\xrk:=\afxrk|_{\uglk}:\uglk\ra\Srk$$
(cf. \cite{CL}).

For $a\in\field^*$ there exists an evaluation map $$\ev_a:\field\affSr\ra\field\fSr$$ satisfying $\ev_a(s_i)=s_i$ and $\ev_a(\bse_j)=a^{-1}$  for $1\leq i\leq r-1$ and $1\leq j\leq r$. This induces a  $\field\fSr$-module homomorphism
$$\varepsilon_a:\Ogk^{\ot r}\ra
\Ognk^{\ot r}$$
defined by $\varepsilon_a(xh)=x\ev_a(h)$
for $x\in\Ognk^{\ot r}$ and $h\in\field\affSr$.
We now use the evaluation map $\ev_a$
to construct an evaluation map $\ti\ev_a$ from the affine Schur algebra $\afSrk$ to the Schur algebra $\Srk$.
\begin{Lem}\label{tieva}
For $a\in\field^*$, the linear map $\ti\ev_a:\afSrk\ra\Srk$ defined by $((\ti\ev_a)(\vi))(x)=
\varepsilon_a(\vi(x))$, for $\vi\in\afSrk$ and $x\in\Ognk^{\ot r}$, is a surjective algebra homomorphism.
\end{Lem}

\subsection{Compatibility between  evaluation maps}
We now establish the compatibility between the evaluation maps $\Ev_a$ and $\ti\ev_a$. This compatibility ensures that the evaluation processes for both the universal enveloping algebra and the Schur algebra are coherently related.

To establish this compatibility, we first need to understand how the comultiplication structure interacts with our generators.
Let $\Delta^{(1)}=\Delta:\afuglq\ra\afuglq^{\ot 2}$ be the comultiplication  for the Hopf algebra $\afuglq$ and define  $\Delta^{(m-1)}:\afuglq\ra\afuglq^{\ot m}$ inductively by $\Delta^{(m-1)}=(\Delta\ot 1^{\ot(m-2)})\circ\Delta^{(m-2)}$
for $m\geq 2$.

We begin with the behavior of divided powers under iterated comultiplication:
\begin{Lem}\label{Deltar(Hi k)}
For $k\in\mbn$, $1\leq i\not=j\leq n$, $l\in\mbz$ and $r\geq 2$ we have
\begin{equation*}
\begin{split}
\Delta^{(r-1)}((\afE_{i,j+ln})^{(k)})
&=\sum_{k_1+k_2+\cdots+k_r=k\atop
k_s\geq 0,\,\forall s} (\afE_{i,j+ln})^{(k_1)} \ot
(\afE_{i,j+ln})^{(k_2)} \ot\cdots\ot (\afE_{i,j+ln})^{(k_r)},\\
 \Delta^{(r-1)}\bigg(\bigg({H_i\atop k}\bigg)\bigg)
&=\sum_{k_1+k_2+\cdots+k_r=k\atop
k_s\geq 0,\,\forall s} \bigg({H_i\atop k_1}\bigg) \ot
 \bigg({H_i\atop k_2}\bigg) \ot\cdots\ot \bigg({H_i\atop k_r}\bigg).
\end{split}
\end{equation*}
\end{Lem}

We next examine the comultiplication of the elements
\(\Lambda_{i,\pm k,l}\), which is needed for the compatibility
of the evaluation maps.
\begin{Lem}\label{Delta(Laikl)}
For $1\leq i\leq n$, $k\in\mbn$ and $l>0$ we have $$\Delta(\La_{i,\pm k,l})=\sum_{a+b=k\atop a,b\geq 0}
\La_{i,\pm a,l}\ot\La_{i,\pm b,l}.$$
\end{Lem}
\begin{proof}
We argue by induction on $k$. The base cases $k=0$ and $k=1$ are straightforward to verify. Now assume $k\geq 2$ and that the result holds for all smaller indices.
By Lemma \ref{Laikl} and the induction hypothesis we have
\begin{equation*}
\Delta(\La_{i,\pm k,l})=-\frac{1}{k}\sum_{0\leq s\leq k-1}\Delta(g_{i,\pm(k-s)l})\Delta(\La_{i,\pm s,l})=-\frac{1}{k}(f_1+f_2)
\end{equation*}
where $$f_1=\sum_{s=0}^{k-1} (g_{i,\pm(k-s)l}\ot 1)\big(\sum_{a+b=s\atop a,b\geq 0}
\La_{i,\pm a,l}\ot\La_{i,\pm b,l}\big)$$
and
$$f_2=\sum_{s=0}^{k-1} (1\ot g_{i,\pm(k-s)l})\big(\sum_{a+b=s\atop a,b\geq 0}
\La_{i,\pm a,l}\ot\La_{i,\pm b,l}\big).$$

We now simplify \( f_1 \). Changing the order of summation and reindexing:
\[
\begin{aligned}
f_1 &=\sum_{b = 0}^{k - 1} \big(\sum_{a+k-s=k-b\atop a\geq 0,\,0\leq s\leq k-1}
g_{i,\pm(k-s)l}\La_{i,\pm a,l}\big)\ot\La_{i,\pm b,l}
\\
&= \sum_{b = 0}^{k - 1} \left( \sum_{a = 0}^{k - b - 1} g_{i,\pm(k - a - b)l} \Lambda_{i,\pm a,l} \right) \otimes \Lambda_{i,\pm b,l}.
\end{aligned}
\]
Using Lemma \ref{Laikl} again, we obtain
\begin{equation*}
\begin{split}
f_1
&=-\sum_{b=0}^{k-1}(k-b) \La_{i,\pm(k-b),l}\ot\La_{i,\pm b,l}=-\sum_{a+b=k\atop a,b\geq 0 }a\La_{i,\pm a,l}\ot\La_{i,\pm b,l}.
\end{split}
\end{equation*}
Similarly,
\begin{equation*}
\begin{split}
f_2&=\sum_{a=0}^{k-1}
\La_{i,\pm a,l}\ot \big(\sum_{b+k-s=k-a\atop b\geq 0,\,0\leq s\leq k-1} g_{i,\pm(k-s)l}\La_{i,\pm b,l}\big)\\
&=-\sum_{a=0}^{k-1}(k-a)
\La_{i,\pm a,l}\ot\La_{i,\pm(k-a)l}\\
&=-\sum_{a+b=k\atop a,b\geq 0 }b\La_{i,\pm a,l}\ot\La_{i,\pm b,l}.
\end{split}
\end{equation*}
Therefore,
\[
\begin{aligned}
\Delta(\Lambda_{i,\pm k,l})
&= -\frac{1}{k} \left( -\sum_{a + b = k} a \Lambda_{i,\pm a,l} \otimes \Lambda_{i,\pm b,l} - \sum_{a + b = k} b \Lambda_{i,\pm a,l} \otimes \Lambda_{i,\pm b,l} \right) \\
&= \frac{1}{k} \sum_{a + b = k} (a + b) \Lambda_{i,\pm a,l} \otimes \Lambda_{i,\pm b,l} \\
&= \sum_{\substack{a + b = k \\ a,b \geq 0}} \Lambda_{i,\pm a,l} \otimes \Lambda_{i,\pm b,l},
\end{aligned}
\]
as required. This completes the induction and the proof.
\end{proof}
\begin{Coro}\label{Deltar(Laikl)}
For $r\geq 2$, $1\leq i\leq n$, $k\in\mbn$ and $l>0$, we have
$$\Delta^{(r-1)}(\La_{i,\pm k,l})=\sum_{k_1+k_2+\cdots+k_r=k\atop k_s\geq 0,\,\forall s}
\La_{i,\pm k_1,l}\ot\La_{i,\pm k_2,l}\ot\cdots\ot\La_{i,\pm k_r,l}.$$
\end{Coro}
\begin{proof}
This follows by induction on
$r$, using Lemma \ref{Delta(Laikl)} as the base case and the coassociativity of the coproduct.
\end{proof}

Let $$I(n,r)=\{(i_1,\ldots,i_r)\in\mbz^r\mid 1\leq i_k\leq
n,\,\forall k\}.$$ Then the set $\{\og_\bfi\mid\bfi\in I(n,r)\}$ forms a $\field$-basis for $\Ognk^{\ot r}$. This basis will be essential in establishing the compatibility of various maps through explicit computations.
\begin{Lem}\label{lem1 for commutative diagram}
For $1\leq i\leq n$, $k\in\mbn$ and $l>0$ we have $\xrk\circ\Ev_a(\La_{i,\pm k,l})=\ti\ev_a\circ\afxrk(\La_{i,\pm k,l})$.
\end{Lem}
\begin{proof}
Let $\bfi\in I(n,r)$.
Note that $\big({H_i\atop s}\big)\og_j=0$ for $s\geq 2$.
By Lemma \ref{Ev(La{i,k,l})} and  \ref{Deltar(Hi k)}
we have
\begin{equation*}
\begin{split}
(\xrk\circ\Ev_a(\La_{i,\pm k,l}))(\og_\bfi)&=(-1)^ka^{\pm kl}\xrk\bigg(\bigg({H_i\atop k}\bigg)\bigg)(\og_\bfi)\\
&=(-1)^ka^{\pm kl}\sum_{k_1+k_2+\cdots+k_r=k\atop
0\leq k_s\leq 1,\,\forall s} \bigg({H_i\atop k_1}\bigg)\og_{i_1}
  \ot\cdots\ot \bigg({H_i\atop k_r}\bigg)\og_{i_r}\\
  &=(-1)^ka^{\pm kl}\sum_{k_1+k_2+\cdots+k_r=k\atop 0\leq k_s\leq 1,\,\forall s}
\dt_{\bar i,\bar{i}_1}^{k_1}\og_{i_1}\ot\cdots\ot\dt_{\bar i,\bar{i}_r}^{k_r}\og_{i_r}.
\end{split}
\end{equation*}
On the other hand, by Proposition \ref{Laikl formula},  the action on basis elements is given by:
\begin{equation}\label{Laikl(ogj)}
\La_{i,\pm k,l}\og_j =
(-1)^k(k\afE_{i,i\pm ln})[\bfl,1]\og_j=\begin{cases}
(-1)^k\dt_{\bar i,\bar j}^k\og_{j\mp kln}&\text{if $0\leq k\leq 1$,}\\
0&\text{otherwise,}
\end{cases}
\end{equation}
for $j\in\mbz$. Using Corollary \ref{Deltar(Laikl)} we obtain
\begin{equation*}
\begin{split}
\ti\ev_a\circ\afxrk(\La_{i,\pm k,l})(\og_\bfi)
&=\varepsilon_a\bigg(\sum_{k_1+k_2+\cdots+k_r=k\atop 0\leq k_s\leq 1,\,\forall s}
\La_{i,\pm k_1,l}\og_{i_1}\ot\cdots\ot\La_{i,\pm k_r,l}\og_{i_r}\bigg)\\
&=\varepsilon_a\bigg(\sum_{k_1+k_2+\cdots+k_r=k\atop 0\leq k_s\leq 1,\,\forall s}
(-1)^{k_1}\dt_{\bar i,\bar{i}_1}^{k_1}\og_{i_1\mp k_1ln}\ot\cdots\ot
(-1)^{k_r}\dt_{\bar i,\bar{i}_r}^{k_r} \og_{i_r\mp k_rln}\bigg)\\
&=(-1)^ka^{\pm kl}\sum_{k_1+k_2+\cdots+k_r=k\atop 0\leq k_s\leq 1,\,\forall s}
\dt_{\bar i,\bar{i}_1}^{k_1}\og_{i_1}\ot\cdots\ot\dt_{\bar i,\bar{i}_r}^{k_r}\og_{i_r}.
\end{split}
\end{equation*}
Hence we have $\xrk\circ\Ev_a(\La_{i,\pm k,l})(\og_\bfi)=\ti\ev_a\circ\afxrk(\La_{i,\pm k,l})(\og_\bfi)$.
Since the set $\{\og_\bfi\mid\bfi\in I(n,r)\}$ forms a $\field$-basis for $\Ognk^{\ot r}$, this establishes the equality of linear maps, and the lemma follows.
\end{proof}

\begin{Lem}\label{lem2 for commutative diagram}
For $1\leq i\not=j\leq n$, $l\in\mbz$ and $k\in\mbn$
we have $\xrk\circ\Ev_a((\afE_{i,j+ln})^{(k)})=\ti\ev_a\circ\afxrk(\afE_{i,j+ln})^{(k)})$.
\end{Lem}
\begin{proof}
Let $\bfi\in I(n,r)$. First observe that
\begin{equation}\label{E{i,j+ln}^{(m)} acts on OG}
(\afE_{i,j+ln})^{(m)}\og_s=0
\end{equation}
for $s\in\mbz$ and $m>1$.
Applying Lemma \ref{Deltar(Hi k)} we have
\begin{equation*}
\begin{split}
\ti\ev_a\circ\afxrk((\afE_{i,j+ln})^{(k)})(\og_\bfi)
&=\varepsilon_a\bigg(\sum_{k_1+k_2+\cdots+k_r=k\atop
0\leq k_s\leq 1,\,\forall s} (\afE_{i,j+ln})^{(k_1)}\og_{i_1}\ot\cdots\ot (\afE_{i,j+ln})^{(k_r)}\og_{i_r}\bigg)\\
&=\varepsilon_a\bigg(\sum_{k_1+k_2+\cdots+k_r=k\atop
0\leq k_s\leq 1,\,\forall s}\dt_{i_1,j}^{k_1}\og_{i_1+k_1(i-j)-k_1ln}\ot\cdots\ot \dt_{i_r,j}^{k_r}\og_{i_r+k_r(i-j)-k_rln}\bigg)\\
&=a^{kl}\sum_{k_1+k_2+\cdots+k_r=k\atop
0\leq k_s\leq 1,\,\forall s}\dt_{i_1,j}^{k_1}\og_{i_1+k_1(i-j)}\ot\cdots\ot \dt_{i_r,j}^{k_r}\og_{i_r+k_r(i-j)}.
\end{split}
\end{equation*}
Conversely, using Lemma \ref{Ev(E i,j+ln)} we have
\begin{equation*}
\begin{split}\xrk\circ\Ev_a((\afE_{i,j+ln})^{(k)})(\og_\bfi)&=
a^{kl}\xrk(E_{i,j}^{(k)})\og_\bfi\\
&=a^{kl}\sum_{k_1+k_2+\cdots+k_r=k\atop
0\leq k_s\leq 1,\,\forall s}\dt_{i_1,j}^{k_1}\og_{i_1+k_1(i-j)}\ot\cdots\ot \dt_{i_r,j}^{k_r}\og_{i_r+k_r(i-j)}.
\end{split}
\end{equation*}
Consequently, we have
$\xrk\circ\Ev_a((\afE_{i,j+ln})^{(k)})(\og_\bfi)=\ti\ev_a\circ\afxrk(\afE_{i,j+ln})^{(k)})(\og_\bfi)$. Since the set $\{\og_\bfi\mid\bfi\in I(n,r)\}$ forms a $\field$-basis for $\Ognk^{\ot r}$, this establishes the equality of linear maps, and the lemma follows.
\end{proof}
Combining Theorem \ref{generators2} (1), Lemma \ref{lem1 for commutative diagram} and \ref{lem2 for commutative diagram}, we obtain the following result.
\begin{Prop}\label{commute diagram}
We have the following
commutative diagram:
$$\xymatrix{
\afuglk \ar[rr]^{\Ev_a}\ar[d]_{\afxrk}  &  &\uglk \ar[d]_{\xrk}  \\
\afSrk \ar[rr]^{\ti\ev_a}&  & \Srk.
}$$
\end{Prop}

\section{Polynomial representations of the hyper loop algebra  $\afuglk$}
Throughout this section, $\field$ denotes an algebraically closed field of characteristic $p \geq 0$. This section is dedicated to the systematic study of polynomial representations of the hyper loop algebra $\afuglk$ over  $\field$.
\subsection{Finite-dimensional $\afuslk$-modules}
We begin by studying finite-dimensional representations of the hyper loop algebra $\afuslk$ (cf. \cite{JM07}).
Let $\afuslq$ be the enveloping algebra of $\afsl$. Following Garland \cite{Ga} let $\afuslz$ be the $\mbz$-subalgebra of $\afuslq$
generated by the elements $(\afE_{i,j+ln})^{(k)}$
for $1\leq i\not=j\leq n$, $l\in\mbz$ and $k\in\mbn$.
This integral form $\afuslz$ naturally embeds into the integral form $\afuglz$ of the enveloping algebra $\afuglq$.

For $1\leq i\leq n-1$ and $k\geq 0$ we define the elements $\Ga_{i,\pm k}$ in $\afuslz$ through the
generating functions in indeterminate $u$
\begin{equation}\label{Gaik}
\Ga_i^\pm(u):=\sum_{k\geq 0}\Ga_{i,\pm k}u^k=\exp\big(-\sum_{k\geq 1}\frac{h_{i,\pm k}}{k}u^k\big),
\end{equation}
where $h_{i,\pm k}=g_{i,\pm k}-g_{i+1,\pm k}$, and the elements $g_{i,\pm k}$ are defined in \eqref{gil}.

Let $\afuslk=\afuslz\ot_\mbz\field$  be the specialization to our ground field $\field$. The images of $\Ga_{i,\pm k}$  in $\afuslk$ will be denoted by the same letters. This specialization process allows us to study representations in positive characteristic as well as characteristic zero.

For $g(u)=\prod_{1\leq i\leq m}(1-a_iu)\in\field[u]$
with constant term $1$ and $a_i\in\field^*$, define
\begin{equation}\label{f^pm(u)}
g^\pm(u)=\prod_{1\leq i\leq m}(1-a_i^{\pm1}u).
\end{equation}

Let $\Pn$ be the set of $(n-1)$-tuple polynomials $\bfP=(P_1(u),\ldots,P_{n-1}(u))$ such that $P_i(u)\in\field[u]$ and the constant term of $P_i(u)$ is $1$ for $1\leq i\leq n-1$.
For $\bfP\in\Pn$, define coefficients $P_{j,s}\in\field$,
for $1\leq j\leq n-1$ and $s\in\mbz$  via the expansions $$P_j^\pm(u)=\sum_{s\geq
0}P_{j,\pm s} u^{s},$$ where $P_j^\pm(u)$ is defined using
\eqref{f^pm(u)}. These coefficients will serve as eigenvalues for the action of the $\Ga_{i,\pm k}$ operators.

For $\bfP\in\Pn$, let $\bar I(\bfP)$
be the left ideal of $\afuslk$ generated by $(\afE_{i,j+sn})^{(m)},\quad
\Ga_{i,s}-P_{i,s},$ and $\big({H_i-H_{i+1}\atop m}\big)- \big({\mu_i\atop m}\big)$, for
$1\leq i<j\leq n$, $m>0$ and $s\in\mbz$, where $\mu_i=\deg P_i(u)$,
and define
$$\bar M(\bfP)=\afuslk/\bar I(\bfP).$$
The module $\bar M(\bfP)$ has a unique irreducible quotient, denoted by $\bar L(\bfP)$.

\subsection{The polynomial representation category}
We begin by defining weight spaces for $\uglk$-modules.
Given a $\uglk$-module $V$ and $\mu\in\mbz^n$ we set $$V_\mu=\bigg\{w\in V\mid\bigg({H_i\atop k}\bigg)w
=\bigg({\mu_i\atop k}\bigg)w,\,\forall k\in\mbn,\,1\leq i\leq n\bigg\}.$$
Clearly if $p>0$ then
we have $V_\mu=\big\{w\in V\mid\big({H_i\atop p^k}\big)w
=\big({\mu_i\atop p^k}\big)w,\,\forall k\in\mbn,\,1\leq i\leq n\big\}.$

Let $\ms C_n$ be the category of finite-dimensional $\uglk$-modules $V$ such that
$V=\oplus_{\la\in\mbnn}V_\la$. We further introduce  $\h{\ms C}_n$ as the category of all finite-dimensional $\afuglk$-modules $V$ such that
the restriction of $V$ to $\uglk$ (denoted $V|_{\uglk}$) lies in $\ms C_n$
and each generalized common eigenvalue of $\La_i^\pm(u)$ in $V_\la$ is a polynomial $Q_i^\pm(u)$ in $u$ of degree $\la_i$, where $$Q_i^\pm(u)=\prod_{1\leq k\leq\la_i}(1-a_{i,k}^{\pm 1}u).$$ We will call the objects of the category $\h{\ms C}_n$ polynomial $\afuglk$-modules.

For $1\leq i\leq n-1$ let
\begin{equation}\label{ali}
 \al_i=(\dt_{i,j}-\dt_{i+1,j})_{1\leq j\leq n}\in\mbz^n.
\end{equation}
These  are the standard simple roots of the $A_{n-1}$ root system.
For $\la,\mu\in\mbz^n$ we write $\la\trianglelefteq\mu$ if $\mu-\la\in\sum_{1\leq i\leq n-1}\mbn\al_i$.  This partial order is known as the dominance order relative to the simple roots $\al_i$.
For $1\leq i\not=j\leq n$ let $\al_{i,j}
=(\dt_{i,k}-\dt_{j,k})_{1\leq k\leq n}\in\mbz^n$.
Then the set $\{\al_{i,j}\mid 1\leq i\not=j\leq n\}$
is the set of roots of type $A_{n-1}$. Clearly we have the following result.
\begin{Lem}\label{Ejk Vmu}
Let $V$ be a $\afuglk$-module.
For $1\leq j\not=k\leq n$, $m\in\mbn$, $s\in\mbz$ and $\mu\in\mbz^n$ we have
$(\afE_{j,k+sn})^{(m)}V_\mu\han V_{\mu+m\al_{j,k}}$.
\end{Lem}

An $n$-tuple of polynomials
$\bfQ=(Q_1(u),\ldots,Q_n(u))\in\field[u]^n$ with constant terms $1$ is called
dominant if, for each $1\leq i\leq n-1$, the ratio
$Q_i(u)/Q_{i+1}(u)\in\field[u]$. Let
$\Qn$ be the set of dominant $n$-tuples of polynomials. For $\bfQ\in\Qn$ let $\deg\bfQ=(\deg Q_1(u),\deg Q_2(u),\ldots,\deg Q_n(u))$.

Recall the elements $\La_{i,\pm k}$ defined in \eqref{Laik generating function},
which are the coefficient components of the generating functions $\La_i^\pm(u)$.
Note that we have
$$\Ga_{i}^\pm(u)=\frac{\La_i^\pm(u)}{\La_{i+1}^\pm(u)}$$
for $1\leq i\leq n-1$. These $\Ga_{i}^\pm(u)$ are auxiliary generating functions that inherit the dominance property from $\bfQ$, simplifying the analysis of the ratio of eigenvalues corresponding to consecutive indices $i$ and $i+1$.
As before, we shall denote the images of $\La_{i,\pm k}$  in $\afuglk$ by the same letters.

For $\bfQ=(Q_1(u),\ldots,Q_{n}(u))\in\Qn$, define
$Q_{i,s}\in\field$, for $1\leq i\leq n$ and $s\in\mbz$, by the
following formula
$$Q_i^\pm(u)=\sum_{s\geq 0}Q_{i,\pm s}u^{s},$$
where $Q_i^\pm(u)$ is defined using \eqref{f^pm(u)}.
Let $I(\bfQ)$
be the left ideal of $\afuglk$ generated by $(\afE_{j,k+sn})^{(m)},\quad
\La_{i,s}-Q_{i,s},$ and $\big({H_i\atop m}\big)- \big({\la_i\atop m}\big)$, for $1\leq i\leq n$,
$1\leq j<k\leq n$, $m>0$ and $s\in\mbz$, where $\la=\deg\bfQ$,
and define
$$M(\bfQ)=\afuglk/I(\bfQ).$$

To further analyze the structure of $M(\bfQ)$, we decompose $\afuglk$ into a triangular product of subalgebras.
Let $\afuglkp$ (respectively, $\afuglkm$) be the $\field$-subalgebra of $\afuglk$ generated by
the elements $(\afE_{i,j+ln})^{(s)}$ (respectively, $(\afE_{j,i+ln})^{(s)}$) for $1\leq i<j\leq n$, $l\in\mbz$ and $s\in\mbn$;
these are the  positive  and  negative  subalgebras of $\afuglk$, analogous to the upper and lower triangular subalgebras in classical semisimple Lie algebras.
Furthermore let $\afuglkz$ be the $\field$-subalgebra of $\afuglk$ generated by
the elements $\big({H_i\atop k}\big)$ and $\La_{i,\pm k,l}$ for $1\leq i\leq n$, $k\in\mbn$ and $l>0$. This subalgebra serves as the Cartan subalgebra of $\afuglk$, capturing the weight space structure and eigenvalue information of representations.

By \cite[Lem. 5.11]{Ga} and \cite[Lem. 1.6]{JM07}, the algebra $\afuglkz$ is generated by
the elements $\big({H_i\atop k}\big)$ and $\La_{i,\pm k}$ for $1\leq i\leq n$, $k\in\mbn$.
Furthermore we have the  triangular decomposition $\afuglk=\afuglkm\afuglkz\afuglkp$ (cf. \cite{Fu19} and \cite{Fu23}), which is fundamental to analyzing the structure of $\afuglk$-modules.

For $\bfQ\in\Qn$, the triangular decomposition of $\afuglk$ implies $M(\bfQ)=\afuglkm\bar 1$ for $\bfQ\in\Qn$, where $\bar 1=1+I(\bfQ)\in M(\bfQ)$.
By Lemma \ref{Ejk Vmu} we obtain the weight space decomposition: $$M(\bfQ)=\bop_{\mu\in\mbz^n\atop\mu\trianglelefteq\la}M(\bfQ)_\mu$$
for $\bfQ\in\Qn$ and $\dim M(\bfQ)_\la=1$, where $\la=\deg(\bfQ)$. For $\bfQ\in\Qn$, the module $M(\bfQ)$ has a unique irreducible quotient, denoted by $L(\bfQ)$.  The polynomials $Q_i(u)$ are called  Drinfeld
polynomials associated with $L(\bfQ)$. These polynomials completely determine the isomorphism type of the irreducible representation and play a crucial role in the classification of finite-dimensional polynomial representations of $\afuglk$.

\subsection{The Drinfeld polynomials of the evaluation module $L(\la)_a$}
Let $\La^+(n)=\{\la\in\mbn^n\mid\la_i\geq\la_{i+1},\forall 1\leq i\leq n-1\}$ be the set of dominant integral weights.
For $r\in\mbn$ let $$\La^+(n,r)=\big\{\la\in\La^+(n)\mid\sum_{1\leq i\leq n}\la_i=r\big\}.$$ For $\la\in\La^+(n,r)$ let $L(\la)$ be the irreducible $\Srk$-module with highest weight $\la$ (cf. Green \cite{Gr80}). Given any $\Srk$-module $V$ and $a\in\field^*$, let $V_a$ be the pullback of $V$ by $\ti\ev_a$, where $\ti\ev_a$ is the evaluation map for affine Schur algebras defined in Lemma \ref{tieva}.

For $\la\in\La^+(n)$ and $a\in\field^*$, define the $n$-tuple of polynomials $$\bfQ_{\la,a}=((1-au)^{\la_1},(1-au)^{\la_2},\ldots,(1-au)^{\la_n})\in\Qn.$$
This tuple encodes the Drinfeld polynomials associated to the evaluation module $L(\la)_a$.

\begin{Lem}\label{evaluation representatin L(la)a}
For $\la\in\La^+(n,r)$ and $a\in\field^*$, the evaluation module $L(\la)_a$ is isomorphic to $L(\bfQ_{\la,a})$ as a $\afuglk$-module. In particular, $L(\bfQ_{\la,a})$ can be regarded as a $\afSrk$-module via  the homomorphism  $\afxrk$ defined in \eqref{afxrk}.
\end{Lem}
\begin{proof}
Since $\dim L(\la)_\la=1$, there is a nonzero highest weight vector $w_0\in L(\la)_\la$.  By Lemma \ref{Ev(E i,j+ln)} and Proposition \ref{commute diagram} we compute the action of the generators of $\afuglk$ on $w_0$:
\begin{equation*}
\begin{split}
\ti\ev_a\circ\afxrk((\afE_{i,j+ln})^{(k)})&w_0=\xrk\circ\Ev_a((\afE_{i,j+ln})^{(k)})w_0
=a^{lk}\xrk((\afE_{i,j})^{(k)})w_0=0 \\
\ti\ev_a\circ\afxrk\bigg(\bigg({H_i\atop k}\bigg)\bigg)w_0
&=\xrk\circ\Ev_a\bigg(\bigg({H_i\atop k}\bigg)\bigg)w_0=\xrk\bigg(\bigg({H_i\atop k}\bigg)\bigg)w_0=\bigg({\la_i\atop k}\bigg)w_0
\end{split}
\end{equation*}
for $1\leq i<j\leq n$, $l\in\mbz$ and $k\in\mbn$. Furthermore by Lemma \ref{Ev(La{i,k,l})} we have
$$\ti\ev_a\circ\afxrk(\La_{i,\pm k})w_0=\xrk\circ\Ev_a(\La_{i,\pm k})w_0=
(-a)^{\pm k}\xrk\bigg(\bigg({H_i\atop k}\bigg)\bigg)w_0=(-a)^{\pm k}\bigg({\la_i\atop k}\bigg)w_0.$$
Hence there is a surjective $\afuglk$-module homomorphism $\vi:M(\bfQ)\ra L(\la)_a$ such that
$\vi(\bar 1)=w_0$, where $\bar 1=1+I(\bfQ)$. Since $L(\bfQ)$ is the unique
irreducible quotient of $M(\bfQ)$, we conclude that $L(\la)_a\cong L(\bfQ)$.
\end{proof}

\subsection{Classification of simple polynomial modules}
Let $\bfQ\in\Qn$ be an $n$-tuple of polynomials. Let
$\{a_s^{-1} \mid 1\leq s\leq m\}$ be the set of distinct roots of the polynomial $\prod_{1\leq i\leq n}Q_i(u)$. For $1\leq i\leq n$, we  can factor $Q_i(u)$ as
$$Q_i(u)=\prod_{1\leq s\leq m}(1-a_su)^{b_{i,s}}$$
where $b_{i,s}\in\mbn$, $a_s\in\field^*$ and $a_s\not=a_t$ for $s\not=t$.

For $1\leq s\leq m$, let $\la^{(s)}=(b_{1,s},b_{2,s},\ldots,b_{n,s})\in\La^+(n)$
be the dominant weight determined by the exponents in the factorization. Then we have the factorization
$$\bfQ=\prod_{1\leq s\leq m}\bfQ_{\la^{(s)},a_s}.$$
Let $\Qnr=\{\bfQ\in\Qn\mid\sum_{1\leq i\leq n}\deg Q_i(u)=r\}$ be the set of $n$-tuples of polynomials with total degree $r$.
\begin{Prop}\label{decomposition of L(Q)}
For $\bfQ\in\Qn$, we write $\bfQ=\prod_{1\leq s\leq m}\bfQ_{\la^{(s)},a_s}$, where
$\{a_s^{-1} \mid 1\leq s\leq m\}$ is the set of distinct roots of $\prod_{1\leq i\leq n}Q_i(u)$ and
$\la^{(s)}\in\La^+(n)$. Then we
have $$L(\bfQ)\cong L(\la^{(1)})_{a_1}\ot L(\la^{(2)})_{a_2}\ot\cdots\ot L(\la^{(m)})_{a_m},$$
and $L(\bfQ)|_{\afuslk}\cong\bar L(\bfP)$, where $\bfP=(P_1(u),\ldots,P_{n-1}(u))$ with $P_i(u)=Q_i(u)/Q_{i+1}(u)$.
In particular, we have $L(\bfQ)\in\h{\ms C}_n$.
\end{Prop}
\begin{proof}
Let $V=L(\la^{(1)})_{a_1}\ot L(\la^{(2)})_{a_2}\ot\cdots\ot L(\la^{(m)})_{a_m}$.
By Lemma \ref{evaluation representatin L(la)a}, we have
$$V\cong L(\bfQ_{\la^{(1)},a_1})\ot L(\bfQ_{\la^{(2)},a_2})\ot\cdots\ot
L(\bfQ_{\la^{(m)},a_m}).$$
For $1\leq s\leq m$ we choose a nonzero  highest weight vector
$w_s\in L(\bfQ_{\la^{(s)},a_s})_{\la^{(s)}}$. Let $w=w_1\ot w_2\ot\cdots\ot w_m$. By Lemma \ref{Deltar(Hi k)} and Corollary \ref{Deltar(Laikl)}
we compute the action of the generators of $\afuglk$ on $w$:
\begin{equation*}
(\afE_{j,k+sn})^{(m)}w=0,\quad
\bigg({H_i\atop m}\bigg)w= \bigg({\la_i\atop m}\bigg)w,\quad
\La_i^\pm(u)w=Q_i^\pm(u)w
\end{equation*}
for $1\leq i\leq n$, $1\leq j<k\leq n$, $s\in\mbz$ and $m>0$, where $\la=\deg\bfQ$. Hence $V$ is a quotient module of $M(\bfQ)$.
Furthermore, by \cite[Cor. 3.5]{JM07} the restriction $V|_{\afuslk}$ is irreducible and  isomorphic to $\Lcp(\bfP)$. It follows that $V$ is an irreducible $\afuglk$-module. Since $L(\bfQ)$ is the unique
irreducible quotient of $M(\bfQ)$, we conclude that $V\cong L(\bfQ)$.
\end{proof}
\begin{Thm}\label{polynomial irreducible afuglk-modules}
The set $\{L(\bfQ)\mid\bfQ\in\Qn\}$ forms a complete set of non-isomorphic finite-dimensional irreducible polynomial $\afuglk$-modules.
\end{Thm}
\begin{proof}
Let $V$ be a finite-dimensional irreducible polynomial $\afuglk$-module. Then there exist $w_0\in V$ such that
$$(\afE_{j,k+sn})^{(m)}w_0=0 ,\quad
 \La_{i}^\pm(u)w_0=Q_i^\pm(u)w_0,\quad  \bigg({H_i\atop m}\bigg)w_0= \bigg({\la_i\atop m}\bigg)w_0$$
for $1\leq i\leq n$,
$1\leq j<k\leq n$, $m>0$ and $s\in\mbz$, where $\la_i=\deg Q_i(u)$. By \cite[Th. 2.4 and Prop. 3.1]{JM07}, we have $Q_i(u)/Q_{i+1}(u)\in\field[u]$ for $1\leq i\leq n-1$,
which implies that $\bfQ\in\Qn$ and $V\cong L(\bfQ)$. The assertion now follows from Proposition \ref{decomposition of L(Q)}.
\end{proof}

\section{Finite-dimensional representations of affine Schur algebras }
We maintain the assumptions from the previous section. In particular, $\field$ denotes an algebraically closed field.
\subsection{The irreducible $\afSrk$-module $L(\bfQ)$}
For $r',r''\in\mbn$, there is a natural algebra homomorphism
$$\vi_{r',r''}:\End(\Ogk^{\ot r'})\ot\End(\Ogk^{\ot r''})\ra\End(\Ogk^{\ot r'+r''})$$
defined by $\vi_{r',r''}(f\ot g)(w_1\ot w_2)
=f(w_1)\ot g(w_2)$ for $f\in\End(\Ogk^{\ot r'})$,
$g\in\End(\Ogk^{\ot r''})$, $w_1\in\Ogk^{\ot r'}$,
$w_2\in\Ogk^{\ot r''}$.

\begin{Lem}
For $r',r''\in\mbn$, the map $\vi_{r',r''}$ is injective.
\end{Lem}
\begin{proof}
Let $f\in\ker\vi_{r',r''}$. Write $f=\sum_{1\leq i\leq m}f_i\ot g_i$, where $f_i\in\End(\Ogk^{\ot r'})$,
$g_i\in\End(\Ogk^{\ot r''})$ and $\{g_i\mid 1\leq i\leq m\}$ is linearly independent.
Then we have $\sum_{1\leq i\leq m}f_i(w_1)\ot g_i(w_2)=0$ for $w_1\in\Ogk^{\ot r'}$ and
$w_2\in\Ogk^{\ot r''}$. It follows that
$$(F\ot id)(\sum_{1\leq i\leq m}f_i(w_1)\ot g_i(w_2))
=\sum_{1\leq i\leq m}F(f_i(w_1))g_i(w_2)=0$$
for $w_1\in\Ogk^{\ot r'}$,
$w_2\in\Ogk^{\ot r''}$ and $F\in(\Ogk^{\ot r'})^*$. Since $\{g_i\mid 1\leq i\leq m\}$ is linearly independent, it follows that $F(f_i(w_1))=0$ for $1\leq i\leq m$, $w_1\in\Ogk^{\ot r'}$ and $F\in(\Ogk^{\ot r'})^*$. This implies that $f_i(w_1)=0$ for $1\leq i\leq m$ and $w_1\in\Ogk^{\ot r'}$, hence $f=0$.
\end{proof}
The injectivity of $\vi_{r',r''}$ allows us to restrict to $\afSk{r'}\ot\afSk{r''}$, yielding an algebra isomorphism $$\vi_{r',r''}:\afSk{r'}\ot\afSk{r''}\ra\vi_{r',r''}(\afSk{r'}\ot\afSk{r''}).$$
Moreover, we have the compatibility relation $$\vi_{r',r''}\circ(\afxk{r'}\ot\afxk{r''})\circ\Delta=\afxk{r'+r''},$$
where $\afxrk$ is defined in \eqref{afxrk} and $\Delta:\afuglk\ra\afuglk\ot\afuglk$ is the comultiplication
of the Hopf algebra $\afuglk$.
Consequently, $$\afSk{r'+r''}=\afxk{r'+r''}(\afuglk)\han\vi_{r',r''}(\afSk{r'}\ot\afSk{r''}.$$
By restricting $\vi_{r'+r''}^{-1}$ to $\afSk{r'+r''}$, we obtain an
algebra homomorphism
$$\Delta_{r',r''}:=\vi_{r'+r''}^{-1}:
\afSk{r'+r''}\ra\afSk{r'}\ot\afSk{r''}.$$
This comultiplication-like structure on affine Schur algebras is fundamental for constructing new representations from existing ones via tensor products.

\begin{Lem}\label{V otimes W}
Let $V$ be a $\afSk{r'}$-module and $W$ be a $\afSk{r''}$-module. Then
$V\ot W$  carries a natural $\afSk{r'+r''}$-module structure.
\end{Lem}
\begin{proof}
Since $\afuglk$ is a Hopf algebra, $V\ot W$ becomes  a $\afuglk$-module via the comultiplication $\Delta$. The commutative diagram
$$\begin{CD}
\afuglk @>\Delta>>\afuglk\ot\afuglk\\
@V\afxk{r'+r''} VV  @VV\afxk{r'}\ot\afxk{r''}V\\
\afSk{r'+r''} @>>\Delta_{r',r''}>
\afSk{r'}\ot\afSk{r''}.
\end{CD}$$
ensures that the $\afuglk$-action on $V\ot W$ factors through $\afxk{r'+r''}$, making $V\ot W$ a $\afSk{r'+r''}$-module.
\end{proof}
With this tensor product construction in hand, we can now establish that the modules $L(\bfQ)$ (for $\bfQ\in\Qnr$) indeed carry an $\afSrk$-module structure.
\begin{Prop}\label{L(bfQ) is a afSrk-module}
For $\bfQ\in\Qnr$, $L(\bfQ)$ is a finite-dimensional irreducible $\afSrk$-module via $\afxrk$.
\end{Prop}
\begin{proof}
By Proposition \ref{decomposition of L(Q)}, as a $\afuglk$-module, $L(\bfQ)$ decomposes as
\[
L(\bfQ) \cong L(\la^{(1)})_{a_1}\ot L(\la^{(2)})_{a_2}\ot\cdots\ot L(\la^{(m)})_{a_m}
\]
for some $\la^{(i)}\in\La^+(n,r_i)$ and $a_{i}\in\mpk^*$, where $\sum_{1\leq i\leq m}r_i=r$. Furthermore from Lemma \ref{V otimes W}, we see that $L(\la^{(1)})_{a_1}\ot L(\la^{(2)})_{a_2}\ot\cdots\ot L(\la^{(m)})_{a_m}$ is a $\afSrk$-module,
since the evaluation module $L(\la^{(i)})_{a_i}$ is a $\afSk{r_i}$-module.
Therefore $L(\bfQ)$ can be regarded as a $\afSrk$-module via $\afxrk$.
\end{proof}
This proposition connects the representation theory of the hyper loop algebra $\afuglk$ with that of affine Schur algebras, showing that the irreducible modules $L(\bfQ)$ naturally descend to representations of the affine Schur algebra.
\subsection{The functor $\sG:\afS{N,r}\hmod\lra\afSrk\hmod$}
To further understand the relationship between affine Schur algebras of different sizes, we now introduce a Schur functor that relates representations of larger affine Schur algebras to smaller ones.

Assume $N\geq n$. There is an injective map
$$\ti{\,\,}:\afThn\lra\afTh(N),\quad A=(a_{i,j})\longmapsto\ti
A=(\ti a_{i,j}),$$ where $\ti A=(\ti a_{i,j})$
is defined by
\begin{equation*}\label{AtoAtilde}
\ti a_{k,l+mN}=\begin{cases} a_{k,l+mn}, &\text{ if }1\leq k,l\leq n;\\
0, &\text{ if either }n< k\leq N\text{ or }n< l\leq N\end{cases}
\end{equation*}
for  $m\in\mbz$.
This induces an injective
algebra homomorphism (not sending 1 to 1)
$$\iota_{n,N}:\afSrk\lra\afS{N,r},\;\;e_A\longmapsto e_{\ti A}\;\;\text{for $A\in\afThnr$}$$
as in \cite[\S 4.1]{DDF}. Similarly, there is an injective map $$\ti{}:\afLanr\ra\afLa(N,r),\quad\la\mapsto\ti\la,$$
where $\ti\la_i=\la_i$ for $1\leq i\leq n$ and $\ti\la_i=0$ for $n+1\leq i\leq N$.
Let $\ti\afLa(n,r)=\{\ti\la\mid\la\in\afLanr\}$.
Define the idempotent $$e=\sum_{\la\in\afLanr}e_{\diag(\ti\la)}\in\afS{N,r}.$$ Then we have
$$e\afS{N,r} e=\iota_{n,N}(\afSrk)\cong\afSrk.$$
We identify $e\afS{N,r} e$ with $\afSrk$.
Let $\afSrk\hmod$ denote  the category of finite-dimensional $\afSrk$-modules.
Define the functor
\begin{equation}\label{the functor G}
\sG:\afS{N,r}\hmod\lra\afSrk\hmod,\qquad
V\longmapsto eV.
\end{equation}
This functor projects representations of the larger algebra $\afS{N,r}$ onto representations of the smaller algebra $\afSrk$, allowing us to relate their representation theories.

For $\bfQ=(Q_1(u),\ldots,Q_n(u))\in\Qnr$ define
\begin{equation}\label{tibfQ}
\ti\bfQ=
(Q_1(u),\ldots,Q_n(u),1,\ldots,1)\in\QNr.
\end{equation}
Let
$\tiQnr=\{\ti\bfQ\mid\bfQ\in\Qnr\}.$
We now determine when the Schur functor $\sG$ yields nonzero results on irreducible modules.
\begin{Prop}\label{the condition of eL(bfQ)=0}
Assume $N\geq n$. For $\bfQ\in\QNr$ we have $\sG(L(\bfQ))\not=0$ if and only if $\bfQ\in\tiQnr$.
\end{Prop}
\begin{proof}
Let $\la=\deg(\bfQ)$. If $\bfQ\in\tiQnr$ then we have  $$e(L(\bfQ)_\la)=e(e_{\diag(\la)}L(\bfQ))=e_{\diag(\la)}L(\bfQ)=L(\bfQ)_\la\not=0,$$ so $\sG(L(\bfQ))\not=0$. Conversely, assume  $\sG(L(\bfQ))\not=0$. Since $1=\sum_{\al\in\afLa(N,r)}e_{\diag(\al)}$, we have
$$L(\bfQ)=\bop_{\al\in\afLa(N,r)}e_{\diag(\al)}L(\bfQ)$$
and $L(\bfQ)_\al=e_{\diag(\al)}L(\bfQ)$ for $\al\in\afLa(N,r)$.
Hence by Lemma \ref{Ejk Vmu}
we have $$L(\bfQ)=\bop_{\al\in\afLa(N,r)\atop\al\unlhd\la}e_{\diag(\al)}L(\bfQ).$$
It follows that
$$eL(\bfQ)=\bop_{\al\in\afLanr\atop\ti\al\leq\la}e_{\diag(\ti\al)}L(\bfQ).$$
Hence $e_{\diag(\ti\al)}L(\bfQ)\not=0$ for some $\al\in\afLanr$ with $\ti\al\unlhd\la$. It follows that $r=\sum_{1\leq i\leq n}\al_i\leq\sum_{1\leq i\leq n}\la_i\leq r$. Hence $\la\in\ti\afLa(n,r)$. Consequently, we have $\bfQ\in\tiQnr$.
\end{proof}
This result characterizes precisely which irreducible representations of the larger algebra survive under the Schur functor. Next, we show that when they do survive, they correspond exactly to the irreducible representations of the smaller algebra.

\begin{Prop}\label{properties of the functor G}
Assume $N\geq n$. For $\bfQ\in\Qnr$ we have $\sG(L(\ti\bfQ))\cong L(\bfQ)$ as a $\afSrk$-module.
\end{Prop}
\begin{proof}
Let $\la=\deg(\bfQ)$ and choose a nonzero element $w_0\in L(\ti\bfQ)_{\ti\la}$. As a
$\afuglkN$-module, $L(\ti\bfQ)$ is the unique irreducible quotient of $M(\ti\bfQ)$. It follows that
\begin{equation}\label{action}
(\afE_{j,k+lN})^{(m)}w_0=0\text{ and }\La_{i}^\pm(u)w_0=  Q_{i}^\pm(u)w_0
\end{equation}
for $1\leq i\leq n$, $1\leq j<k\leq n$ and $m>0$.
The $\field$-vector space $\sG(L(\ti\bfQ))$ is a $\afuglk$-module via $\iota_{n,N}\circ\afxrk$.
Since $w_0=e_{\diag(\ti\la)}w_0$, we have $ew_0=e_{\diag(\ti\la)}w_0=w_0.$
Hence we have
\[\iota_{n,N}\circ\afxrk\bigg(\bigg({H_i\atop t}\bigg)\bigg)(ew_0)=
\sum_{\mu\in\afLanr}\bigg({\mu_i\atop t}\bigg)e_{\diag(\ti\mu)}(e_{\diag(\ti\la)}w_0)=
\bigg({\la_i\atop t}\bigg)(ew_0)
\]
for $1\leq i\leq n$. Furthermore by Proposition \ref{generators}(2)  we have
\[
\begin{split}
\iota_{n,N}\circ\afxrk((\afE_{j,k+ln})^{(m)})(ew_0)&=
\iota_{n,N}(m\afE_{j,k+ln})[\bfl,r](e_{\diag(\ti\la)}w_0)\\
&=\sum_{\mu\in\afLa(n,r-m)}e_{m\afE_{j,k+lN}+\diag(\ti\mu)}(e_{\diag(\ti\la)}w_0)\\
&=e_{m\afE_{j,k+lN}+\diag(\ti\la-m\afbse_k)}w_0\\
&=(\afE_{j,k+lN})^{(m)}w_0\\
\end{split}
\]
for $1\leq j<k\leq n$, $l\in\mbz$ and $m>0$.
It follows from \eqref{action} that $$\iota_{n,N}\circ\afxrk((\afE_{j,k+ln})^{(m)})(ew_0)=0$$
for $1\leq j<k\leq n$, $l\in\mbz$ and $m>0$. In addition, by Proposition \ref{Laikl formula} we have
\[
\begin{split}
\iota_{n,N}\circ\afxrk(\La_{i,\pm k})(ew_0)&=
\iota_{n,N}((-1)^k(k\afE_{i,i\pm n})[\bfl,r])(e_{\diag(\ti\la)}w_0)\\
&=(-1)^k\sum_{\mu\in\afLa(n,r-k)}e_{k\afE_{i,i\pm N}+\diag(\ti\mu)}(e_{\diag(\ti\la)}w_0)\\
&=(-1)^ke_{k\afE_{i,i\pm N}+\diag(\ti\la-k\afbse_i)}w_0\\
&=\La_{i,\pm k} w_0
\end{split}
\]
for $1\leq i\leq n$ and $k\in\mbn$.
Hence by \eqref{action} we have
$$\iota_{n,N}\circ\afxrk(\La_{i,\pm k})(ew_0)=Q_{i,\pm k}(w_0)=Q_{i,\pm k}(ew_0)$$ for $1\leq i\leq n$ and $k\in\mbn$. Consequently, there is a surjective $\afuglk$-module homomorphism
$\vi:M(\bfQ)\ra \sG(L(\ti\bfQ))$. It follows that $\sG(L(\ti\bfQ))\cong M(\bfQ)/\ker\vi$ as a $\afuglk$-module. By Proposition \ref{the condition of eL(bfQ)=0} we have $\sG(L(\ti\bfQ))\not=0$. Hence by \cite{Gr80}, $\sG(L(\ti\bfQ))$ is an irreducible $\afuglk$-module. Therefore, since
$L(\bfQ)$ is the unique irreducible quotient of $M(\bfQ)$, we have $\sG(L(\ti\bfQ))$ is isomorphic to $L(\bfQ)$ as a $\afuglk$-module.
\end{proof}

\subsection{Classification of finite-dimensional irreducible $\afSrk$-modules}
We first treat the case \(n>r\).
\begin{Lem}\label{n>r case}
Assume that $n>r$. Then the set $\{L(\bfQ)\mid\bfQ\in\Qnr\}$ forms a complete set of pairwise non-isomorphic finite-dimensional irreducible $\afSrk$-modules.
\end{Lem}
\begin{proof}
Let $V$ be a finite-dimensional irreducible $\afSrk$-module.
By \cite[Th. 8.2]{Lu99} we have
\begin{equation}\label{image of afuslk}
\afxrk(\afuslk)=\afSrk.
\end{equation}
 Hence $V$ is also a finite-dimensional irreducible
$\afuslk$-module. By \cite[Cor. 3.5]{JM07}, $V$ is isomorphic to $\bar L(\bfP)$ as $\afuslk$-modules for some $\bfP\in\Pn$. Hence there exists some $w_0\not=0\in V$ such that
$$\bigg({H_{i}-H_{i+1}\atop k}\bigg)w_0
=\bigg({\mu_i\atop k}\bigg)w_0$$
for $1\leq i\leq n-1$ and $k>0$, where $\mu_i=\deg P_i(u)$.
Writing
$w_0=\sum_{\bt\in\afLanr}e_{\diag(\bt)}w_0,$
we find $e_{\diag(\la)}w_0\not=0$ for some $\la\in\afLanr$.
Furthermore, for $1\leq i\leq n-1$ and $k>0$
we have
$$\bigg({\mu_i\atop k}\bigg)w_0=\bigg({H_{i}-H_{i+1}\atop k}\bigg)w_0=\sum_{\bt\in\afLanr}
\bigg({\bt_{i}-\bt_{i+1}\atop k}\bigg)e_{\diag(\bt)}w_0$$
since $\afxrk\big(\big({H_{i}-H_{i+1}\atop k}\big)\big)=\sum_{\bt\in\afLanr}
\big({\bt_{i}-\bt_{i+1}\atop k}\big)e_{\diag(\bt)}$.
This implies that
$$\bigg({\mu_i\atop k}\bigg)e_{\diag(\la)}w_0=e_{\diag(\la)}\sum_{\bt\in\afLanr}
\bigg({\bt_{i}-\bt_{i+1}\atop k}\bigg)e_{\diag(\bt)}w_0=
\bigg({\la_{i}-\la_{i+1}\atop k}\bigg)e_{\diag(\la)}w_0$$
for $1\leq i\leq n-1$, $k>0$.
Since $e_{\diag(\la)}w_0\not=0$ we have
$$\bigg({\la_{i}-\la_{i+1}\atop k}\bigg)1_\field=\bigg({\mu_i\atop k}\bigg)1_\field
$$
for $1\leq i\leq n-1$ and $k>0$. It follows that $\la_i-\la_{i+1}=\mu_i$ for $1\leq i\leq n-1$.
Since $n>r$ and $\la\in\afLanr$ we have $\la_n=0$. Hence we have
$$\la_i=\sum_{i\leq s\leq n-1}\mu_s$$ for $1\leq i\leq n-1$.
Define
$$Q_i(u)=\prod_{i\leq s\leq n-1}P_s(u)\text{ and }Q_n(u)=1$$
for $1\leq i\leq n-1$.
Then we have $\deg Q_n(u)=0=\la_n$ and $$\deg(Q_i(u))=\sum_{i\leq s\leq n-1}\deg(P_s(u))=\sum_{i\leq s\leq n-1}\mu_s=\la_i$$ for $1\leq i\leq n-1$.
It follows that $\bfQ=(Q_1(u),Q_2(u),\ldots,Q_n(u))\in\Qnr$. By Proposition \ref{decomposition of L(Q)} and \ref{L(bfQ) is a afSrk-module}, $L(\bfQ)$ is a $\afSrk$-module and
$L(\bfQ)|_{\afuslk}\cong\bar L(\bfP)\cong V|_{\afuslk}$.
Hence there is a $\afuslk$-module isomorphism $\vi: L(\bfQ)\ra V$.
Therefore, by \eqref{image of afuslk} we conclude that $\vi$ must be a $\afSrk$-module isomorphism. On the other hand by Proposition \ref{L(bfQ) is a afSrk-module}, we know that
$L(\bfQ')$ is an irreducible $\afSrk$-module for $\bfQ'\in\Qnr$. The proof is completed.
\end{proof}

We now extend the classification in Lemma~\ref{n>r case} to the general case
using the Schur functor.

\begin{Thm}\label{classification}
The set $\{L(\bfQ)\mid\bfQ\in\Qnr\}$ forms a complete set of non-isomorphic finite-dimensional irreducible $\afSrk$-modules.
\end{Thm}
\begin{proof}
Choose $N>\max\{n,r\}$. By Lemma \ref{n>r case}
the set $\{L(\bfQ)\mid\bfQ\in\QNr\}$ is a complete set of non-isomorphic finite-dimensional irreducible $\afS{N,r}$-modules. This together with \cite[(6.2g)]{Gr80} implies that the set
$$\sX_r:=\{\sG(L(\bfQ))\not=0\mid\bfQ\in\QNr\}$$
forms a complete set of non-isomorphic finite-dimensional irreducible $\afSrk$-modules, where the functor $\sG$ is defined in \eqref{the functor G}. By Proposition \ref{the condition of eL(bfQ)=0} and \ref{properties of the functor G}
we have $$\sX_r=\{\sG(L(\ti\bfQ))\mid\bfQ\in\Qnr\}=\{L(\bfQ)\mid\bfQ\in\Qnr\}.$$ The proof is completed.
\end{proof}

\subsection{Character of irreducible modules }
With the classification in hand, we now investigate the character theory of these irreducible modules.

Combining Proposition \ref{decomposition of L(Q)} and Theorem \ref{classification}, we obtain the following result.
\begin{Thm}\label{decomposition}
Each finite-dimensional irreducible $\afSrk$-module is isomorphic to a tensor product of evaluation modules.
\end{Thm}

For $V\in\ms C_n$ the character of $V$ is an element of the polynomial ring $\mbz[X_1,X_2,\cdots,X_n]$ given by $$\chi(V)=\sum_{\la\in\mbnn}\dim V_\la X^\la,$$
where $X^\la=X_1^{\la_1}\cdots X_n^{\la_n}$.

Let $V\in\h{\ms C}_n$. Let $v_1,\cdots,v_m$ be a basis of common generalized eigenvectors
of $\{\big({H_i\atop k}\big),\,\La_i^\pm(u)\}$ in $V$. Then the generalized eigenvalues of $\La^\pm_i(u)$ on each vector $v_s$ are polynomials $Q_{i,s}^\pm(u)=\prod_{1\leq k\leq\la_{i,s}}(1-a_{i,s,k}^{\pm 1}u)$ and we have $\big({H_i\atop k}\big)v_s=
\big({\la_{i,s}\atop k}\big)v_s$.  The character of $V$, denoted by $\chi(V)$, is an element of the polynomial ring $\mbz[X_{i,a}\mid 1\leq i\leq n,\,a\in\mpk^*]$, given by
$$\chi(V)=\sum_{1\leq s\leq m}\prod_{1\leq i\leq n}\prod_{1\leq k\leq\la_{i,s}}X_{i,a_{i,s,k}}.$$

For $a\in\mpk^*$, we define an injective ring homomorphism
$$\iota_a:\mbz[X_1,X_2,\cdots,X_n]\ra\mbz[X_{i,a}\mid 1\leq i\leq n,\,a\in\mpk^*]$$ such that
$\iota_a(X_i)=X_{i,a}$ for $1\leq i\leq n$.
\begin{Thm}\label{character}
Let $\bfQ\in\Qnr$ and write $\bfQ=\prod_{1\leq s\leq m}\bfQ_{\la^{(s)},a_s}$, where
$\{a_s^{-1} \mid 1\leq s\leq m\}$ is the set of distinct roots of $\prod_{1\leq i\leq n}Q_i(u)$ and
$\la^{(s)}\in\La^+(n,r_s)$ with $r=\sum_{1\leq s\leq m}r_s$. Then we have
$$\chi(L(\bfQ))=\prod_{1\leq s\leq m}\iota_{a_s}\chi(L(\la^{(s)})).$$
In particular we have $\dim(L(\bfQ))=\prod_{1\leq s\leq m}\dim(L(\la^{(s)}))$.
\end{Thm}
\begin{proof}
If $w\in (L(\la)_a)_\mu$ then we have $\La_{i,\pm k}w=(-a)^{\pm k}\big({\mu_i\atop k}\big)w$ for $1\leq i\leq n$, $k\in\mbn$, and hence $\La_i^\pm(u)w=(1-a^{\pm 1}u)^{\mu_i}w$. It follows that $$\chi(L(\la)_a)=\sum_{\mu\in\mbnn}\dim L(\la)_\mu\prod_{1\leq i\leq n}X_{i,a}^{\mu_i}=\iota_a(\chi(L(\la))).$$ Now the assertion follows from Proposition \ref{decomposition of L(Q)}.
\end{proof}

\subsection{Tensor realization of irreducible modules}

Recall from \S \ref{Ognk definiton} $\Ognk$ is a $\field$-submodule of $\Ogk$ spanned by the set $\{\og_i\mid 1\leq i\leq n\}$. Recall the evaluation map $\ti\ev_a:\afSrk\ra\Srk$ defined in Lemma \ref{tieva}.
The pullback of $\Ognk^{\ot r}$ by $\ti\ev_a$ is an $\afSrk$-module which we denote by $(\Ognk^{\ot r})_a$.
For $a\in\field^*$, let $$M_a=\field\affSr/J_a,$$ where $J_a$ is the left ideal of
$\field\affSr$ generated by the elements $\bse_i-a^{-1}$ for $1\leq i\leq r$.

\begin{Lem}\label{Ogkr tensor Ma}
For $a\in\field^*$, there is an $\afSrk$-module isomorphism
$$\varrho_a:(\Ognk^{\ot r})_a\ra\Ogk^{\ot r}\ot_{\field\affSr} M_a
$$ defined by sending $\og_\bfi$ to $\og_\bfi\ot\bar 1$
for $\bfi\in I(n,r)$,
where $\bar 1=1+J_a\in M_a$.
\end{Lem}
\begin{proof}
By \cite[Lem. 4.1]{DDF} we have $\Ogk^{\ot r}\cong\Ognk^{\ot r}\ot_{\field\fSr}\field\affSr$.
Furthermore, $M_a$ is isomorphic to $\field\fSr$ as a left $\fSr$-module. Therefore
$\varrho_a$ is a linear isomorphism. By the algebra homomorphism $\afxrk$ defined by
\eqref{afxrk}, every $\afSrk$-module can be regarded as a $\afuglk$-module.
It remains to prove that $\varrho_a$ is a $\afuglk$-module homomorphism.
Clearly we have
$$\varrho_a\bigg(\bigg({H_i\atop t}\bigg)\og_\bfi\bigg)=\bigg({H_i\atop t}\bigg)\varrho_a(\og_\bfi)$$
for $1\leq i\leq n$, $t\in\mbn$ and $\bfi\in I(n,r)$. Furthermore by \eqref{Laikl(ogj)} we have
$$\La_{i,\pm k,l}(\og_j)=
\begin{cases}
(-1)^k\big({H_i\atop k}\big)\og_{j\mp kln}&\text{if $0\leq k\leq 1$,}\\
0&\text{if $k\geq 2$}
\end{cases}
$$
for $1\leq i,j\leq n$ and $l>0$. Hence by Corollary \ref{Deltar(Laikl)} and Lemma \ref{Deltar(Hi k)}
we have
\[
\begin{split}
\La_{i,\pm k,l}(\og_\bfi\ot\bar 1)&=
\sum_{k_1+k_2+\cdots+k_r=k\atop
0\leq k_s\leq 1,\,\forall s}(-1)^k
\bigg(\bigg({H_i\atop k_1}\bigg)\ot\cdots\ot
\bigg({H_i\atop k_r}\bigg)\bigg)(\og_{i_1\mp k_1ln}\ot\cdots\ot\og_{i_r\mp k_rln}\ot\bar 1)
\\
&=(-1)^ka^{\pm kl}\bigg({H_i\atop k}\bigg)\og_\bfi\ot\bar 1
\end{split}
\]
for $1\leq i\leq n$, $k\in\mbn$, $l>0$ and $\bfi\in I(n,r)$.
By Lemma \ref{Ev(La{i,k,l})} and Proposition \ref{commute diagram}
we have
$$\ti\ev_a\circ\afxrk(\La_{i,\pm k,l})=(-1)^ka^{\pm kl}\xrk\bigg(
\bigg({H_i\atop k}\bigg)\bigg)$$
for $1\leq i\leq n$, $k\in\mbn$ and $l>0$.
Therefore we have
$$\varrho_a (\La_{i,\pm k,l}\og_\bfi)= \La_{i,\pm k,l}(\varrho_a (\og_\bfi))$$
for $1\leq i\leq n$, $k\in\mbn$ and $l>0$.
By \eqref{E{i,j+ln}^{(m)} acts on OG} we have
\[
\begin{split}
(\afE_{i,j+ln})^{(k)}(\og_{\bfi}\ot\bar 1)
&=\sum_{k_1+k_2+\cdots+k_r=k\atop
0\leq k_s\leq 1,\,\forall s}((\afE_{i,j})^{(k_1)}\ot\cdots\ot
(\afE_{i,j})^{(k_r)})(\og_{i_1-lk_1n}\ot\cdots\og_{i_r-lk_rn}\ot\bar 1)\\
&=a^{kl}\sum_{k_1+k_2+\cdots+k_r=k\atop
0\leq k_s\leq 1,\,\forall s}((\afE_{i,j})^{(k_1)}\ot\cdots\ot
(\afE_{i,j})^{(k_r)})(\og_{\bfi}\ot\bar 1)\\
&=a^{kl}(\afE_{i,j})^{(k)}(\og_\bfi\ot\bar 1)
\end{split}
\]
for $1\leq i\not=j\leq n$, $l\in\mbz$, $k\in\mbn$ and $\bfi\in I(n,r)$.
So by Lemma \ref{Ev(E i,j+ln)} and Proposition \ref{commute diagram} we have
$$\varrho_a ((\afE_{i,j+ln})^{(k)}\og_\bfi)=(\afE_{i,j+ln})^{(k)}(\varrho_a (\og_\bfi))$$
for $1\leq i\not=j\leq n$, $l\in\mbz$ and $k\in\mbn$. Consequently,
by Theorem \ref{generators2} (1), we conclude that $\varrho_a$ is a $\afuglk$-module homomorphism. The proof is completed.
\end{proof}

\begin{Prop}\label{subquotien module of tensor space}
Let $V$ be a finite-dimensional irreducible $\afSrk$-module. Then $V$ is a subquotient module of
$\Ogk^{\ot r}$.
\end{Prop}
\begin{proof}
By Proposition \ref{decomposition of L(Q)} and Theorem \ref{classification}, there exists
$\la^{(i)}\in\La^+(n)$ and $a_i\in\field^*$ $(1\leq i\leq m)$ such that
$$V\cong L(\la^{(1)})_{a_1}\ot L(\la^{(2)})_{a_2}\ot\cdots\ot L(\la^{(m)})_{a_m}$$
and $r=\sum_{1\leq i\leq m}r_i$, where $r_i=\sum_{1\leq s\leq n}\la_s^{(i)}$.
It is well known that $L(\la^{(i)})$ is a subquotient module of $\Ognk^{\ot r_i}$ for $1\leq i\leq m$.
Hence  $L(\la^{(i)})_{a_i}$ is a subquotient module of $(\Ognk^{\ot r_i})_{a_i}$ for $1\leq i\leq m$. Therefore by Lemma \ref{Ogkr tensor Ma} we conclude that $V$ is a subquotient module of
$\Ogk^{\ot r}$.
\end{proof}

\subsection{Category equivalence}
We end this section by relating representations of $\afSrk$ and $\field\affSr$.
Let $\field\affSr\hmod$ be the category of all finite-dimensional $\field\affSr$-modules.
We define a functor
\begin{equation*}
\sfF_1:\field\affSr\hmod\lra\afSrk\hmod,\;L\longmapsto\Ogk^{\ot
r}\ot_{\field\affSr}L.
\end{equation*}
Assume that $n\geq r$. Let $\varpi=(\varpi_1,\cdots,\varpi_n)\in\afLanr$, where $\varpi_i=1$ if $1\leq i\leq r$, and $\varpi_{i}=0$ if $r<i\leq n$.
Then we have $e_{\diag(\varpi)}\afSrk e_{\diag(\varpi)}\cong\field\affSr$. We will identify
$e_{\diag(\varpi)}\afSrk e_{\diag(\varpi)}$ with $\field\affSr$.
We define the functor $\sF_2$ as follows:
\begin{equation*}
\sfF_2:\afSrk\hmod\lra\field\affSr\hmod,\;M\longmapsto e_{\diag(\varpi)}M.
\end{equation*}
The following result can be proved as in \cite[Th. 4.1.3]{DDF}.
\begin{Thm}\label{equivalence}
Assume that $n\geq r$ and $\text{char}\field=0$. Then we have $\sfF_1\circ\sfF_2\cong id_{\afSrk\hmod}$ and  $\sfF_2\circ\sfF_1\cong id_{\field\affSr\hmod}$. Hence the categories $\field\affSr{\text-}{\mathsf{mod}}$  and $\afSrk{\text-}{\mathsf{mod}}$ are equivalent.
\end{Thm}

Combining Theorems \ref{classification} and \ref{equivalence}, we obtain the following result.
\begin{Coro}
Assume that $n\geq r$ and $\text{char}\field=0$. Then the set $\{e_{\diag(\varpi)} L(\bfQ)\mid\bfQ\in\Qnr\}$
forms a complete set of  non-isomorphic finite-dimensional irreducible
$\field\affSr$-modules.
\end{Coro}
\begin{Rem}\label{Remark}
Assume that $n\geq r$ and $\text{char}\field>0$. By Theorem \ref{classification} and \cite[6.2(g)]{Gr80},  the set $\{e_{\diag(\varpi)} L(\bfQ)\not=0\mid\bfQ\in\Qnr\}$
forms a complete set of  non-isomorphic finite-dimensional irreducible
$\field\affSr$-modules.
It would be interesting to determine the necessary and sufficient conditions for $e_{\diag(\varpi)} L(\bfQ)$ to be nonzero.
Finite-dimensional irreducible modules for affine Hecke algebras of type $A$ in arbitrary
characteristic were classified in terms of  aperiodic multisegments in \cite{AM}. It would be interesting to investigate the relation between aperiodic multisegments and Drinfeld polynomials.
\end{Rem}

\section{Finite-dimensional representations of Lusztig algebras}

We maintain the assumptions from the previous section. In particular, $\field$ denotes an algebraically closed field.

\subsection{The Lusztig algebra $\afunrk$}
Following Lusztig \cite{Lu99}, let $\afunz$ be the $\mbz$-subalgebra of $\afuglz$  generated by the elements $(\afE_{i,j+ln})^{(k)}$ and $\big({H_i\atop k}\big)$ for $1\leq i\not=j\leq n$, $l\in\mbz$ and $k\in\mbn$. Let
$$\afunrz=\afxr(\afuslz)=\afxr(\afunz).$$ In \cite{Lu99}, Lusztig proved that
$\afunrz$ is a proper subalgebra of the affine Schur algebra $\afSrmbz$ in the case $n\leq r$ and constructed a canonical basis for $\afunrz$. The algebra $\afunrz$ is called Lusztig algebra. Let
\begin{equation}\label{afunk}
\afunk=\afunz\ot\field,\quad\afunrk=\afunrz\ot\field.
\end{equation}
By restriction, the surjective algebra homomorphism $\afxrk:\afuglk\ra\afSrk$ induces a surjective algebra homomorphism $\afxrk:\afunk\ra\afunrk$.

In this section we classify finite-dimensional irreducible $\afunrk$-modules, building upon our previous results about affine Schur algebras and establishing the precise relationship between these two algebraic structures.

Given a $\afunk$-module $V$ and $\mu\in\mbz^n$ we define the weight space $$V_\mu=\bigg\{w\in V\mid\bigg({H_i\atop k}\bigg)w
=\bigg({\mu_i\atop k}\bigg)w,\,\forall k\in\mbn,\,1\leq i\leq n\bigg\}.$$
Let $\bfP\in\Pn$ and $\la\in\La^+(n,r)$ be such that $\la_i-\la_{i+1}=\deg(P_i(u))$ for $1\leq i\leq n-1$. Let $\bar I(\bfP,\la)$
be the left ideal of $\afunk$ generated by $(\afE_{j,k+sn})^{(m)},\quad
\Ga_{j,s}-P_{j,s},$ and $\big({H_i\atop m}\big)- \big({\la_i\atop m}\big)$, for
$1\leq i\leq n$, $1\leq j<k\leq n$, $m>0$ and $s\in\mbz$, and define
$$\bar M(\bfP,\la)=\afunk/\bar I(\bfP,\la).$$
Then by Lemma \ref{Ejk Vmu} we have the weight space decomposition
$$\bar M(\bfP,\la)=\bop_{\mu\in\mbz^n\atop\mu\unlhd\la}\bar M(\bfP,\la)_\mu.$$
The $\afunk$-module $\bar M(\bfP,\la)$ has a unique irreducible quotient $\afunk$-module, denoted by $\bar L(\bfP,\la)$.

\subsection{Classification of finite-dimensional irreducible $\afunrk$-modules}
To classify finite-dimensional irreducible $\afunrk$-modules, we need to establish several key relationships between representations of different algebras. The following lemmas will help us connect the representation theories of $\afuglk$, $\afunk$, and $\afunrk$.

\begin{Lem}\label{restriction afuglk to afuk}
For $\bfQ\in\Qn$, let $\la=\deg\bfQ$ and
$\bfP\in\Pn$ be such that $P_j(u)= Q_j(u)/Q_{j+1}(u),$
for $1\leq j\leq n-1$. Then
we have $L(\bfQ)|_{\afunk}\cong\bar L(\bfP,\la)$ and $\bar L(\bfP,\la)|_{\afuslk}\cong\bar L(\bfP)$.
\end{Lem}
\begin{proof}
Let $w_0$ be a nonzero element in $L(\bfQ)_\la$.
Then we have
$$(\afE_{j,k+sn})^{(m)}w_0=0,\quad
\La_{i}^\pm(u)w_0=Q_{i}^\pm(u)w_0,\quad \bigg({H_i\atop m}\bigg)w_0=\bigg({\la_i\atop m}\bigg)w_0$$  for $1\leq i\leq n$, $1\leq j<k\leq n$, $m>0$ and $s\in\mbz$.
Since $\Ga_j^\pm(u)=\La_j^\pm(u)/\La_{j+1}^\pm(u)$ we have
$\Ga_j^\pm(u)w_0=P_j^\pm(u)w_0$ for $1\leq j\leq n$.
Hence, there is a $\afunk$-module homomorphism
$\vi:\Mcp(\bfP,\la)\ra L(\bfQ)$ such that $\vi(\bar 1)=w_0$, where
$\bar 1=1+\bar I(\bfP,\la)\in\Mcp(\bfP,\la)$. Furthermore,  by Proposition \ref{decomposition of L(Q)} we have
\begin{equation}\label{restriction of L(bfQ)}
L(\bfQ)|_{\afuslk}\cong\bar L(\bfP).
\end{equation}
It follows that $L(\bfQ)|_{\afunk}$ is irreducible. Therefore
$\vi$ is surjective and  $L(\bfQ)|_{\afunk}\cong \Mcp(\bfP,\la)/{\rm
Ker}\,\vi$. This together with the fact that
$\Lcp(\bfP,\la)$ is the unique irreducible quotient of $\Mcp(\bfP,\la)$
implies that $L(\bfQ)|_{\afunk}\cong\Lcp(\bfP,\la)$. Hence by \eqref{restriction of L(bfQ)}
we have $\Lcp(\bfP,\la)|_{\afuslk}\cong\Lcp(\bfP)$.
\end{proof}

Next, we show that the modules $\bar L(\bfP,\la)$ naturally carry an $\afunrk$-module structure.

\begin{Coro}\label{lem1 for rep of afurk}
Let $\bfP\in\Pn$ and $\la\in\La^+(n,r)$ with $\la_i-\la_{i+1}
=\mathrm{deg}P_i(u)$ for $1\leq i\leq n-1$. Then $\Lcp(\bfP,\la)$ is
a $\afunrk$-module via $\afxrk$.  Furthermore we have
\begin{equation}\label{weight space for Lcp(bfp,la)}
\Lcp(\bfP,\la)=\bop_{\mu\unlhd\la\atop\mu\in\La(n,r)}\Lcp(\bfP,\la)_\mu,
\end{equation}
where $\Lanr=\{\la\in\mbnn\mid\sum_{1\leq i\leq n}\la_i=r\}$.
\end{Coro}
\begin{proof}
Let $\bfQ=(Q_1(u),Q_2(u),\ldots,Q_n(u))\in\sQ(n)$, where $Q_n(u)=(1+u)^{\la_n}$ and
$$Q_i(u)=Q_n(u)\prod_{i\leq s\leq n-1}P_s(u)$$
for $1\leq i\leq n-1$.
By Proposition \ref{L(bfQ) is a afSrk-module}, $L(\bfQ)$ is an $\afSrk$-module. So, by Lemma
\ref{restriction afuglk to afuk}, $\Lcp(\bfP,\la)$ is a $\afunrk$-module.
Hence we have
$$\Lcp(\bfP,\la)=\bop_{\mu\in\La(n,r)}e_{\diag(\mu)}\Lcp(\bfP,\la),$$
and $\Lcp(\bfP,\la)_\mu=e_{\diag(\mu)}\Lcp(\bfP,\la)$
for $\mu\in\Lanr$.
The decomposition \eqref{weight space for Lcp(bfp,la)} now follows from Lemma \ref{Ejk Vmu}.
\end{proof}
This corollary shows that the irreducible $\afunk$-modules $\bar L(\bfP,\la)$ naturally descend to $\afunrk$-modules, and they inherit a nice weight space decomposition.
We now establish that these modules are non-isomorphic for different parameters.

\begin{Lem}\label{lem3 for rep of afurk}
For $\la,\la'\in\La^+(n)$ and $\bfP,\bfP'\in\Pn$,
if there is a $\afunk$-module isomorphism $\vi:\Lcp(\bfP,\la)\cong\Lcp(\bfP',\la')$, then $\bfP=\bfP'$
and $\la=\la'$.
\end{Lem}
\begin{proof}
By \eqref{weight space for Lcp(bfp,la)} we have $\la=\la'$. Let $w_0$ be a nonzero element in $\Lcp(\bfP,\la)_\la$. Then we have $\vi(w_0)\in\Lcp(\bfP',\la')_\la$. Hence
we have $$P_i(u)\vi(w_0)=\vi(\Ga_i^+(u)w_0)=\Ga_i^+(u)\vi(w_0)= P_i'(u)\vi(w_0)$$
for $1\leq i\leq n-1$. Therefore $\bfP=\bfP'$.
\end{proof}

Combining Lemma \ref{restriction afuglk to afuk} with \ref{lem3 for rep of afurk}, we obtain the following characterization of when restrictions of irreducible $\afuglk$-modules become isomorphic as $\afunk$-modules.
\begin{Coro}
For $\bfQ,\bfQ'\in\Qn$ we have $L(\bfQ)|_{\afunk}\cong L(\bfQ')|_{\afunk}$ if and only if
$\deg Q_i(u)=\deg Q_i'(u)$ and $Q_j(u)/Q_{j+1}(u)=Q_j'(u)/Q_{j+1}'(u)$ for $1\leq i\leq n$ and $1\leq j\leq n-1$.
\end{Coro}

This result clarifies the relationship between the parameterizations for $\afuglk$ and $\afunk$, showing exactly when different Drinfeld polynomials give rise to isomorphic restrictions. In particular, it demonstrates that the  irreducible $\afunk$-module structure only depends on the highest weight $\la = \deg\bfQ$ and the ratios $Q_j(u)/Q_{j+1}(u)$, rather than the individual Drinfeld polynomials $Q_i(u)$.

The next lemma completes the classification picture by showing that all finite-dimensional irreducible $\afunrk$-modules arise from the construction of $\bar L(\bfP,\la)$ described above.

\begin{Lem}\label{lem2 for rep of afurk}
Let $V$ be a finite-dimensional irreducible $\afunrk$-module. Then there
exist $\bfP\in\Pn$ and $\la\in\La^+(n,r)$ with $\la_i-\la_{i+1}
=\mathrm{deg}P_i(u)$, for all $1\leq i\leq n-1$, such that $V\cong
\Lcp(\bfP,\la)$.
\end{Lem}
\begin{proof}
Let $V$ be a finite-dimensional irreducible $\afunrk$-module.
Since
\begin{equation}\label{image of afuslk is equal to image of afunk}
\afxrk(\afuslk)=\afxrk(\afunk)=\afunrk,
\end{equation}
we see that $V$ is a finite-dimensional irreducible $\afuslk$-module. By \cite[Cor. 3.5]{JM07} we have $V|_{\afuslk}$ is isomorphic to $\bar L(\bfP)$ for some $\bfP\in\Pn$. Hence there exists some $w_0\not=0\in V$ such that
$$(\afE_{j,k+sn})^{(m)}w_0=0,\,\bigg({H_{i}-H_{i+1}\atop m}\bigg)w_0
=\bigg({\mu_i\atop m}\bigg)w_0,\,\Ga_{j}^\pm(u)w_0=P_j^\pm(u)w_0$$
for $1\leq j<k\leq n$, $s\in\mbz$ and $m>0$, where $\mu_i=\deg P_i(u)$.
Since $\afxrk\big(\big({H_{i}-H_{i+1}\atop m}\big)\big)=\sum_{\bt\in\afLanr}
\big({\bt_{i}-\bt_{i+1}\atop m}\big)e_{\diag(\bt)}$,
we have
$$\bigg({\mu_i\atop m}\bigg)w_0 =\sum_{\bt\in\afLanr}
\bigg({\bt_{i}-\bt_{i+1}\atop m}\bigg)e_{\diag(\bt)}w_0$$
for $1\leq i\leq n-1$ and $m>0$.
This implies that
$$\bigg({\mu_i\atop m}\bigg)e_{\diag(\bt)}w_0 =
\bigg({\bt_{i}-\bt_{i+1}\atop m}\bigg)e_{\diag(\bt)}w_0$$
for $\bt\in\afLanr$, $1\leq i\leq n-1$ and $m>0$.
Consequently, if $\nu\in\afLanr$ is such that $e_{\diag(\nu)}w_0\not=0$ then
$$\bigg({\nu_{i}-\nu_{i+1}\atop m}\bigg)1_\field=\bigg({\mu_i\atop m}\bigg)1_\field
$$
for $1\leq i\leq n-1$ and $m>0$, and hence $\nu_i-\nu_{i+1}=\mu_i$ for $1\leq i\leq n-1$.
Therefore, since
$w_0=\sum_{\bt\in\afLanr}e_{\diag(\bt)}w_0$ and $w_0\not=0$, there is a unique $\la\in\afLanr$ such that
$e_{\diag(\la)}w_0\not=0.$ It follows that $w_0=e_{\diag(\la)}w_0$. This implies that
$$\bigg({H_{i}\atop m}\bigg)w_0
=\bigg({\la_i\atop m}\bigg)w_0$$
for $1\leq i\leq n$. Hence there is a surjective $\afunk$-module homomorphism
$\vi:\bar M(\bfP,\la)\ra V$ such that $\vi(\bar 1)=w_0$, where $\bar 1=1+\bar I(\bfP,\la)\in\bar M(\bfP,\la)$. Therefore, since $\bar L(\bfP,\la)$ is the unique irreducible quotient of $\bar M(\bfP,\la)$, we have $V\cong\bar L(\bfP,\la)$.
\end{proof}

Combining Corollary \ref{lem1 for rep of afurk}, Lemma  \ref{lem3 for rep of afurk} and \ref{lem2 for rep of afurk}, we obtain the main classification theorem for finite-dimensional irreducible $\afunrk$-modules.
\begin{Thm}\label{simple U(n,r)-modules}
The set
$$
\{\Lcp(\bfP,\la)\mid\bfP\in\Pn,
\,\la\in\La^+(n,r),\,\la_i-\la_{i+1} =\mathrm{deg}P_i(u),\;\text{ for
$1\leq i\leq n-1$}\}
$$
is a complete set of non-isomorphic finite-dimensional irreducible
$\afunrk$-modules.
\end{Thm}

This theorem provides a complete classification of finite-dimensional irreducible representations of the Lusztig algebra $\afunrk$ in terms of Drinfeld polynomials and highest weights, establishing a parallel with the representation theory of the hyper loop algebra $\afuglk$ and affine Schur algebras.

\section{Polynomial representations of $\Gk$}
The aim of this section is to relate finite-dimensional
representations of the semigroup $\widehat G_{\mathpzc K}(n)$
to affine Schur algebras. In the affine setting, one has to
distinguish the polynomial coefficient space
$\widehat A_{\mathpzc K}(n,r)$ from its formal completion.
This distinction is already visible for evaluation modules:
their coefficient functions are, in general, infinite formal
linear combinations of the coordinate monomials, although their
values at each element of $\widehat G_{\mathpzc K}(n)$ are
given by finite sums.

A formal-coalgebra approach to affine Schur algebras was
developed by Yang \cite{Yang1} for a different family of
semigroups $\widetilde{GL}_{n,a}$ depending on an evaluation
parameter $a\in\mathpzc K^\times$. We shall compare his
semigroups with $\widehat G_{\mathpzc K}(n)$ below.
\subsection{The restricted dual and its formal completion}
Recall that $\field$ is a commutative ring with unity.
Let $\sAn$ be the associative $\field$-algebra generated by $X_{i,j}$
($i,j\in\mbz$) with relations $X_{i,j}X_{k,l}=X_{k,l}X_{i,j}$ and $X_{i,j}=X_{i+n,j+n}$
for $i,j,k,l\in\mbz$.
This is naturally graded by regarding each generator as having
degree $1$; denote by $\sAnr$ the $r$th graded component.

We first introduce some notation that will be used throughout this subsection.
Let
\begin{equation*}\label{afInr}
\afInr=\{\bfi=(i_k)_{k\in\mbz}\mid i_k\in\mbz,i_{k+r}=i_k+n\text{
for all $k\in\mbz$}\}.
\end{equation*}
We shall
identify $\afInr$ with $\mbz^r$ by the following bijection
\begin{equation*}\label{afInr4}
I_\vtg(n,r)\lra \mbz^r,\;\; \bfi\longmapsto (i_1,\ldots,i_r).
\end{equation*}
The group $\affSr$ acts
on the set $\afInr$ by place permutation:
\begin{equation*}\label{place permutation}
\bfi w=(i_{w(k)})_{k\in\mbz},\quad\text{
for $\bfi\in \afInr$ and $w\in\affSr$.}
\end{equation*}
The group $\affSr$ also acts on the set $\afInr\times\afInr$ by
$(\bfi,\bfj)w=(\bfi w,\bfj w)$ for $\bfi,\bfj\in\afInr$ and $w\in\affSr$.
We write $(\bfi,\bfj)\sim(\bfk,\bl)$ to indicate that the elements $(\bfi,\bfj)$ and $(\bfk,\bl)$ are in the same $\affSr$-orbit.
These group actions allow us to define monomials in $\sAnr$ in an orbit-invariant way.
For $\bfi,\bfj\in\afInr$ let
$$X_{\bfi,\bfj}=X_{i_1,j_1}X_{i_2,j_2}\cdots X_{i_r,j_r}\in\sAnr.$$
For $\bfi,\bfj,\bfk,\bl\in\afInr$ we have $X_{\bfi,\bfj}=X_{\bfk,\bl}$ if and only if
$(\bfi,\bfj)\sim(\bfk,\bl)$.

To study the dual picture, we define linear functionals on $\sAnr$.
For $\bfi,\bfj\in\afInr$ let $X_{\bfi,\bfj}^*$ be the element of $\sAnr^*=\Hom_\mpk(\sAnr,\mpk)$ such that
$$X_{\bfi,\bfj}^*(X_{\bfk,\bl})=
\begin{cases}
1&\text{if $(\bfi,\bfj)\sim(\bfk,\bl)$}\\
0&\text{if $(\bfi,\bfj)\not\sim(\bfk,\bl)$}
\end{cases},\text{ all $\bfk,\bl\in\afInr$}.
$$

Let $\sAnr^\#$ be the $\mpk$-submodule of $\sAnr^*$ spanned by $X_{\bfi,\bfj}^*$ for $\bfi,\bfj\in\afInr$. For $f,g\in\sAnr^\#$ and $\bfi,\bfj\in\afInr$,
there are only finitely many $\bfk\in\afInr$ such that $f(X_{\bfi,\bfk})g(X_{\bfk,\bfj})\not=0$. Hence we can define $fg\in\sAnr^*$
such that
$$(fg)(X_{\bfi,\bfj})=\sum_{\bfk\in\afInr}f(X_{\bfi,\bfk})g(X_{\bfk,\bfj})$$
for $\bfi,\bfj\in\afInr$. It is easy to see that $fg\in\sAnr^\#$. This defines an associative algebra structure on  $\sAnr^\#$.

We next choose canonical representatives for the
\(\widehat{\mathfrak S}_r\)-orbits.
For $\la\in\afLanr$ let
$$\bfi_\la^\vtg=(\ldots,\underbrace{0,\ldots,0}_{\la_n},
\underbrace{1,\ldots,1}_{\la_1},\underbrace{2,\ldots,2}_{\la_2},
\ldots,\underbrace{n,\ldots,n}_{\la_n},\underbrace{1+n,\ldots,1+n}_{\la_1},
\underbrace{2+n,\ldots,2+n}_{\la_2},\ldots).$$
\begin{Lem}\label{representative}
For $\bfi,\bfj\in\afInr$ there exist unique $\la,\mu\in\afLanr$ and $d\in\afmsD_{\la,\mu}$ such that $(\bfi,\bfj) \sim(\bfi^\vtg_\la d,\bfi^\vtg_\mu)$.
\end{Lem}
\begin{proof}
By \cite[3.2.1.3]{DDF}, for $\bfi,\bfj\in\afInr$ there exist $\la,\mu\in\afLanr$ and $w\in\affSr$ such that $(\bfi,\bfj)\sim(\bfi^\vtg_\la w,\bfi^\vtg_\mu)$.
Let $d$ be a representative of $\fS_\la w\fS_\mu$ with minimal length. Then we have $d\in\msD_{\la,\mu}$ and
$(\bfi,\bfj)\sim(\bfi^\vtg_\la w,\bfi^\vtg_\mu)\sim(\bfi^\vtg_\la d,\bfi^\vtg_\mu)$. Assume now that
$\al,\bt,\ga,\dt\in\afLanr$, $d_1\in\msD_{\al,\bt}$, $d_2\in\msD_{\ga,\dt}$ and
$(\bfi^\vtg_\al d_1,\bfi^\vtg_\bt )\sim(\bfi^\vtg_\ga d_2,\bfi^\vtg_\dt)$.
By  \cite[3.2.1.3]{DDF}
we have $\al=\ga$ and $\bt=\dt$. Hence there exists $w\in\affSr$ such that $\bfi^\vtg_\al d_1 =\bfi^\vtg_\al d_2w$ and $\bfi_\bt^\vtg=\bfi_\bt^\vtg w$. It follows that $\fS_\al d_1=\fS_\al d_2w$ and $w\in\fS_\bt$. This implies that $\fS_\al d_1\fS_\bt=\fS_\al d_2\fS_\bt$. Therefore, since $d_1,d_2\in\fS_{\al,\bt}$, we have $d_1=d_2$.
\end{proof}

By Lemma \ref{representative} we obtain the following result.
\begin{Coro}\label{basis of sAnr}
The set $\{X_{\bfi_\la^\vtg d,\bfi_\mu^\vtg} \mid\la,\mu\in\afLanr,\,d\in\afmsD_{\la,\mu}\}$ forms a $\mpk$-basis for $\sAnr$. In particular the set $\{X_{\bfi_\la^\vtg d,\bfi^\vtg_\mu}^*\mid\la,\mu\in\afLanr,\,d\in\afmsD_{\la,\mu}\}$ forms a $\mpk$-basis for $\sAnr^\#$.
\end{Coro}

Using the preceding basis of  $\sAnr^\#$, we obtain the following isomorphism.

\begin{Thm}\label{res dual}
There is a $\mpk$-algebra isomorphism $\phi_r:\sAnr^\#\ra\afSrk$ such that
$\phi_r(X_{\bfi_\la^\vtg d,\bfi_\mu^\vtg}^*)=e_A$
for $\la,\mu\in\afLanr$ and $d\in\afmsD_{\la,\mu}$, where $A=\afjmath(\la,d,\mu)$.
\end{Thm}
\begin{proof}
We have an $\sAnr^\#$-module structure on $\Ogk^{\ot r}$ given by
$$f(\og_\bfi)=\sum_{\bfj\in\afInr}f(X_{\bfj,\bfi})\og_\bfj$$
for $f\in\sAnr^\#$ and $\bfi\in\afInr$.
Clearly the action of $\sAnr^\#$ and $\affSr$ on $\Ogk^{\ot r}$ commute.
The $\sAnr^\#$-$\affSr$-module structure on $\Ogk^{\ot r}$ induces a $\mpk$-algebra homomorphism
$$\phi_r:\sAnr^\#\lra\afSrk.$$
For $\la,\mu,\al\in\afLanr$ and $d\in\afmsD_{\la,\mu}$
we have
\[
 X_{\bfi_\la^\vtg d,\bfi_\mu^\vtg}^*(\og_{\bfi_\al^\vtg})
 =\sum_{\bfj\in\afInr}X_{\bfi_\la^\vtg d,\bfi_\mu^\vtg}^*(X_{\bfj,\bfi_\al^\vtg})\og_\bfj =\dt_{\mu,\al}\sum_{\bfj\in\afInr
\atop(\bfi_\la^\vtg d,\bfi_\mu^\vtg)\sim(\bfj,\bfi_\mu^\vtg)} \og_\bfj
 =\dt_{\mu,\al}\sum_{w\in\fS_\mu\cap\afmsD_\nu}\og_{\bfi_\la^\vtg dw}.
\]
where $\fS_\nu=d^{-1}\fS_\la d\cap\fS_\mu$. It follows that
$\phi_r(X_{\bfi_\la^\vtg d,\bfi_\mu^\vtg}^*)=e_A$
for $\la,\mu\in\afLanr$ and $d\in\afmsD_{\la,\mu}$, where $A=\afjmath(\la,d,\mu)$.
Therefore by Corollary \ref{basis of sAnr} we conclude that $\phi_r$ is a $\mpk$-algebra isomorphism.
\end{proof}

We define the formal completion  of
$\widehat \sA_{\mathpzc K}(n,r)$ to be the $\mathpzc K$-module
\[
\overline{\widehat\sA_{\mathpzc K}(n,r)}
=
\bigg\{
\sum_{\substack{\lambda,\mu\in\Lambda_\vtg(n,r)\\
d\in\mathscr D_{\lambda,\mu}^\vtg}}
a_{\lambda,d,\mu}
X_{\mathbf{i}_\lambda^\vtg d,\mathbf{i}_\mu^\vtg}
\ \big  |\
a_{\lambda,d,\mu}\in\mathpzc K
\bigg\},
\]
where the sums are formal and no finiteness condition is imposed
on the coefficients $a_{\lambda,d,\mu}$.
By Theorem~\ref{res dual}
and the definition above, there is a
$\mathpzc K$-module isomorphism
\begin{equation}\label{S*}
\overline{\widehat \sA_{\mathpzc K}(n,r)}
\cong
\widehat{\mathcal S}_{\mathpzc K}(n,r)^*,
\end{equation}
under which
\(
\sum_{\lambda,\mu,d}
a_{\lambda,d,\mu}
X_{\mathbf{i}_\lambda^\vtg d,\mathbf{i}_\mu^\vtg}
\)
corresponds to the linear functional sending
$e_{\jmath_\vtg(\lambda,d,\mu)}$ to
$a_{\lambda,d,\mu}$, where
\(
\widehat{\mathcal S}_{\mathpzc K}(n,r)^*
=
\operatorname{Hom}_{\mathpzc K}
\bigl(\widehat{\mathcal S}_{\mathpzc K}(n,r),\mathpzc K\bigr).
\)

\subsection{Formal polynomial representations of
$\widehat G_{\mathpzc K}(n)$}
Throughout this subsection, we assume that $\mathpzc K$ is an
infinite field.
Recall from  \eqref{the definition of G} that $\Gk=\{A\in \afGL\mid a_{i,j}\in\field[\vep,\vep^{-1}],\,\forall i,j\}$.
The set $\field^{\Gk}$  of all functions $f:\Gk\ra\field$ is a commutative $\field$-algebra, with algebra operations defined ``pointwise'', e.g. the product $f_1f_2$ of two functions $f_1,f_2:\Gk\ra\field$ is defined to be the function which takes $g\ra f_1(g)f_2(g)$, for all $g\in \Gk$.

For $i,j\in\mbz$, let $\ttx_{i,j}\in\field^{\Gk}$ be the function which assigns to each $g\in \Gk$ the $(i,j)$-coefficient of $\vartheta(g)$, where $\vartheta$ is given in \eqref{vartheta}. Then we have
$\ttx_{i,j}=\ttx_{i+n,j+n}$ for $i,j\in\mbz$.
Let $\An$ be the $\field$-subalgebra of $\field^{\Gk}$ generated by the functions $\ttx_{i,j}$ ($i,j\in\mbz$). The elements of $\An$ are called the polynomial functions on $\Gk$. For $r\in\mbn$ let $\Anr$ be the subspace of $\An$ consisting of all polynomial functions on $\Gk$
which are expressible as homogeneous  polynomials of degree $r$ in the $\ttx_{i,j}$.

Since $\field$ is an infinite field,
the functions $\ttx_{i,j}$ ($1\leq i\leq n$, $j\in\mbz$) are algebraically independent over $\field$.
Thus $\An$ is isomorphic to the algebra $\sAn$. We henceforth identify
$\widehat{\mathcal A}_{\mathpzc K}(n)$ with
$\widehat A_{\mathpzc K}(n)$, and likewise identify
$\widehat{\mathcal A}_{\mathpzc K}(n,r)$ with
$\widehat A_{\mathpzc K}(n,r)$.

We next pass from polynomial functions to formal polynomial
functions. Let
\[
F=
\sum_{\substack{\lambda,\mu\in\Lambda_\vtg(n,r)\\
d\in\mathscr D_{\lambda,\mu}^\vtg}}
a_{\lambda,d,\mu}
X_{\mathbf{i}_\lambda^\vtg d,\mathbf{i}_\mu^\vtg}
\in
\overline{\widehat{\mathcal A}_{\mathpzc K}(n,r)}.
\]
For $g\in\widehat G_{\mathpzc K}(n)$, define
\[
F(g)=
\sum_{\substack{\lambda,\mu\in\afLanr\\
d\in\mathscr D_{\lambda,\mu}^\vtg}}
a_{\lambda,d,\mu}
X_{\mathbf{i}_\lambda^\vtg d,\mathbf{i}_\mu^\vtg}(g),
\]
where each $X_{i,j}$ is evaluated as the corresponding
coordinate function $\ttx_{i,j}$.
This is well defined. Indeed, for a fixed
$g\in\widehat G_{\mathpzc K}(n)$, only finitely many of the
coordinate functions $\ttx_{i,j}$ with $1\leq i\leq n$ and
$j\in\mathbb Z$ have a nonzero value at $g$, since every entry
of $g$ is a Laurent polynomial. Since the basis elements of
$\widehat{\mathcal A}_{\mathpzc K}(n,r)$ are monomials of
degree $r$, only finitely many terms of the above formal sum
are nonzero at $g$.

Thus evaluation defines a linear map
\[
\overline{\Psi}_r:
\overline{\widehat{\mathcal A}_{\mathpzc K}(n,r)}
\longrightarrow
\mathpzc K^{\widehat G_{\mathpzc K}(n)}.
\]

\begin{Lem}
The map $\overline{\Psi}_r$ is injective.
\end{Lem}

\begin{proof}
Let
\(
F\in
\overline{\widehat{\mathcal A}_{\mathpzc K}(n,r)},
\)
and let
\(
\widetilde F\in
\widehat{\mathcal S}_{\mathpzc K}(n,r)^*
\)
be the corresponding linear functional under the isomorphism given in \eqref{S*}.
For $g\in\widehat G_{\mathpzc K}(n)$, let
\[
\operatorname{ev}_g:
\widehat{\mathcal A}_{\mathpzc K}(n,r)\longrightarrow\mathpzc K,
\qquad
F\longmapsto F(g).
\]
Then we have
$\operatorname{ev}_g\in
\widehat{\mathcal A}_{\mathpzc K}(n,r)^\#$.
By the definition of $\phi_r$,
\(
\phi_r(\operatorname{ev}_g)
=
\widehat\zeta_{r,\mathpzc K}(g).
\)
Consequently, Theorem~\ref{res dual} gives
\[
\widehat\zeta_{r,\mathpzc K}(g)
=
\sum_{\substack{\lambda,\mu\in\Lambda_\vtg(n,r)\\
d\in\mathscr D_{\lambda,\mu}^\vtg}}
X_{\mathbf{i}_\lambda^\vtg d,\mathbf{i}_\mu^\vtg}(g)
e_{\jmath_\vtg(\lambda,d,\mu)}.
\]
It follows that
\begin{equation}\label{ti F}
F(g)
=
\widetilde F
\bigl(\widehat\zeta_{r,\mathpzc K}(g)\bigr).
\end{equation}
Suppose that $\overline{\Psi}_r(F)=0$. Then
\(
\widetilde F
\bigl(\widehat\zeta_{r,\mathpzc K}(g)\bigr)=0
\
\text{for all }g\in\widehat G_{\mathpzc K}(n).
\)
Since $\mathpzc K$ is infinite, Theorem~\ref{surjective2} implies that
$\widehat\zeta_{r,\mathpzc K}$ is surjective. Therefore
$\widetilde F=0$, and hence $F=0$.
\end{proof}

We denote the image of $\overline{\Psi}_r$ by
\(
\overline{\widehat A_{\mathpzc K}(n,r)}.
\)
Thus
\[
\overline{\Psi}_r:
\overline{\widehat{\mathcal A}_{\mathpzc K}(n,r)}
\xrightarrow{\ \sim\ }
\overline{\widehat A_{\mathpzc K}(n,r)}
\]
is a $\mathpzc K$-linear isomorphism. In particular,
\(
\widehat A_{\mathpzc K}(n,r)
\subseteq
\overline{\widehat A_{\mathpzc K}(n,r)}
\subseteq
\mathpzc K^{\widehat G_{\mathpzc K}(n)}.
\)
By  \eqref{S*}, we also have a
$\mathpzc K$-linear isomorphism
\begin{equation}\label{S*'}
\overline{\widehat A_{\mathpzc K}(n,r)}
\cong
\widehat{\mathcal S}_{\mathpzc K}(n,r)^*.
\end{equation}

Let $V$ be a  finite-dimensional $\field \Gk$-module. Let $\{v_i\mid i\in B \}$ be a $\field$-basis of $V$. We write
$$gv_j=\sum_{i\in B}c_{i,j}(g)v_i$$
for $g\in \Gk$ and $j\in B$, where $c_{i,j}(g)\in\field$. The functions $c_{i,j}:\Gk\ra\field$ $(i,j\in B)$ are called coefficient functions of the $\field \Gk$-module $V$. Let $\cf(V)$ be the $\field$-span of the functions $c_{i,j}$ for $i,j\in B$. It is easy to see that $\cf(V)$ is independent of the choice of the basis $\{v_i\mid i\in B\}$.

We call $V$ a  {  polynomial
$\mathpzc K\widehat G_{\mathpzc K}(n)$-module of homogeneous
degree $r$} if
\(
\operatorname{cf}(V)
\subseteq
\overline{\widehat A_{\mathpzc K}(n,r)}.
\)
Let
\(
\widehat{\mathcal P}_{\mathpzc K}(n,r)
\)
denote the category of finite-dimensional polynomial
$\mathpzc K\widehat G_{\mathpzc K}(n)$-modules of homogeneous
degree $r$.

\begin{Thm}\label{equivalence1}
There is an equivalence of categories
\[
\widehat{\mathcal S}_{\mathpzc K}(n,r)\text{-mod}
\ \simeq\
\widehat{\mathcal P}_{\mathpzc K}(n,r),
\]
where $\widehat{\mathcal S}_{\mathpzc K}(n,r)\text{-mod}$ denotes the
category of finite-dimensional
$\widehat{\mathcal S}_{\mathpzc K}(n,r)$-modules.
\end{Thm}

\begin{proof}
Let $V$ be a finite-dimensional
$\widehat{\mathcal S}_{\mathpzc K}(n,r)$-module.
Let $\{v_i\mid i\in I\}$ be a $\mathpzc K$-basis of $V$.
For $i,j\in I$, define
$\lambda_{i,j}\in\widehat{\mathcal S}_{\mathpzc K}(n,r)^*$ by
\[
sv_j=\sum_{i\in I}\lambda_{i,j}(s)v_i
\qquad
(s\in\widehat{\mathcal S}_{\mathpzc K}(n,r)).
\]
Let $F_{i,j}\in
\overline{\widehat{A}_{\mathpzc K}(n,r)}$
correspond to $\lambda_{i,j}$ under the identification
\(
\overline{\widehat{A}_{\mathpzc K}(n,r)}
\cong
\widehat{\mathcal S}_{\mathpzc K}(n,r)^*
\) given in \eqref{S*'}.
Regard $V$ as a
$\mathpzc K\widehat G_{\mathpzc K}(n)$-module via
$\widehat\zeta_{r,\mathpzc K}$, and write
\[
gv_j=\sum_{i\in I}c_{i,j}(g)v_i
\qquad
(g\in\widehat G_{\mathpzc K}(n)).
\]
Then by \eqref{ti F} we have
\(
c_{i,j}(g)
=
\lambda_{i,j}
\bigl(\widehat\zeta_{r,\mathpzc K}(g)\bigr)
=
F_{i,j}(g).
\)
Thus
\(
\operatorname{cf}(V)
\subseteq
\overline{\widehat A_{\mathpzc K}(n,r)},
\)
and $V\in \widehat{\mathcal P}_{\mathpzc K}(n,r)$.

Conversely, let
$V\in\widehat{\mathcal P}_{\mathpzc K}(n,r)$.
For each coefficient function $c_{i,j}$ of $V$, we have
\(
c_{i,j}\in
\overline{\widehat A_{\mathpzc K}(n,r)}.
\)
Using the identification
\(
\overline{\widehat A_{\mathpzc K}(n,r)}
\cong
\widehat{\mathcal S}_{\mathpzc K}(n,r)^*,
\)
let
\(
\lambda_{i,j}\in
\widehat{\mathcal S}_{\mathpzc K}(n,r)^*
\)
be the corresponding linear functional.
 Then by \eqref{ti F} we have
\(
c_{i,j}(g)
=
\lambda_{i,j}
\bigl(\widehat\zeta_{r,\mathpzc K}(g)\bigr)
\)
for all $g\in\widehat G_{\mathpzc K}(n)$.
If $u\in\ker\widehat\zeta_{r,\mathpzc K}$, then
\(
c_{i,j}(u)
=
\lambda_{i,j}
\bigl(\widehat\zeta_{r,\mathpzc K}(u)\bigr)
=
0
\)
for all $i,j\in I$. Hence
\(
uV=0.
\)
By  Theorem~\ref{surjective2},
$\widehat\zeta_{r,\mathpzc K}$ is surjective since $\mathpzc K$ is  infinite.
Therefore the action of
$\mathpzc K\widehat G_{\mathpzc K}(n)$ on $V$ factors uniquely
through $\widehat{\mathcal S}_{\mathpzc K}(n,r)$.
\end{proof}

Combining Theorems \ref{classification} and \ref{equivalence1} yields the
following consequence.

\begin{Thm}\label{classification Mnr}
Assume that $\field$ is an algebraically closed field. Then the set
\[
\left\{
L(\mathbf Q)
\ \middle|\
\mathbf Q\in\mathcal Q_{\mathpzc K}(n,r)
\right\}
\]
forms a complete set of pairwise non-isomorphic irreducible
objects of
$\widehat{\mathcal P}_{\mathpzc K}(n,r)$.
\end{Thm}

 \subsection{Polynomial representations over finite fields}

Let $q=p^f>r$ and let $R$ be an infinite field containing
$\mathbb F_q$. We have a natural inclusion
\(
\widehat G_{\mathbb F_q}(n)
\subseteq
\widehat G_R(n).
\)

A finite-dimensional
$R\widehat G_{\mathbb F_q}(n)$-module $V$ is called a
 {polynomial $R\widehat G_{\mathbb F_q}(n)$-module of
homogeneous degree $r$} if
\(
\operatorname{cf}(V)
\subseteq
\operatorname{res}_{q,R}
\left(
\overline{\widehat A_R(n,r)}
\right),
\)
where
\[
\operatorname{res}_{q,R}:
R^{\widehat G_R(n)}
\longrightarrow
R^{\widehat G_{\mathbb F_q}(n)}
\]
is the restriction map of functions.
We denote by
\(
\widehat{\mathcal P}_{q,R}(n,r)
\)
the category of finite-dimensional polynomial
$R\widehat G_{\mathbb F_q}(n)$-modules of homogeneous degree
$r$.

Recall the surjective algebra homomorphism
\(
\widehat\zeta^{\,q}_{r,R}:
R\widehat G_{\mathbb F_q}(n)
\longrightarrow
\widehat{\mathcal S}_R(n,r)
\)
defined in \eqref{hat zeta^q_{r,R}}.

\begin{Thm}\label{equivalence2}
Assume that $q>r$. Then the category
\(
\widehat{\mathcal P}_{q,R}(n,r)
\)
is equivalent to the category of finite-dimensional
$\widehat{\mathcal S}_R(n,r)$-modules.
\end{Thm}
\begin{proof}
Let $M$ be a finite-dimensional
$\widehat{\mathcal S}_R(n,r)$-module. By Theorem~\ref{equivalence1},
$M$ gives rise to a polynomial
$R\widehat G_R(n)$-module of homogeneous degree $r$.
Restricting the action to the subsemigroup
\(
\widehat G_{\mathbb F_q}(n)
\subseteq
\widehat G_R(n),
\)
we obtain an
$R\widehat G_{\mathbb F_q}(n)$-module.
We claim that this restricted module is polynomial.
Indeed, the coefficient functions of $M$ as an
$R\widehat G_R(n)$-module belong to
\(
\overline{\widehat A_R(n,r)}
\)
by Theorem~\ref{equivalence1}. After restriction to
$\widehat G_{\mathbb F_q}(n)$, these coefficient functions belong to
\(
\operatorname{res}_{q,R}
\left(
\overline{\widehat A_R(n,r)}
\right).
\)
Hence the restricted module is polynomial.

Conversely, let
\(
V\in\widehat{\mathcal P}_{q,R}(n,r).
\)
We prove that $\ker\widehat\zeta^q_{r,R}$ annihilates $V$.
Choose an $R$-basis
\(
\{v_i\mid i\in I\}
\)
of $V$. For
$g\in\widehat G_{\mathbb F_q}(n)$, write
\[
gv_j=\sum_{i\in I}c_{i,j}(g)v_i ,
\]
where
\(
c_{i,j}:\widehat G_{\mathbb F_q}(n)\rightarrow R
\)
is a coefficient function of $V$.
By the assumption on $V$, we have
\(
c_{i,j}\in
\operatorname{res}_{q,R}
\left(
\overline{\widehat A_R(n,r)}
\right).
\)
Hence there exists
\(
\widetilde c_{i,j}
\in
\overline{\widehat A_R(n,r)}
\)
such that
\(
c_{i,j}
=
\operatorname{res}_{q,R}
(\widetilde c_{i,j}).
\)

Using the identification
\(
\overline{\widehat A_R(n,r)}
\cong
\widehat{\mathcal S}_R(n,r)^*,
\) given in \eqref{S*'},
let
\(
\lambda_{i,j}\in
\widehat{\mathcal S}_R(n,r)^*
\)
be the linear functional corresponding to
$\widetilde c_{i,j}$.
By \eqref{ti F}, for every
\(
g\in\widehat G_R(n),
\)
we have
\(
\widetilde c_{i,j}(g)
=
\lambda_{i,j}
\bigl(
\widehat\zeta_{r,R}(g)
\bigr).
\)
Restricting to
\(
g\in\widehat G_{\mathbb F_q}(n),
\)
we obtain
\(
c_{i,j}(g)
=
\lambda_{i,j}
\bigl(
\widehat\zeta^q_{r,R}(g)
\bigr).
\)
Both sides extend $R$-linearly to
$R\widehat G_{\mathbb F_q}(n)$. Therefore, for every
\(
u\in R\widehat G_{\mathbb F_q}(n),
\)
we have
\(
c_{i,j}(u)
=
\lambda_{i,j}
\bigl(
\widehat\zeta^q_{r,R}(u)
\bigr).
\)
Now let
\(
u\in \ker\widehat\zeta^q_{r,R}.
\)
Since
\(
\widehat\zeta^q_{r,R}(u)=0,
\)
we get
\(
c_{i,j}(u)=0
\)
for all $i,j$. Consequently, $uV=0$.
Hence
\(
\ker\widehat\zeta^q_{r,R}V=0.
\)
Therefore the action of
$R\widehat G_{\mathbb F_q}(n)$ on $V$ factors through
\(
R\widehat G_{\mathbb F_q}(n)/\ker\widehat\zeta^q_{r,R}
\simeq
\widehat{\mathcal S}_R(n,r).
\)
\end{proof}

Combining Theorems \ref{classification} and \ref{equivalence2} yields the
following consequence.

\begin{Thm}\label{classification of polynomial representations}
Let $q=p^f>r$, and let $R$ be an algebraically closed field
containing $\mathbb F_q$. Then
\[
\{L_R(\mathbf Q)\mid
\mathbf Q\in\mathcal Q_R(n,r)\}
\]
forms a complete set of pairwise non-isomorphic finite-dimensional
irreducible polynomial
$R\widehat G_{\mathbb F_q}(n)$-modules of homogeneous degree $r$.
\end{Thm}

\end{document}